\documentclass[final,notitlepage,12pt,reqno]{amsart}

\usepackage[table]{xcolor}
\usepackage{arydshln}
\makeatletter
  \renewcommand*\env@matrix[1][*\c@MaxMatrixCols c]{%
    \hskip -\arraycolsep
    \let\@ifnextchar\new@ifnextchar
  \array{#1}}
\makeatother
\usepackage{etoolbox,mathtools}
\makeatletter
\patchcmd{\@makefnmark}{\fontsize}{\check@mathfonts\fontsize}{}{}
\makeatother

\makeatletter

\DeclareFontFamily{U}{mathx}{}
\DeclareFontShape{U}{mathx}{m}{n}{<-> mathx10}{}
\DeclareSymbolFont{mathx}{U}{mathx}{m}{n}
\DeclareMathAccent{\widecheck}{0}{mathx}{"71}

\def\smallunderbrace#1{\mathop{\vtop{\m@th\ialign{##\crcr
   $\hfil\displaystyle{#1}\hfil$\crcr
   \noalign{\kern3\p@\nointerlineskip}%
   \tiny\upbracefill\crcr\noalign{\kern3\p@}}}}\limits}
\makeatother

\UseRawInputEncoding
\usepackage{graphicx,extarrows}
\usepackage{calc}
\usepackage{url}
\newlength{\depthofsumsign}
\makeatletter
\let\I\@undefined
\makeatother

\usepackage{geometry}
\usepackage{indentfirst}
\usepackage[normalem]{ulem}
\usepackage{float}
\usepackage{amsthm}

\usepackage{enumitem}[2011/09/28]
\setenumerate{align=left, leftmargin=0pt,labelsep=.5em, labelindent=0\parindent,listparindent=\parindent,itemindent=*}
\usepackage{array}
\usepackage[all]{xy}

\usepackage{exscale,relsize}
\usepackage{multirow}

\usepackage{caption}[2012/02/19]

\usepackage{longtable,lscape}

\usepackage{colonequals}

\usepackage{amssymb,bm,amsmath}
\usepackage{amsfonts}

\usepackage[OT2,T1,T2A]{fontenc}
\usepackage[utf8x]{inputenc}
\usepackage[french,german,russian,english]{babel}
\usepackage{appendix}

\DeclareMathOperator{\Span}{span}

\DeclareMathOperator{\Shf}{\text{\cyrins{\textsf{Ш}}}}
\DeclareMathOperator{\shf}{\text{\cyrins{\textsf{ш}}}}

\DeclareMathOperator{\Ls}{Ls}

\DeclareMathOperator{\Li}{Li}

\DeclareMathOperator{\D}{d}
\DeclareMathOperator{\I}{Im}
\DeclareMathOperator{\R}{Re}
\DeclareMathOperator*{\Reg}{Reg}

\def\XXint#1#2#3{{\setbox0=\hbox{$#1{#2#3}{\int}$}       \vcenter{\hbox{$#2#3$}}\kern-.5\wd0}}

\def\eor{\hfill$ \square$}

\newcolumntype{L}{>{$}l<{$}}
\newcolumntype{C}{>{$}c<{$}}
\newcolumntype{R}{>{$}r<{$}}

\theoremstyle{plain}
\newtheorem{theorem}{Theorem}[section]
\newtheorem{proposition}[theorem]{Proposition}
\newtheorem{lemma}[theorem]{Lemma}
\newtheorem{corollary}[theorem]{Corollary}

\newtheorem{conjecture}[theorem]{Conjecture}

\newenvironment{remark}[1][Remark]{\begin{trivlist}
\item[\hskip \labelsep {\bfseries #1}]}{\end{trivlist}}

\theoremstyle{definition}

\numberwithin{equation}{section}

\usepackage{Baskervaldx}
\usepackage[baskervaldx]{newtxmath}
\usepackage[cal=cm,scr=rsfs,frak=euler]{mathalfa}

\DeclareFontFamily{OT1}{mathc}{}
\DeclareFontShape{OT1}{mathc}{m}{it}{<-> mathc10}{}
\DeclareMathAlphabet{\mathcal}{OT1}{mathc}{m}{it}

\usepackage[scr=rsfs]{mathalfa}
\DeclareMathAlphabet{\mathsf}{OT1}{\sfdefault}{m}{n}
\SetMathAlphabet{\mathsf}{bold}{OT1}{\sfdefault}{m}{n}
\DeclareSymbolFontAlphabet{\mathbb}{AMSb}
\DeclareRobustCommand{\cyrins}[1]{%
  \begingroup\fontfamily{erewhon-TLF}%
  \foreignlanguage{russian}{#1}%
  \endgroup
}

\usepackage{color}

\usepackage{eurosym}

\begin{document}

\pagenumbering{roman}
\selectlanguage{english}
\title{Galois descents of certain multiple polylogarithms}\dedicatory{Dedicated to 19260817, a prime number}\author{Yajun Zhou}
\address{Program in Applied and Computational Mathematics (PACM), Princeton University, Princeton, NJ 08544} \email{yajunz@math.princeton.edu}\curraddr{\textrm{} \textsc{Academy of Advanced Interdisciplinary Studies (AAIS), Peking University, Beijing 100871, P. R. China}}\email{yajun.zhou.1982@pku.edu.cn}

\date{\today}\thanks{\indent * This work was supported in part  by the Applied Mathematics Program within the Department of Energy
(DOE) Office of Advanced Scientific Computing Research (ASCR) as part of the Collaboratory on
Mathematics for Mesoscopic Modeling of Materials (CM4).}


\begin{abstract}We prove some conjectures of K. C. Au concerning the  descents of cyclotomic levels in certain sums over multiple polylogarithms. Via answers to a question of Au in some particular cases, we also confirm Broadhurst's conjecture on honorary multiple zeta values at even weights. \end{abstract}

\maketitle
\pagenumbering{arabic}

\section{Introduction\label{sec:intro}}
Multiple polylogarithms (MPLs)  are defined by (analytic continuations of) \begin{align}
\Li_{a_1,\dots,a_n}(z_1,\dots,z_n)\colonequals {\sum\limits_{\ell_{1}>\dots>\ell_{n}>0}\prod\limits_{j=1}^n\frac{\smash[t]{z_{j}^{\ell_{j}}}}{\ell_j^{a_j}}}.\label{eq:Mpl_defn}
\end{align} Here, the right-hand side of \eqref{eq:Mpl_defn}  converges absolutely  for  $ a_1,\dots,a_n\in\mathbb Z_{>0}, \prod_{j=1}^m|z_j|<1,m\in\mathbb Z\cap[1,n]$. The series in question also converges (though not necessarily in the mode of absolute convergence) when  $ (a_1,z_1)\neq(1,1)$ and $ |z_1|=\dots=|z_n|=1$. A convergent MPL $ \Li_{a_1,\dots,a_n}(z_1,\dots,z_n)$ is said to have weight $ a_1+\dots+a_n$ and depth $n$.

Using convergent MPLs, one may introduce the $ \mathbb Q$-vector space spanned by cyclotomic multiple zeta values (CMZVs) of weight $ w\in\mathbb Z_{>0}$ and level $N\in\mathbb Z_{>0}$:\begin{align}
 \mathfrak Z_{w}(N)\colonequals\Span_{\mathbb Q}\left\{\Li_{a_1,\dots,a_n}(z_1,\dots,z_n)\middle|\begin{smallmatrix}a_1,\dots,a_n\in\mathbb Z_{>0}\\z_{1}^{N}=\dots=z_n^N=1\\\sum _{j=1}^{n}a_{j}=w\\(a_1,z_1)\neq(1,1)\end{smallmatrix} \right\}.\label{eq:Zk(N)}
\end{align}Conventionally, one sets $ \mathfrak Z_{0}(N) \colonequals \mathbb Q$. CMZVs of levels 1 and 2 are known as multiple zeta values (MZVs) and alternating multiple zeta values (AMZVs), respectively. Customarily, (A)MZVs are  denoted by \begin{align}
\zeta_{A_{1},\dots,A_n}\equiv\Li_{|A_1|,\dots,|A_n|}\bigg(\frac{A_1}{|A_1|},\dots,\frac{A_n}{|A_n|}\bigg)\in \mathfrak Z_{w}(2),
\end{align}where $ A_1,\dots,A_n\in\mathbb Z\smallsetminus\{0\}$ and $ A_1\neq1$. With symbolic computations involving provable algebraic relations\footnote{In particular, Au's software \cite{Au2022a} gives automated reductions  in $ \mathfrak Z_{w}(N)$ for $ w\in\{1,2,3,4,5,6\}$, $N=4$ and  $ w\in\{1,2,3,4,5\}$, $ N=6$.} among CMZVs of mod\-er\-ate-sized weights \cite{Au2022a}, Kam Cheong Au discovered some remarkable patterns concerning \begin{align}\Li_{a_1,\dots,a_n}(z,\mathbf 1_{n-1})
\equiv\Li_{a_1,\dots,a_n}(z,\underset{n-1}{\underbrace{1,\dots,1}})\colonequals \sum_{\ell_{1}>\dots>\ell_{n}>0}\frac{z^{\ell_{1}}}{\prod_{j=1}^n\ell_j^{a_j}}\quad(\text{where } |z|\leq1)
\end{align}and \begin{align}
g_{w,n}(z)\colonequals {}&(-1)^n\sum_{\substack{a_1,\dots,a_n\in\mathbb Z_{>0}\\a_1+\dots+a_n=w}} \Li_{a_1,\dots,a_n}(z,\mathbf 1_{n-1}),\label{eq:g_w,n}
\end{align}
as recapitulated in the conjecture below.

\begin{conjecture}[K. C. Au {\cite[Conjecture 1.5(a)(b)]{Au2022a}}]\label{conj:Au1.5}\begin{enumerate}[leftmargin=*,  label=\emph{(\alph*)},ref=(\alph*),
widest=d, align=left] \item
For each fixed pair of weight $w\in\mathbb Z_{>0}$ and depth $ n\in\mathbb Z\cap[1,w]$, one has  {\allowdisplaybreaks
\begin{align}
\R g_{w,n}\bigg( \frac{1+i}{2} \bigg)\in{}&\mathfrak Z_w(2),\label{eq:Au1.5a1}\\\R g_{w,n}( i )\in{}&\mathfrak Z_w(2),\label{eq:Au1.5a2}
\end{align}
}where each individual summand is in $ \mathfrak Z_w(4)$.\footnote{In an unlabeled formula right above  \cite[Conjecture 1.5]{Au2022a}, Au stated a borderline case of \eqref{eq:Au1.5a1}, where $n=1$.}
\item Set $ \varrho\colonequals  e^{\pi i/3}$. For each fixed pair of weight $w\in\mathbb Z_{>0}$ and depth $ n\in\mathbb Z\cap[1,w]$, one has \begin{align}
\R g_{w,n}(\varrho)\in{}&\mathfrak Z_w(1),\label{eq:Au1.5b}
\end{align}where each individual summand is in $ \mathfrak Z_w(6)$.\footnote{\label{fn:desc3}According to either \cite[Example 3.9]{Au2022a} or \cite[Proposition 3.3]{SunZhou2026MCV}, there is a sharper result $ \mu_{a_1,\dots, a_n}\in\mathfrak Z_{a_1+\dots+a_n}(3)$.  Also note that one may verify \eqref{eq:Au1.5b} for $ w\in\mathbb Z\cap[1,8]$ and $ n\in\mathbb Z\cap[1,w]$, using the algorithms given in \cite[\S\S3.1--3.2]{SunZhou2026MCV}. }
 \end{enumerate}
\end{conjecture}

Au referred to these  relations as ``Galois descent for multiple polylogarithms---a subtle and
rather unexplored phenomenon''.
 Here, 
 {the   wording  ``rather unexplored'' was particularly  judicious. Galois descents of CMZVs did not belong to an uncharted terrain before Au's work~\cite{Au2022a}: Broadhurst's numerical experiments in 2014 suggested the presence of honorary multiple zeta values  (or just ``hononary MZVs'' for short) at even weights  \cite[Conjecture 5]{Broadhurst2014MDV}\begin{align}
\R \Li_{w-1,1}(\varrho,1)\in\mathfrak Z_w(1),\quad \text{for }w\in 2\mathbb Z_{>0},\label{eq:BroadhurstConj5}
\end{align} while  Xu's empirical discovery  \cite[Conjecture 7.2]{Xu2019nonlinEuler}  \begin{align}\begin{split}&\zeta_{-2m,2}+2m\zeta_{-2m-1,1}\equiv\Li_{2m,2}(-1,1)+2m\Li_{2m+1,1}(-1,1)\\\in{}& \mathfrak Z_{2m+2}(1)\cap\Span_{\mathbb Q}\{\zeta_{a,2m+2-a}\equiv\Li_{a,2m+2-a}(1,1)|a\in\mathbb Z\cap[2,2m+1]\}\label{eq:XuConj}
\end{split}\end{align}  in 2019 has been  proved very recently by Zheng--Yang  \cite{ZhengYang2025}  (in part) and by Charlton  \cite{Charlton2026} (in full). Moreover, by working with the categories of mixed Tate motives, Glanois \cite{Glanois2016} established criteria for Galois descents in motivic versions for CMZVs of levels $N\in\{ 2,3,4,6,8\}$, which paved the way for the proofs of Zheng--Yang~\cite{ZhengYang2025} and  Charlton~\cite{Charlton2026}.}

In this note, we not only answer Au's questions \eqref{eq:Au1.5a1}--\eqref{eq:Au1.5b} in the positive, but also confirm Broadhurst's \eqref{eq:BroadhurstConj5}. Our solutions to these conjectures will be both effective and algorithmic.

We open with a brief  overview of our  toolkit in  \S\ref{sec:MPL_GPL}, before establishing ``depth-blind'' versions of \eqref{eq:Au1.5a1} and \eqref{eq:Au1.5b}, which replace their left-hand sides by (alternating) sums over $ n\in\mathbb Z\cap[1,w]$.  In  \S\ref{sec:desc4}, we prove the full versions of  both \eqref{eq:Au1.5a1}  and \eqref{eq:Au1.5a2}, along with another conjecture of Au \cite[Conjecture 1.6]{Au2022a},  upon a conversion of $ g_{w,n}(z)$ from its defining multi-index summation [see \eqref{eq:g_w,n}] to an integral representation.
Afterwards, these techniques are extended to the complete verifications of  \eqref{eq:Au1.5b} and \eqref{eq:BroadhurstConj5}
in \S\ref{sec:desc3}, together with  generalizations of a recently established result on Galois descents of deformed Ap\'ery-like series \cite[Theorem 1.3]{SunZhou2026MCV}. \section{ Depth-blind Galois descents of multiple polylogarithms \label{sec:MPL_GPL}}
\subsection{Generalized polylogarithms and their properties\label{subsec:GPLprop}}
The generalized polylogarithms (GPLs) are   defined through the recursion \cite[(2.1)]{Frellesvig2016} \begin{align}\notag\\[-12pt]
G(\alpha_{1},\dots,\alpha_n;z)\colonequals{\displaystyle\int_0^z\frac{G(\alpha_2,\dots,\alpha_n;x)\D x}{x-\alpha_1}},\quad\text{for }\displaystyle\smash[t]{\sum_{k=1}^n|}\alpha_k|\neq0\label{eq:GPL_rec}
\end{align}and the initial condition \cite[(2.2)]{Frellesvig2016}\begin{align}
{ G(\boldsymbol0_{n};z)\equiv G(\underset{n }{\underbrace{0,\dots,0 }};z)}\colonequals{\dfrac{\log^nz}{n!}},\quad G(-\!\!-;z)\colonequals1.\label{eq:GPL0}
\end{align}To better appreciate the combinatorial structure of MPLs, we convert them to GPLs  (see \cite[(2.5)]{Frellesvig2016} or \cite[(1.3)]{Panzer2015}): \begin{align}
\begin{split}\Li_{a_1,\dots,a_n}(z_1,\dots,z_n)={}&(-1)^{n}G\left(\smash[b]{\underset{a_1-1 }{\underbrace{0,\dots,0 }}},\frac{1}{z_{1}},\smash[b]{\underset{a_2-1 }{\underbrace{0,\dots,0 }}},\frac{1}{z_{1}z_2},\dots,\smash[b]{\underset{a_n-1 }{\underbrace{0,\dots,0 }}},\frac{1}{\prod_{j=1}^nz_j};1\right)\\[-5pt]\\\equiv{}&(-1)^nG\left(  \boldsymbol0_{a_{1}-1},\frac{1}{z_{1}},\boldsymbol0_{a_{2}-1},\frac{1}{z_{1}z_2},\dots,\boldsymbol0_{a_{n}-1},\frac{1}{\prod_{j=1}^nz_j};1\right),
\end{split}\label{eq:MPLtoGPL}
\end{align}where $ \prod_{j=1}^nz_j\neq0$. Note that when  $ \alpha_n\neq0$, the GPL recursion \eqref{eq:GPL_rec} effectively defines $G(\alpha_{1},\dots,\alpha_n;z)$ as a function of the homogeneous coordinates $ [\alpha_{1}:\cdots:\alpha_n:z]$, namely  \cite[(2.3)]{Frellesvig2016}
\begin{align}
G(\alpha_{1},\dots,\alpha_n;z)=G(\sigma\alpha_{1},\dots,\sigma\alpha_n;\sigma z)\label{eq:GPLscaling}
\end{align}for $\sigma\neq0 $.
In particular, we have\begin{align}
\begin{split}\begin{split}\Li_{a_1,\dots,a_n}(z,\mathbf 1_{n-1})={}&(-1)^{n}G({\underset{a_1-1 }{\underbrace{0,\dots,0 }}},1,{\underset{a_2-1 }{\underbrace{0,\dots,0 }}},1,\dots,\smash[b]{\underset{a_n-1 }{\underbrace{0,\dots,0 }}},1;z)\\\equiv{}&(-1)^{n}G(\boldsymbol0_{a_1-1},1,\boldsymbol0_{a_2-1},1,\dots,\boldsymbol0_{a_n-1},1;z),
\end{split}\label{eq:LtoGPL}
\end{split}
\end{align}according to the MPL-GPL conversion \eqref{eq:MPLtoGPL}  and the scaling property \eqref{eq:GPLscaling}.

 The products of GPLs satisfy \cite[(2.4)]{Frellesvig2016} \begin{align}\label{eq:GPL_shuffle}G(\alpha_{1},\dots,\alpha_{j};z)G(\beta_{1},\dots,\beta_{k};z)=\sum_{\bm v\in\bm \alpha\Shf\bm \beta }G(v_{1},\dots,v_{j+k}; z),\end{align}where the set $ \bm \alpha\Shf\bm \beta$ runs over all the shuffles of the vector components in  $ \bm \alpha$ and $\bm \beta$   that maintain the internal orders  within both vectors. 

The information from the last two paragraphs allows us to give an alternative characterization for the $ \mathbb Q$-vector space spanned by MZVs, which will be used later in \S\ref{subsec:desc3}.\begin{lemma}For all $ w\in\mathbb Z_{\geq0}$, we have
\begin{align}
\mathfrak Z_w(1)=\Span_{\mathbb Q}\left\{Z_{\ell}G(0,\beta_2,\dots,\beta_{w-\ell};1)\middle|\begin{smallmatrix}\ell\in\mathbb Z_{\geq0},Z_\ell\in\mathfrak Z_\ell(1)\\\beta_2,\dots,\beta_{w-\ell}\in\{0,1\}\\\end{smallmatrix}\right\}.\label{eq:altMZV}
\end{align}\end{lemma}\begin{proof}Using \eqref{eq:LtoGPL} and the GPL shuffles in \eqref{eq:GPL_shuffle}, one can show that for any pair of multiple polylogarithms $ \Li_{a_1,\dots,a_m}(\bm 1_{m})\in\mathfrak Z_{a_1+\dots+a_m}(1)$ (where $ a_1,\dots,a_m\in\mathbb Z_{>0}$ and $ a_1\neq1$) and  $ \Li_{b_1,\dots,b_n}(\bm 1_{n})\in\mathfrak Z_{b_1+\dots+b_n}(1)$ (where $ b_1,\dots,b_{n}\in\mathbb Z_{>0}$ and $ b_1\neq1$), the product  $ \Li_{a_1,\dots,a_m}(\bm 1_{m})\Li_{b_1,\dots,b_n}(\bm 1_{n})$ is a $ \mathbb Z$-linear combination of $ \Li_{c_1,\dots,c_r}(\bm 1_r)\in\mathfrak Z_{a_1+\dots+a_m+b_1+\dots+b_n}(1)$ (where $ c_1,\dots,c_r\in\mathbb Z_{>0}$ and $ c_1\neq1$). (This is a special case of Goncharov's filtration property \cite[\S1.2]{Goncharov1998}.) 

 Therefore, to prove \eqref{eq:altMZV}, we only need to show that\begin{align}
G(0,\beta_2,\dots,\beta_{m};1)\in\mathfrak Z_{m}(1),
\end{align}whenever $ \beta_2,\dots,\beta _m\in\{0,1\}$. 
 If $ \beta_{m}=1$, then the statement above follows directly from \eqref{eq:Zk(N)} and \eqref{eq:LtoGPL}. If $ \beta_k=1$ and $ \beta_j=0$ for $ j\in\mathbb Z\cap[k+1,m]$, then we may expand\begin{align}
0\xlongequal{\text{\eqref{eq:GPL0}}}G(0,\beta_2,\dots,\beta_k;1)G(\bm 0_{m-k};1)
\end{align} using the GPL shuffles in \eqref{eq:GPL_shuffle}, so that $ G(0,\beta_2,\dots,\beta_k,\bm 0_{m-k};1)$ and its partners with shuffled parameters and strictly fewer trailing zeros (namely, in the form of $ G(0,\beta'_{2},\dots,\beta'_{k'-1},1,\bm 0_{m-k'};1)$ where $ \beta'_{2},\dots,\beta'_{k'-1}\in\{0,1\}$, $ \beta'_{2}+\dots+\beta'_{k'-1}=\beta_2+\dots+\beta_{k-1}$, and $ m\geq k'>k$) add up to zero. With finitely many applications of such shuffles, we may   
reexpress $ G(0,\beta_2,\dots,\beta_k,\bm 0_{m-k};1)$    as a  $ \mathbb Z$-linear combination of GPLs without any trailing zeros in their parameters  (namely, those GPLs that are equivalent to $ \Li_{a_1,\dots,a_n}(\bm 1_n)\in\mathfrak Z_{a_1+\dots+a_n}(1)$ where $ a_1,\dots,a_n\in\mathbb Z_{>0}$, $ a_1\neq1$, $ n=\beta_2+\dots+\beta_k$, and $ a_1+\dots+a_n=m$).
\end{proof}

\subsection{Equations of motion for  ``depth-blind'' Galois descents\label{subsec:depth-blind}}
In view of \eqref{eq:LtoGPL}, we can rewrite \eqref{eq:g_w,n} as\begin{align}
g_{w,n}(z)\colonequals {}&\sum_{\substack{\alpha_1,\dots,\alpha_{w-1}\in\{0,1\}\\\alpha_1+\dots+\alpha_{w-1}+1=n}}G(\alpha_1,\dots,\alpha_{w-1},1;z) .\tag{\ref{eq:g_w,n}$'$}
\end{align}
Now, we settle ``depth-blind'' versions of \eqref{eq:Au1.5a1} and \eqref{eq:Au1.5b}
for all the positive weights, as described in the two propositions below.\begin{proposition}

\label{prop:G_avg_desc}Define\begin{align}
\mathsf G_w(z)\colonequals {}&\sum_{n=1}^w g_{w,n}(z)
\end{align}for all $ w\in\mathbb Z_{>0}$. We have \begin{align}
\mathsf G_w(z)+\mathsf G_w(1-z)\in\Span_\mathbb Q\left\{ Z_\ell [\log z+\log(1-z)]^{w-\ell}\middle|\begin{smallmatrix}\ell,w-\ell\in\mathbb Z_{\geq0}\\Z_\ell\in\mathfrak Z_\ell(1)\end{smallmatrix}\right\}.\label{eq:Gw_refl}
\end{align}In particular, this implies that \begin{align}
\R \mathsf G_w\bigg(\frac{1+i}{2}\bigg)\in{}&\mathfrak Z_w(2),\\\R \mathsf G_{w}(\varrho)\in{}&\mathfrak Z_w(1).
\end{align}
\end{proposition}
\begin{proof}Thanks to the GPL recursion \eqref{eq:GPL_rec}, we obtain an equation of motion\footnote{Since empty sums are zero, we   have $ g_{w,n}(z)=0$ for $ n\notin\mathbb Z\cap[1,w]$. }\begin{align}
\frac{\partial g_{w,n}(z)}{\partial z}=\frac{ g_{w-1,n}(z)}{z}+\frac{g_{w-1,n-1}(z)}{z-1}.\label{eq:gwn_evolv}
\end{align}Summing over all the depths $ n\in\mathbb Z\cap[1,w]$, we get \begin{align}
\frac{\partial\mathsf G_w(z)}{\partial z}={}&\left( \frac{1}{z} +\frac{1}{z-1}\right)\mathsf G_{w-1}(z).
\end{align}Trading $z$ for $1-z$, one has\begin{align}
\frac{\partial\mathsf G_w(1-z)}{\partial z}={}&\left( \frac{1}{z} +\frac{1}{z-1}\right)\mathsf G_{w-1}(1-z).
\end{align}To deduce \eqref{eq:Gw_refl} from the recursion \begin{align}
\frac{\partial[\mathsf G_w(z)+\mathsf G_w(1-z)]}{\partial z}={}&\left( \frac{1}{z} +\frac{1}{z-1}\right)[\mathsf G_{w-1}(z)+G_{w-1}(1-z)]
\end{align}with initial condition\begin{align}
\mathsf G_1(z)+\mathsf G_1(1-z)=\log z+\log(1-z),
\end{align}one only needs to check the regularized limits \cite[\S2.3]{Panzer2015}\begin{align}
\Reg_{z\to0}[\mathsf G_w(z)+\mathsf G_w(1-z)]=\Reg_{z\to0}\mathsf G_w(1-z)\in \mathfrak Z_w(1),\label{eq:RegG(1-z)}
\end{align}which is a routine exercise.
[By working a little harder, one can actually represent $ \Reg_{z\to0}\mathsf G_w(1-z)$ as a polynomial (in rational coefficients) of $ \pi^2$ and all the odd zeta values $ \zeta_{2n+1}\equiv\Li_{2n+1}(1),n\in\mathbb Z_{>0}$. See  \S\ref{subsec:shf_gwn} for details.]

The rest  follow from the observations that $ \log\frac{1+i}{2}+\log\frac{1-i}{2}=\Li_1(-1)\in\mathfrak Z_1(2)$ and $ \log \varrho +\log(1-\varrho)=0$, as well as the GPL shuffles in \eqref{eq:GPL_shuffle}.\end{proof}\begin{proposition}\label{prop:G_alt_avg_desc}Define\begin{align}
\widetilde{\mathsf G}_w(z)\colonequals {}&\sum_{n=1}^w (-1)^{n}g_{w,n}(z)
\end{align}for all $ w\in\mathbb Z_{>0}$. For $ |z|\le1$ and $ z\neq1$, we have
\begin{align}
\widetilde{\mathsf G}_w(z)= g_{w,1}\bigg(\frac{z}{z-1}\bigg)\equiv-\Li_w\bigg(\frac{z}{z-1}\bigg),
\label{eq:gw1Landen}\end{align}which entails \begin{align}
\R \widetilde{\mathsf G}_1\bigg(\frac{1+i}{2}\bigg)=\frac{\log2}{2}\in\mathfrak Z_1(2),\quad \R \widetilde {\mathsf G}_1(\varrho)=0\in\mathfrak Z_1(1)=\{0\}
\end{align}and\begin{align}
\R \widetilde{\mathsf G}_w\bigg(\frac{1+i}{2}\bigg)\in \mathbb Q\zeta _{w}\subseteq\mathfrak Z_w(1),\\\R \widetilde{\mathsf G}_w(\varrho)\in \mathbb Q\zeta _{w}\subseteq\mathfrak Z_w(1),
\end{align}for all $ w\in\mathbb Z_{>1}$.
\end{proposition}\begin{proof}The function\begin{align}
\widetilde{\mathsf \Delta}_w(z)\colonequals g_{w,1}\bigg(\frac{z}{z-1}\bigg)-\widetilde{\mathsf G}_w(z) \end{align}is governed by the recursion [cf.\ \eqref{eq:gwn_evolv}] \begin{align}
\frac{\partial \widetilde{\mathsf \Delta}_w(z)}{\partial z}=\left( \frac{1}{z} -\frac{1}{z-1}\right)\widetilde{\mathsf \Delta}_{w-1}(z)
\end{align}and the initial condition $ \widetilde{\mathsf\Delta}_w(0)=0$ for all  $ w\in\mathbb Z_{>1}$. In other words, one has\begin{align}
\widetilde{\mathsf \Delta}_w(z)=\int_{0}^z\left( \frac{1}{x} -\frac{1}{x-1}\right)\widetilde{\mathsf \Delta}_{w-1}(x)\D x
\end{align} for all  $ w\in\mathbb Z_{>1}$. Direct computation reveals that  $ \widetilde{\mathsf\Delta}_1(z)\equiv0$, so we must have  $ \widetilde{\mathsf \Delta}_w(z)\equiv0$ for all $ z\in\mathbb C$ and   $ w\in\mathbb Z_{>1}$.

To wrap up this proposition, simply note that  \begin{align}
\R\Li_w(-i)={}&\frac{1-2^{w-1}}{2^{2w-1}}\zeta _{w}\intertext{and}\R\Li_w(\varrho)={}&\left( 1-\frac{1}{2^{w-1}} \right)\left( 1-\frac{1}{3^{w-1}} \right)\frac{\zeta _{w}}{2}
\end{align}hold for all $ w\in\mathbb Z_{>1}$.
\end{proof}

\section{Galois descents from square vertices\label{sec:desc4}}\subsection{Shuffle structure of Au's sums\label{subsec:shf_gwn}}The  ``depth-blind'' appetizers in \S\ref{subsec:depth-blind} primarily drew on the GPL recursion  \eqref{eq:GPL_rec},  without making heavy use of the GPL shuffles in  \eqref{eq:GPL_shuffle}.  In this subsection, we will produce identities of a different flavor, to investigate the ``in-depth'' structure of Au's sum rules.

We start from an application of the  GPL shuffles in  \eqref{eq:GPL_shuffle}:\begin{align}
G(1;z)G(\boldsymbol0_{w-1},1;z)=G(\boldsymbol0_{w-1},1,1;z)+g_{w,2}(z),\label{eq:gw2_prod}
\end{align}  or in terms of word shuffles,\begin{align}
1\shf \boldsymbol0_{w-1}1= \boldsymbol0_{w-1}11+\sum_{\substack{\alpha_1,\dots,\alpha_{w-1}\in\{0,1\}\\\alpha_1+\dots+\alpha_{w-1}+1=2}}\alpha_1\cdots\alpha_{w-1}1.
\tag{\ref{eq:gw2_prod}$'$}\label{eq:gw2_prod'}\end{align}   In the next lemma, we will generalize \eqref{eq:gw2_prod} and \eqref{eq:gw2_prod'} to all the available weights and depths. \begin{lemma}
For $ w\in\mathbb Z_{>0}$ and $ n\in\mathbb Z\cap[1,w]$, we have \begin{align}
g_{w,n}(z)=\sum_{m=1}^n(-1)^{m-1}G(\boldsymbol1_{n-m};z)G(\boldsymbol0_{w-n},\boldsymbol1_{m};z),\label{eq:gwn_prod_expn}
\end{align}with the understanding that $G(\boldsymbol0_{0},\boldsymbol1_{m};z)=G(\boldsymbol1_{m};z)$ and $ G(\boldsymbol1_{0};z)=1$. \end{lemma}\begin{proof}For $ n=1$,  the proposed formula follows from the definition of $ g_{w,1}(z)$.  For $n=w$, we have \begin{align}
g_{w,w}(z)=G(\boldsymbol1_{w};z)=\frac{\log^w(1-z)}{w!}\label{eq:g_w,w}
\end{align}by  iterative invocations of either the GPL recursion  \eqref{eq:GPL_rec} or the GPL shuffles   \eqref{eq:GPL_shuffle} for $ 1\shf\textbf{1}_n=(n+1)\boldsymbol1_{n+1}$, whereas the right-hand side of \eqref{eq:gwn_prod_expn} becomes \begin{align}
\sum_{m=1}^w(-1)^{m-1}G(\boldsymbol1_{ w-m};z)G(\boldsymbol1_{m};z)=\sum_{m=1}^w(-1)^{m-1}\frac{\log^w(1-z)}{(w-m)!m!}=\frac{\log^w(1-z)}{w!},
\end{align}according to the binomial theorem.

 In the light of \begin{align}
\frac{\partial G(\boldsymbol1_{n-m};z)}{\partial z}={}&\begin{cases}0 & \text{if }n-m=0, \\
\frac{ G(\boldsymbol1_{n-m-1};z)}{z-1} & \text{if }n-m>0, \\
\end{cases}\intertext{and}\frac{\partial G(\boldsymbol0_{w-n},\boldsymbol1_{m};z)}{\partial z}={}&\begin{cases}\frac{G(\boldsymbol1_{m-1};z)}{z-1} & \text{if }w-n=0, \\
\frac{ G(\boldsymbol0_{w-n-1},\boldsymbol1_{m};z)}{z} & \text{if }w-n>0, \\
\end{cases}
\end{align}we see that\begin{align}
\delta_{w,n}(z)\colonequals g_{w,n}(z)-\sum_{m=1}^n(-1)^{m-1}G(\boldsymbol1_{n-m};z)G(\boldsymbol0_{w-n},\boldsymbol1_{m};z)
\end{align}evolves in the same fashion as $  g_{w,n}(z)$:\begin{align}
\frac{\partial \delta_{w,n}(z)}{\partial z}=\frac{ \delta_{w-1,n}(z)}{z}+\frac{\delta_{w-1,n-1}(z)}{z-1},\quad\text{where }w\in\mathbb Z_{>0},n\in\mathbb Z\cap[1,w],
\end{align}so long as one adopts the convention that $ \delta_{w-1,w}(z)=0$. Since $ \delta_{w,n}(0)=0$, we may evaluate\begin{align}
 \delta_{w,n}(z)=\int_{0}^z\left[ \frac{ \delta_{w-1,n}(x)}{x}+\frac{\delta_{w-1,n-1}(x)}{x-1} \right]\D x\label{eq:delta_int_evolv}
\end{align} iteratively,  as follows:\begin{itemize}
\item
Using $ \delta_{2,1}(x)=\delta_{2,2}(x)=0$ proved in the last paragraph, deduce $ \delta_{3,2}(z)=0$.\item Fixing the depth $n=2$, prove $ \delta_{w,2}(z)=0$  by induction on the weight  $w$, through repeated invocations of \eqref{eq:delta_int_evolv}. [Here, one does not need to know the shuffle identity \eqref{eq:gw2_prod} beforehand.]\item Using $ \delta_{3,3}(x)=0$ proved in the last paragraph, deduce $ \delta_{4,3}(z)=0$, and then promote to all the weights $ w>4$ while fixing the depth  $n=3$.\item Proceed as above, advancing the depth $n$ by one step at a time, before validating $ \delta_{w,n}(z)=0$ for all $ w\geq n$.
\end{itemize} This completes the proof of  \eqref{eq:gwn_prod_expn}.
\end{proof}\begin{remark}In terms of word shuffles, the last lemma converts a multi-index summation over certain binary words  into a single sum over $ m\in\mathbb Z\cap[1,n]$:\begin{align}
\sum_{\substack{\alpha_1,\dots,\alpha_{w-1}\in\{0,1\}\\\alpha_1+\dots+\alpha_{w-1}+1=n}}\alpha_1\cdots\alpha_{w-1}1= \sum_{m=1}^n(-1)^{m-1}\boldsymbol1_{n-m}\shf\boldsymbol0_{w-n}\boldsymbol1_{m},\tag{\ref{eq:gwn_prod_expn}$'$}\label{eq:gwn_prod_expn'}
\end{align}so long as  the empty words $ \boldsymbol0_0$ and $ \boldsymbol1_0$  are suitably interpreted.
\eor\end{remark}\begin{remark}During the proof of Proposition \ref{prop:G_avg_desc}, we encountered a quantity\begin{align}
\Reg_{z\to0}\mathsf G_w(1-z)=\Reg_{z\to0} \sum _{n=1}^wg_{w,n}(1-z),
\end{align} where the operation    ``$ \Reg_{z\to0}$'' was meant to suppress all the positive integer powers of $ \log z$. With regard to the quantity in the last displayed equation, we also mentioned in passing the $ \mathbb Q$-algebra\begin{align}
\mathcal{Z}\colonequals \mathbb Q[\pi^2,\zeta_{3},\zeta_{5},\dots]\label{eq:QalgZ}
\end{align}generated from $ \pi^2$ and all the odd zeta values (namely, depth-1 MZVs). With \eqref{eq:gwn_prod_expn'}, we have\begin{align}
\begin{split}\Reg_{z\to0} g_{w,n}(1-z)={}&\Reg_{z\to0}\sum_{m=1}^n(-1)^{m-1}G(\boldsymbol1_{n-m};1-z)G(\boldsymbol0_{w-n},\boldsymbol1_{m};1-z)\\={}&(-1)^{n-1}G(\boldsymbol0_{w-n},\boldsymbol1_{n};1)
\end{split}
\end{align}for $ n\in\mathbb Z\cap[1,w-1]$ and $ \Reg_{z\to0}g_{w,w}(1-z)=\frac{1}{w!}\Reg_{z\to0}\log^w z=0$, so \begin{align}
\begin{split}\Reg_{z\to0}\mathsf G_w(1-z)={}&\sum_{n=1}^{w-1}(-1)^{n-1}G( \boldsymbol0_{w-n},\boldsymbol1_{n};1)=-\sum_{n=1}^{w-1}\Li_{w-n+1,\boldsymbol1_{n-1}}(\boldsymbol1_{n})\\\underset{\text{\cite[(5)]{BorweinStraub2015Snp}}}{\xlongequal{\text{\cite[(1.3)]{Koelbig1986}}}}{}&-\sum_{n=1}^{w-1}\frac{(-1)^{w-1}}{(w-n-1)!n!}\int_0^1\frac{\log^{w-n-1}t\log^{n}(1-t)\D t}{t}\\={}&\frac{(-1)^w}{(w-1)!}\int_0^1\frac{[\log t+\log(1-t)]^{w-1}-\log^{w-1}t}{t}\D t\in\mathfrak Z_{w}(1)\cap\mathcal{Z}
\end{split}\label{eq:RegGw(1-z)}
\end{align} gives a quantitative refinement of the statement in \eqref{eq:RegG(1-z)}, where  the ``$ \in$'' step follows from a standard argument \cite[\S9.1]{Koelbig1986}.
\eor\end{remark}

To facilitate further discussions, we rewrite the shuffle identities \eqref{eq:gwn_prod_expn} and \eqref{eq:gwn_prod_expn'}
with the  integral representations of Nielsen polylogarithms, as studied by K\"olbig \cite[(1.3)]{Koelbig1986} and Borwein--Straub \cite[(5)]{BorweinStraub2015Snp}. \begin{proposition}
For $ |z|\leq1$, $ z\neq1$, $ w\in\mathbb Z_{>0}$ and $ n\in\mathbb Z\cap[1,w-1]$, we have\begin{align}
g_{w,n}(z)=\frac{(-1)^w}{(w-n-1)!n!}\int_0^1 \frac{[\log^{}(1-zt)-\log(1-z)]^{n}-[-\log(1-z)]^n}{t}\log^{w-n-1}t\D t.\label{eq:gwn_int}
\end{align} \end{proposition}\begin{proof}We compute \begin{align}
\begin{split}g_{w,n}(z)={}&-\sum_{m=1}^n\frac{\log^{n-m}(1-z)}{(n-m)!}\Li_{w-n+1,\boldsymbol1_{m-1}}(z,\boldsymbol1_{m-1})\\\underset{\text{\cite[(5)]{BorweinStraub2015Snp}}}{\xlongequal{\text{\cite[(1.3)]{Koelbig1986}}}}{}&\frac{1}{(w-n-1)!}\sum_{m=1}^n\frac{(-1)^{w-n+m}\log^{n-m}(1-z)}{(n-m)!m!}\int_0^1\frac{\log^{w-n-1}t\log^{m}(1-zt)\D t}{t}.
\end{split}
\end{align}  The binomial theorem then leads us to  \eqref{eq:gwn_int}.
\end{proof}
\begin{remark}The integral representations for Nielsen polylogarithms do not work when $n=w$. Fortunately,  direct computations reveal these exceptional cases as simply $ g_{w,w}(z)=\frac{\log^w(1-z)}{w!}$ for all $ w\in\mathbb Z_{>0}$.
\eor\end{remark}

Consider the hypergeometric  series\begin{align}
_2F_1\left( \left.\begin{array}{@{}c@{}}
a_{1},a_{2} \\
b_{1} \\
\end{array}\right|z \right)=1+\sum_{k=1}^\infty\frac{(a_{1})_k(a_{2})_k}{(b_{1})_k}\frac{z^k}{k!},\quad|z|\leq 1,z\neq1, \label{eq:2F1_defn}
\end{align}
where $ (A)_{k}=\prod_{m=0}^{k-1}(A+m)$ is the rising factorial. In the next proposition, we turn the right-hand side of    \eqref{eq:gwn_int} into partial derivatives of  $ _2F_1$   series with respect to their  hypergeometric parameters.\begin{proposition}
For   $ |z|\leq1$, $ \R z\leq\frac12$, $ m\in\mathbb Z_{\geq0}$, and $n\in\mathbb Z_{>0}$, we have\begin{subequations}\begin{align}
g_{m+n,n}(z)={}&\frac1{m!n!}\left.\!\frac{\partial^{m+n}}{\partial u^{m}\partial v^n}\left[ (1-z)^{v}{_2}F_1\left(\begin{array}{@{}c@{}}
v,-u \\
1-u \\
\end{array}\middle| z\right) -1\right]\right|_{u=v=0}\label{eq:gwn_dd_2F1}\\={}&\frac1{m!n!}\left.\!\frac{\partial^{m+n}}{\partial u^{m}\partial v^n}\left[ {_2}F_1\left(\begin{array}{@{}c@{}}
1,v \\
1-u \\
\end{array}\middle| \frac{z}{z-1}\right) -1\right]\right|_{u=v=0}.\label{eq:gwn_dd_2F1Pfaff}
\end{align}\end{subequations}\end{proposition}\begin{proof}Note that we have  $ |z|\leq1$, $ \left\vert \frac{z}{z-1} \right\vert\leq1$, and $ z\neq1$, so the $ _2F_1$ series are both convergent. 

With Euler's integral representation for the hypergeometric series  $ _2F_1$ \cite[Theorem 2.2.1]{AAR}, we have\begin{align}
\begin{split}\sum _{m=0}^\infty \sum _{n=1}^\infty g_{m+n,n}(z)u^mv^n={}&\sum _{n=1}^\infty \frac{\log^n(1-z)}{n!}v^n+u\int_0^1 \frac{(1-z)^{v}-\left(\frac{1-z}{1-zt}\right)^v}{t^{u+1}}\D t\\={}&[(1-z)^{v}-1]+(1-z)^{v}\left[{_2}F_1\left(\begin{array}{@{}c@{}}
v,-u \\
1-u \\
\end{array}\middle| z\right) -1 \right]\\={}&(1-z)^{v}{_2}F_1\left(\begin{array}{@{}c@{}}
v,-u \\
1-u \\
\end{array}\middle| z\right)-1
\end{split}
\end{align}for sufficiently small positive numbers $-u $ and $|v|$, so  \eqref{eq:gwn_dd_2F1} follows immediately.

Pfaff's transformation \cite[(2.2.6)]{AAR}  brings us\begin{align}
{_2}F_1\left(\begin{array}{@{}c@{}}
v,-u \\
1-u \\
\end{array}\middle| z\right)=(1-z)^{-v}{_2}F_1\left(\begin{array}{@{}c@{}}
v,1 \\
1-u \\
\end{array}\middle| \frac{z}{z-1}\right),
\end{align} which explains \eqref{eq:gwn_dd_2F1Pfaff}.
\end{proof}

\subsection{Multiple polylogarithms at $\frac{1+i}2$\label{subsec:desc4to2}}Now we embark on the proof of  Au's  sum rule in \eqref{eq:Au1.5a1}. Towards this end, we unpack\begin{align}\R
g_{m+n,n}\bigg(\frac{1+i}2\bigg)={}&\left.\!\frac1{m!n!}\R\frac{\partial^{m+n}}{\partial u^{m}\partial v^n}\left[ {_2}F_1\left(\begin{array}{@{}c@{}}
1,v \\
1-u \\
\end{array}\middle| \frac{1}{i}\right) -1\right]\right|_{u=v=0}
\end{align}into Euler's integral representation \cite[Theorem 2.2.1]{AAR}
as follows:\begin{align}
\begin{split}\R
g_{m+n,n}\bigg(\frac{1+i}2\bigg)={}&\frac1{m!n!}\left.\!\R\frac{\partial^{m+n}}{\partial u^{m}\partial v^n}\left[ \frac{\Gamma (1-u) }{\Gamma (v) \Gamma (1-u-v)} \int_0^1 \frac{t^{v-1} (1-t)^{-u-v}}{1+it} \D t-1\right]\right|_{u=v=0}\\={}&\frac1{m!n!}\left.\!\R\frac{\partial^{m+n}}{\partial u^{m}\partial v^n}\left[ \frac{1 }{\mathrm B(1-u-v,v)} \int_0^1 \frac{t^{v-1} (1-t)^{-u-v}}{1+it} \D t-1\right]\right|_{u=v=0}\\={}&\frac1{m!n!}\left.\!\R\frac{\partial^{m+n}}{\partial u^{m}\partial v^n}\left[ \frac{1 }{\mathrm B(1-u-v,v)} \int_0^1 \frac{t^{v} (1-t)^{-u-v}}{i-t} \D t\right]\right|_{u=v=0},
\end{split}
\end{align}where $ \mathrm B(\cdot,\cdot)$ 
and $ \Gamma(\cdot)$ are Euler's beta and gamma functions. Now that [see \eqref{eq:QalgZ} for the definition of the $ \mathbb Q$-algebra  $ \mathcal Z$]\begin{align}
\left.\!\frac{\partial^{m+n}}{\partial u^{m}\partial v^n}\left[ \frac{1 }{\mathrm B(1-u-v,v)} -v\right]\right|_{u=v=0}\in \mathfrak Z_{m+n-1}(1)\cap\mathcal{Z}\label{eq:Beta_deriv}
\end{align}is expressible through the Riemann zeta values $ \zeta_{w}\equiv\Li_w(1)$ for $ w\in\mathbb Z_{>1}$ (cf.\ \cite[(2.17)]{Xu2017}),  it will suffice to show that \begin{align}
\left.\!\R\frac{\partial^{m+n}}{\partial u^{m}\partial v^n}\int_0^1 \frac{t^{v} (1-t)^{-u-v}}{i-t} \D t\right|_{u=v=0}
\end{align}are representable through MZVs for $ m,n\in\mathbb Z_{>0}$.  Thus, we will focus on \begin{align}
\R\frac{1}{t-i}=\frac{t}{1+t^2},
\end{align} a rational function whose poles are two vertices of a square inscribed in the unit circle.

\begin{theorem}\label{thm:1.5a1.6}We have \begin{align}
\int_0^1\frac{t}{1+t^{2}}\log^a t \log^{b+1}(1-t)\D t\in\mathfrak Z_{a+b+2}(2)\label{eq:desc4int}
\end{align}for all $a,b\in\mathbb Z_{\geq0}$.  As a result, Au's relation $\R g_{w,n}\big(\frac{1+i}2\big)\in\mathfrak Z_w(2)$ holds for all $ w\in\mathbb Z_{>0}$ and $ n\in\mathbb Z\cap[1,w]$.
\end{theorem}\begin{proof}We note that Au proposed a generalization of  \eqref{eq:desc4int}  for all $ a,b,c\in\mathbb Z_{\geq0}$ \cite[Conjecture 1.6]{Au2022a}:\begin{align}
 \int_0^1\frac{t}{1+t^{2}}\log^a t \log^{b}(1-t)\log^{c}(1+t)\D t\in\mathfrak Z_{a+b+c+1}(2),\label{eq:AuConj1.6}
\end{align}whereupon na\"ive applications of  GPL recursion \eqref{eq:GPL_rec}  and GPL shuffles \eqref{eq:GPL_shuffle} would produce summands that are in $\mathfrak Z_{a+b+c+1}(4)$. We will  prove this general conjecture of Au, by targeting its equivalent form \begin{align}\begin{split}\mathscr I_{a,b,c}\colonequals{}&\int_{0}^{1}\frac{t}{1+t^{2}}\log ^{a}t\log^{b}\frac{1-t}{1+t}\log^{c}\frac{1-t^2}{2t}\D t\in\mathfrak Z_{a+b+c+1}(2)
\end{split}\tag{\ref{eq:AuConj1.6}$'$}\label{eq:AuConj1.6'}
\end{align} for all $ a,b,c\in\mathbb Z_{\geq0}$.\footnote{We recall that    $ \log2=-\Li_1(-1)\in\mathfrak Z_{1}(2)$. 
 } To facilitate further discussions, we will introduce the short-hands\begin{align}
A\colonequals \log t,\quad B\colonequals \log\frac{1-t}{1+t},\quad C\colonequals \log\frac{1-t^2}{2t} 
\end{align}for $ t\in(0,1)$.

For any   integrable function $ f(t)$ on the open unit interval, the identity\begin{align}
\int_0^1\frac{t}{1+t^{2}}\left[f(t)+f\bigg(\frac{1-t}{1+t}\bigg) \right]\D t=\int_{0}^1\frac{f(t)\D t}{1+t}\label{eq:f_01_automorph}
\end{align}is a  direct  consequence of variable substitution.
In particular, a reciprocity law\begin{align}
\mathscr I_{a,b,c}+(-1)^{c}
\mathscr I_{b,a,c}={}&\frac12\int_{0}^1\frac{A^{a}B^b C^c}{1+t}\D t\in\mathfrak Z_{a+b+c+1}(2)\label{eq:Jabc_desc_prep}
\end{align}builds on the algebraic relations\begin{align}
\frac{1-\frac{1-t}{1+t}}{1+\frac{1-t}{1+t}}=t,\quad \frac{1-\left(\frac{1-t}{1+t}\right)^2}{2\frac{1-t}{1+t}}=\frac{2t}{1-t^2},
\end{align} [for the ``$=$'' step in  \eqref{eq:Jabc_desc_prep}] as well as GPL recursion \eqref{eq:GPL_rec} and GPL shuffles \eqref{eq:GPL_shuffle} [for the ``$\in$'' step in  \eqref{eq:Jabc_desc_prep}]. The  borderline cases for $ a,c\in\mathbb Z_{\geq0}$ and $ b=0$ are easy:\begin{align}
\begin{split}\begin{split}\mathscr I_{a,0,c}={}&\int_{0}^{1}\frac{t}{1+t^{2}}A^{a}C^{c}\D t\\\xlongequal{u=t^2}{}&\frac{1}{2^{a+1}}\int_0^1\frac{1}{1+u}\log ^{a}u\left[\log\frac{1-u}{2}-\frac{\log u}{2} \right]^{c}\D u\in\mathfrak Z_{a+c+1}(2);
\end{split}
\end{split}\label{eq:Ja0c}
\end{align} so are the  borderline cases for $ b,c\in\mathbb Z_{\geq0}$ and $ a=0$: \begin{align}
\mathscr I_{0,b,c}\in\mathfrak Z_{b+c+1}(2),\label{eq:J0bc}
\end{align}which follow\  from the reciprocity law   \eqref{eq:Jabc_desc_prep} and the fact that $ \mathscr I_{b,0,c}\in \mathfrak Z_{b+c+1}(2)$.
For other values of $ a,b,c\in\mathbb Z_{\geq0}$, we will proceed by induction, in the next paragraph.  

We evaluate an absolutely convergent integral\begin{align}
\frac{1}{(2\pi i)^{2}}\int_{i0^+-\infty}^{i0^++\infty}\frac{z}{1+z^{2}}\log ^{a+1}z\log^{b+1}\frac{z-1}{z+1}\log^{c}\frac{1-z^2}{2z}\D z
\end{align}in two ways. First, closing the contour upwards, and picking up the residue at $z=i$, we find \begin{align}
\frac{1}{4\pi i}\left( \frac{\pi i}{2} \right)^{a+b+2}\left( -\frac{\pi i}{2} \right)^{c}.
\end{align}Second, applying a transformation $ z=1/t$ to $ \R z>1$, we get
\begin{align}
\begin{split}&\frac{1}{(2\pi i)^{2}}\int_{i0^+}^{i0^++\infty}\frac{z}{1+z^{2}}\log ^{a+1}z\log^{b+1}\frac{z-1}{z+1}\log^{c}\frac{1-z^2}{2z}\D z\\={}&\frac{1}{(2\pi i)^{2}}\int_{0}^{1}\frac{t}{1+t^{2}}A ^{a+1}\left[ \left(B+\pi i\right)^{b+1}C^{c}+(-1)^{a}B^{b+1} (C-\pi i)^c\right]\D t\\{}&-\frac{(-1)^{a}}{(2\pi i)^{2}}\int_{0}^{1}\frac{1}{t}A ^{a+1}B^{b+1}(C-\pi i)^c\D t;\end{split}\label{eq:I+}
\intertext{similarly, one has}
\begin{split}&\frac{1}{(2\pi i)^{2}}\int_{i0^+-\infty}^{i0^+}\frac{z}{1+z^{2}}\log ^{a+1}z\log^{b+1}\frac{z-1}{z+1}\log^{c}\frac{1-z^2}{2z}\D z\\={}&-\frac{(-1)^{a}}{(2\pi i)^{2}}\int_{0}^{1}\frac{t}{1+t^{2}}(A-\pi i)^{a+1}\left(-B\right)^{b+1}C^{c}\D t\\{}&-\frac{1}{(2\pi i)^{2}}\int_{0}^{1}\frac{t}{1+t^{2}}(A+\pi i)^{a+1} \left(-B+\pi i\right)^{b+1} (C-\pi i)^c\D t\\{}&+\frac{(-1)^{a}}{(2\pi i)^{2}}\int_{0}^{1}\frac{1}{t}(A-\pi i)^{a+1}\left(-B\right)^{b+1}C^{c}\D t.\label{eq:I-}
\end{split}
\end{align}Reading off the real parts of the last two displayed equations, we arrive at \begin{align}
\begin{split}&\frac{[(-1)^a+1] [(-1)^b+1]}{\pi^{2}}\int_{0}^{1}\frac{ A^{a+1} B^{b+1}C^{c}t\D t}{1+t^{2}}+(-1)^b (a+1)  (b+1) \int_0^1\frac{ A^{a} B^{b}C^{c}t\D t}{1+t^{2}}\\&{}-\int_0^1\frac{[(-1)^a+1] (-1)^b a (a+1) B^2+ [(-1)^b+1]b (b+1)A^{2}}{2}\frac{ A^{a-1} B^{b-1}C^{c}t\D t}{1+t^{2}}\\&{}-\frac{(-1)^{a}}{\pi ^{2}}\R\int_{0}^{1}\frac{A^{a+1} B^{b+1}(C-\pi i)^c-(A-\pi i)^{a+1}(-B)^{b+1}C^c}{t}\D t
\\\in{}&\mathfrak Z_{a+b+c+1}(2)+\Span_{\mathbb Q}\left\{\pi^{2\ell}\mathscr I_{a',b',c'}\middle|\begin{smallmatrix}\max\{\ell,c-c'\}\in\mathbb Z_{>0}\\\ell,a',b',c'\in\mathbb Z_{\geq0}\\2\ell+a'+b'+c'=a+b+c\end{smallmatrix}\right\}.
\end{split}
\end{align}In particular, this implies \begin{align}
\mathscr I_{a,b,c}\in \mathfrak Z_{a+b+c+1}(2)+\Span_{\mathbb Q}\left\{\pi^{2\ell}\mathscr I_{a',b',c'}\middle|\begin{smallmatrix}\max\{\ell,c-c'\}\in\mathbb Z_{>0}\\\ell,a',b',c'\in\mathbb Z_{\geq0}\\2\ell+a'+b'+c'=a+b+c\end{smallmatrix}\right\}\label{eq:a1b1prep}
\end{align} when both  $ a$ and  $b$ are odd, as well as \begin{align}
(b+1)\mathscr I_{a,b,c}-a\mathscr I_{a-1,b+1,c}\in \mathfrak Z_{a+b+c+1}(2)+\Span_{\mathbb Q}\left\{\pi^{2\ell}\mathscr I_{a',b',c'}\middle|\begin{smallmatrix}\max\{\ell,c-c'\}\in\mathbb Z_{>0}\\\ell,a',b',c'\in\mathbb Z_{\geq0}\\2\ell+a'+b'+c'=a+b+c\end{smallmatrix}\right\}\label{eq:a0b1prep}
\end{align}when $a$ is even and $b$ is odd. In view of the reciprocity law \eqref{eq:Jabc_desc_prep}, the last displayed equation is unscathed when its left-hand side is replaced by  $ a\mathscr I_{b+1,a-1,c}-(b+1)\mathscr I_{b,a,c}$. This effectively extends  \eqref{eq:a0b1prep} to all the cases where $ a+b$ is odd. After this, combining   \eqref{eq:a0b1prep}  linearly with the borderline cases \eqref{eq:Ja0c} and \eqref{eq:J0bc}, we see that  \eqref{eq:a1b1prep}   holds true whenever the non-negative integers $ a$ and $b$ add up to an odd number. It remains to treat the cases where  both  $ a$ and  $b$ are even. Reading off the imaginary part of  \eqref{eq:I+}+\eqref{eq:I-}, we get\begin{align}
  \begin{split}&\int_{0}^1[(b+1)A+(-1)^{a+b}(a+1) B]\frac{ A^{a} B^{b}C^{c}t\D t}{1+t^{2}}\\{}&-\frac{(-1)^{a}}{\pi }\I\int_{0}^{1}\frac{A^{a+1} B^{b+1}(C-\pi i)^c-(A-\pi i)^{a+1}(-B)^{b+1}C^c}{t}\D t\\\in{}&\mathfrak Z_{a+b+c+2}(2)+\Span_{\mathbb Q}\left\{\pi^{2\ell}\mathscr I_{a',b',c'}\middle|\begin{smallmatrix}\max\{\ell,c-c'\}\in\mathbb Z_{>0}\\\ell,a',b',c'\in\mathbb Z_{\geq0}\\2\ell+a'+b'+c'=a+b+c+1\end{smallmatrix}\right\},
\end{split}
\end{align} or equivalently,\begin{align}
\begin{split}&(b+1)\mathscr I_{a+1,b,c}+(-1)^{a+b}(a+1)\mathscr I_{a,b+1,c}\\\in{}& \mathfrak Z_{a+b+c+2}(2)+\Span_{\mathbb Q}\left\{\pi^{2\ell}\mathscr I_{a',b',c'}\middle|\begin{smallmatrix}\max\{\ell,c-c'\}\in\mathbb Z_{>0}\\\ell,a',b',c'\in\mathbb Z_{\geq0}\\2\ell+a'+b'+c'=a+b+c+1\end{smallmatrix}\right\}.
\end{split}
\end{align}Therefore,   \eqref{eq:a1b1prep}   holds true whenever the non-negative integers $ a$ and $b$ add up to an even number.
                  
Thus far, we have completed an inductive proof of \eqref{eq:AuConj1.6'} for all $ a,b\in\mathbb Z_{\geq0}$, which in turn, settles both   \eqref{eq:desc4int} and  \eqref{eq:AuConj1.6}. \end{proof}\subsection{Landen-type transformation}
A generalization of the Landen-type transformation formula \eqref{eq:gw1Landen} will bring us a verification of Au's \eqref{eq:Au1.5a2}.\begin{theorem}For $ w\in\mathbb Z_{>0}$ and $ n\in\mathbb Z\cap[1,w]$, we have
\begin{align}
g_{w,n}\bigg(\frac z{z-1}\bigg)={}&\sum_{j=n}^w(-1)^j\binom{j-1}{n-1}g_{w,j}(z).\label{eq:gwn_Landen}
\end{align} As a result, Au's relation $\R g_{w,n}(i)\in\mathfrak Z_w(2)$ holds for all $ w\in\mathbb Z_{>0}$ and $ n\in\mathbb Z\cap[1,w]$.
\end{theorem}\begin{proof}One may construct an equation of motion  for the left-hand side of \eqref{eq:gwn_Landen}, by applying a  variable transformation to  \eqref{eq:gwn_evolv}, before realizing that the right-hand side of \eqref{eq:gwn_Landen} evolves in the same way, with identical initial and boundary conditions. Concretely speaking,    one may derive the following differential equations from \eqref{eq:gwn_evolv}: \begin{align}
\frac{\partial}{\partial z}g_{w,n}\bigg(\frac z{z-1}\bigg)=\left(\frac{1}{z}-\frac{1}{z-1}\right)g_{w-1,n}\bigg(\frac z{z-1}\bigg)-\frac{1}{z-1}g_{w-1,n-1}\bigg(\frac z{z-1}\bigg)
\end{align} and\begin{align}
\begin{split}{}&\frac{\partial}{\partial z}\sum_{j=n}^w(-1)^j\binom{j-1}{n-1}g_{w,j}(z)=\sum_{j=n}^w(-1)^j\binom{j-1}{n-1}\left[ \frac{ g_{w-1,j}(z)}{z}+\frac{g_{w-1,j-1}(z)}{z-1} \right]\\={}&\sum_{j=n}^{w-1}(-1)^j\binom{j-1}{n-1}\frac{ g_{w-1,j}(z)}{z}+\sum_{j=n}^w(-1)^j\left[\binom{j-2}{n-2}+\binom{j-2}{n-1}\right]\frac{ g_{w-1,j-1}(z)}{z-1}\\={}&\sum_{j=n}^{w-1}(-1)^j\binom{j-1}{n-1}\left[\frac{ g_{w-1,j}(z)}{z}-\frac{ g_{w-1,j}(z)}{z-1}\right]-\sum_{j=n-1}^{w-1}(-1)^j\binom{j-1}{n-2}\frac{g_{w-1,j}(z)}{z-1},
\end{split}
\end{align}where we have exploited the facts that $ g_{w-1,w}(z)=0$ and $ \binom{n-2}{n-1}=0$. As $z\to0$, the expression\begin{align}
\widetilde{\mathsf \Delta}_{w,n}(z)\colonequals g_{w,n}\bigg(\frac z{z-1}\bigg)-{}&\sum_{j=n}^w(-1)^j\binom{j-1}{n-1}g_{w,j}(z)
\end{align}vanishes. Thus, after equipping the recursion\begin{align}
\widetilde{\mathsf \Delta}_{w,n}(z)=\int_{0}^z\left[ \left(\frac{1}{x}-\frac{1}{x-1}\right) \widetilde{\mathsf \Delta}_{w-1,n}(x)-\frac{1}{x-1}\widetilde{\mathsf \Delta}_{w-1,n-1}(x)\right]\D x
\end{align} with the boundary conditions \begin{align}\begin{split}
\widetilde{\mathsf \Delta}_{w,w}(z)={}&g_{w,w}\bigg(\frac z{z-1}\bigg)-(-1)^{w}g_{w,w}(z)\\\xlongequal{\text{\eqref{eq:g_w,w}}}{}&\frac{1}{w!}\log^w\bigg(1-\frac z{z-1}\bigg)-(-1)^{w}\frac{\log^w(1-z)}{w!}\equiv0
\end{split}
\end{align}and  $ \widetilde{\mathsf \Delta}_{w,1}(z)\equiv0$  (proved in Proposition \ref{prop:G_alt_avg_desc}), one can  establish $ \widetilde{\mathsf \Delta}_{w,n}(z)\equiv0$ inductively for all  $ w\in\mathbb Z_{>0}$ and $ n\in\mathbb Z\cap[1,w]$. 

 Specializing  \eqref{eq:gwn_Landen} to $ z=\frac{1+i}{2}$, we get\begin{align}
\R g_{w,n}(i)=\R g_{w,n}(-i)=\sum_{j=1}^w(-1)^j\binom{j-1}{n-1}\R g_{w,j}\bigg(\frac{1+i}2\bigg),
\end{align}thereby concluding our analysis of Galois descents from level 4 to level 2. 
\end{proof}
\section{ Galois descents from hexagon vertices\label{sec:desc3}}
\subsection{Some properties of generalized hypergeometric series}In preparation for the main course of this section, we recall some useful facts about the generalized hypergeometric series  [cf.\ \eqref{eq:2F1_defn}] \begin{align}
_pF_q\left( \left.\begin{array}{@{}c@{}}
a_{1},\dots,a_p \\
b_1,\dots,b_q \\
\end{array}\right|z \right)=1+\sum_{k=1}^\infty\frac{\prod_{\ell=1}^p(a_{\ell})_k}{\prod_{m=1}^q(b_{m})_k}\frac{z^k}{k!}.\label{eq:pFq_defn}\end{align}If the series on the right-hand side of the equation above does not terminate, then it converges in the following scenarios \cite[Theorems 2.1.1--2.1.2]{AAR}:\begin{itemize}
\item 
$p\leq q$ and $ z\in\mathbb C$;
\item $p=q+1$ and $ |z|<1$;\item $ p=q+1$, $ |z|=1$, $ z\neq1$, and\begin{align}
1+\R\sum_{m=1}^qb_m>\R\sum_{\ell=1}^{q+1} a_\ell;
\end{align} \item $ p=q+1$, $ |z|=1$, and \begin{align}
\R\sum_{m=1}^qb_m>\R\sum_{\ell=1}^{q+1} a_\ell.\label{eq:abs_conv}
\end{align}
\end{itemize}For generic hypergeometric parameters, we have \begin{align}\begin{split}{_{q+1}}
F_q\left( \left.\begin{array}{@{}c@{}}
a_{1},\dots,a_{q+1} \\
b_1,\dots,b_q \\
\end{array}\right|t \right)={}&\frac{\prod_{m=1}^q\Gamma(b_m)}{\prod_{\ell=1}^{q+1}\Gamma(a_\ell)}\Gamma\left(\sum_{\ell=1}^{q+1} a_\ell -\sum_{m=1}^qb_m\right)\\&\times(1-t)^{\sum_{m=1}^qb_m-\sum_{\ell=1}^{q+1} a_\ell}[1+O(1-t)]+O(1),
\end{split}\label{eq:near1}
\end{align}as the real number  $ t$ approaches $1$ from below.

At unit argument $ z=1$, we say that a convergent series    
$ {}_{q+1}F_q\left( \!\begin{smallmatrix}a_0,\!\!&\hfill a_1,\!\!&\dots,\!\!&\hfill a_q\\\!\!&b_1,\!\!&\dots,\!\!&b_{\smash{q}}\end{smallmatrix}\!\middle|1 \right)$ is nearly-poised (of the first kind) in Whipple's sense \cite[\S1]{Whipple1927} if $a_s+b_s$ remains a constant for all $ s\in\mathbb Z\cap[1,q]$; when the aforementioned  constant is equal to $ 1+a_0$,  we say that a nearly-poised series is well-poised \cite[Definition 3.3.2]{AAR}; a well-poised series   
$ {}_{q+1}F_q\left( \!\begin{smallmatrix}a_0,\!\!&\hfill a_1,\!\!&\dots,\!\!&\hfill a_q\\\!\!&1+a_{0}-a_1,\!\!&\dots,\!\!&1+a_{0}-a_{\smash{q}}\end{smallmatrix}\!\middle|1 \right)$     becomes very well-poised if $ a_1=1+\frac{a_0}2$  \cite[p.~146]{AAR}. Hereafter we typeset nearly-poised and (very) well-poised series with  right-aligned columns, so that the hypergeometric parameters  sitting in the same column add up to the same number. 

A  convergent series    
$ {}_{q+1}F_q\left( \!\begin{smallmatrix}a_1,...,\hfill a_{q+1}\\b_1,...,b_{\smash{q}}\end{smallmatrix}\!\middle|1 \right)$ is branded as  Saalsch\"utzian if $\sum_{m=1}^qb_m=1+\sum_{\ell=1}^{q+1} a_\ell$. Bailey 
\cite[(4.1)]{Bailey1932}  expressed a  very well-poised  $ _7F_6$ series in terms of two Saalsch\"utzian $ _4F_3$ series:
\begin{align}
\begin{split}&{_7F_6}\left(\left. \begin{array}{@{}r@{\,}r@{,{}}r@{,{}}r@{,{}}r@{,{}}r@{,{}}r@{}}
a_{0},{}&1+\frac{a_{0}}{2}&a_{2}&a_{3}&a_{4}&a_{5}&a_{6} \\[4pt]
&\frac{a_{0}}{2}&1+a_{0}-a_{2}&1+a_{0}-a_{3}&1+a_{0}-a_{4}&1+a_{0}-a_{5}&1+a_{0}-a_{6} \\
\end{array} \right|1\right)\\={}&\frac{\Gamma(1+a_{0}-a_{2}-a_{3}-a_{4})\prod_{j=2}^4 \Gamma (1+a_{0}-a_{j})}{\Gamma (1+a_{0})\prod_{2\leq j<k\leq4}\Gamma (1+a_{0}-a_{j}-a_{k}) }{_4}F_3\left(\begin{array}{@{}c@{}}
a_2,a_3,a_4,1+a_0-a_5-a_6 \\
a_2+a_3+a_4-a_{0},1+a_0-a_5,1+a_0-a_6 \\
\end{array} \middle|1 \right)\\&{}+\frac{\prod_{j=2}^6 \Gamma (1+a_{0}-a_{j})}{\Gamma (1+a_{0})\Gamma (1+a_{0}-a_{5}-a_{6} )\prod_{\ell=2}^{4}\Gamma (a_{\ell})}\frac{\Gamma (a_{2}+a_3+a_4-a_0-1)\Gamma\big(2+2a_0-\sum_{\ell=2}^6a_\ell\big)}{\Gamma\big(2+2a_0+a_{5}-\sum_{\ell=2}^6a_\ell\big)\Gamma\big(2+2a_0+a_{6}-\sum_{\ell=2}^6a_\ell\big)}\\&{}\times{_4}F_3\left(\begin{array}{@{}c@{}}
1+a_{0}-a_3-a_{4},1+a_{0}-a_{2}-a_4,1+a_{0}-a_2-a_{3},2+2a_0-\sum_{\ell=2}^6a_{\ell} \\
2+2a_0+a_{5}-\sum_{\ell=2}^6a_\ell,2+2a_0+a_{6}-\sum_{\ell=2}^6a_\ell,2+a_{0}-a_2-a_3-a_4 \\
\end{array} \middle|1 \right).
\end{split}\label{eq:Bailey7F6to4F3}
\end{align}A Saalsch\"utzian $ _3F_2$ series  satisfies  
\cite[Corollary 2.4.5]{AAR}\begin{align}
\begin{split}{}&{_3}F_2\left(\begin{array}{@{}c@{}}
a_1,a_2,b_1+b_2-a_1-a_2-1 \\
b_{1} ,b_2\\
\end{array}\middle|1 \right)\\={}&\frac{\Gamma(b_1)\Gamma(b_2)\Gamma(b_1-a_1-a_2)\Gamma(b_2-a_1-a_2)}{\Gamma(b_1-a_1)\Gamma(b_1-a_2)\Gamma(b_2-a_1)\Gamma(b_2-a_2)}+\frac{1}{a_1+a_2-b_1}\frac{\Gamma(b_1)\Gamma(b_2)}{\Gamma(a_{1})\Gamma(a_2)\Gamma(b_1+b_2-a_1-a_2)}\\&{}\times{_3}F_2\left(\begin{array}{@{}c@{}}
b_{1}-a_1,b_{2}-a_2,1 \\
1+b_{1}-a_{1}-a_{2},b_1+b_2-a_1-a_2\\
\end{array}\middle|1 \right),\label{eq:Saal3F2}
\end{split}
\end{align} where the $ _3F_2$ series on the right-hand side is not necessarily Saalsch\"utzian, so it still has to meet the requirement \eqref{eq:abs_conv} for convergence.     

In the next lemma, we represent a very well-poised $ _7F_6$ series by an integral over the product of two hypergeometric functions on the unit interval.\begin{lemma}
We have\begin{align}
\begin{split}&{_7F_6}\left(\left. \begin{array}{@{}r@{\,}r@{,{}}r@{,{}}r@{,{}}r@{,{}}r@{,{}}r@{}}
a_{0},{}&1+\frac{a_{0}}{2}&a_{2}&a_{3}&a_{4}&a_{5}&a_{6} \\[4pt]
&\frac{a_{0}}{2}&1+a_{0}-a_{2}&1+a_{0}-a_{3}&1+a_{0}-a_{4}&1+a_{0}-a_{5}&1+a_{0}-a_{6} \\
\end{array} \right|1\right)\\={}&\frac{\Gamma (1+a_{0}-a_{2})\Gamma (1+a_{0}-a_{3})\Gamma (1+a_{0}-a_{5})\Gamma (1+a_{0}-a_{6})}{\Gamma (1+a_{0})\Gamma (1+a_{0}-a_{4})\Gamma(a_{4})\Gamma (1+a_{0}-a_{2}-a_{3})\Gamma (1+a_{0}-a_{5}-a_{6}) }\\&{}\times\int_0^1\xi^{a_4-1}(1-\xi)^{a_0-a_4}{_2}F_1\left(\begin{array}{@{}c@{}}
a_2,a_3 \\
1+a_0-a_4 \\
\end{array} \middle|1-\xi \right){_2}F_1\left(\begin{array}{@{}c@{}}
a_5,a_6 \\
1+a_0-a_4 \\
\end{array} \middle|1-\xi \right)\D\xi,\label{eq:7F6wp2F1int}
\end{split}
\end{align}when both sides are convergent [cf.\ \eqref{eq:abs_conv} and \eqref{eq:near1}].\end{lemma}\begin{proof}By a special case of Zudilin's formula \cite[(2)]{Zudilin2002}, we may convert the very well-poised $ _7F_6$ series into a triple integral\begin{align}
\begin{split}&\frac{\prod_{j=2}^6 \Gamma (1+a_{0}-a_{j})}{\Gamma (1+a_{0}) \prod_{2\leq k\leq5}\Gamma (1+a_{0}-a_{k}-a_{k+1}) }\frac{1}{\prod_{\ell=3}^{5}\Gamma (a_{\ell})}\\&{}\times\int_0^1\int_0^1\int_0^1\frac{x^{a_3-1} (1-x)^{a_0-a_3-a_4} y^{a_4-1} (1-y)^{a_0-a_4-a_5} z^{a_5-1}(1-z)^{a_0-a_5-a_6}   }{\{1-x [1-y (1-z)]\}^{a_2}}\D x\D y\D z.
\end{split}
\end{align}Integrating over $x\in(0,1)$, with the help from Euler's integral representation for the hypergeometric series  $ _2F_1$ \cite[Theorem 2.2.1]{AAR}, one identifies the last displayed equation with\begin{align}
\begin{split}&\frac{\Gamma (1+a_{0}-a_{3}-a_{4})\prod_{j=2}^6 \Gamma (1+a_{0}-a_{j})}{\Gamma (1+a_{0})\Gamma (1+a_{0}-a_{4}) \prod_{2\leq k\leq5}\Gamma (1+a_{0}-a_{k}-a_{k+1}) }\frac{1}{\prod_{\ell=4}^{5}\Gamma (a_{\ell})}\\&{}\times\int_0^1y^{a_4-1} (1-y)^{a_0-a_4-a_5} z^{a_5-1}(1-z)^{a_0-a_5-a_6}\ {_2}F_1\left(\begin{array}{@{}c@{}}
a_2,a_3 \\
1+a_0-a_4 \\
\end{array} \middle|1-y(1-z) \right)\D y\D z.
\end{split}
\end{align}A variable substitution $ \xi= y(1-z)$ then turns the last expression into (cf.\ \cite[(12)--(15)]{Wan2012} and \cite[(3.7)]{Zhou2018LaportaSunrise})\begin{align}
\begin{split}&\frac{\Gamma (1+a_{0}-a_{3}-a_{4})\prod_{j=2}^6 \Gamma (1+a_{0}-a_{j})}{\Gamma (1+a_{0})\Gamma (1+a_{0}-a_{4}) \prod_{2\leq k\leq5}\Gamma (1+a_{0}-a_{k}-a_{k+1}) }\frac{1}{\prod_{\ell=4}^{5}\Gamma (a_{\ell})}\\&{}\times\int_0^1\left[ \int_\xi^1\frac{\left(1-\frac{\xi}{y}\right)^{a_5-1}\left( \frac{\xi}{y} \right)^{a_0-a_5-a_6}}{y^{2-a_4} (1-y)^{a_4+a_5-a_0}}\D y\right] {_2}F_1\left(\begin{array}{@{}c@{}}
a_2,a_3 \\
1+a_0-a_4 \\
\end{array} \middle|1-\xi \right)\D\xi.
\end{split}
\end{align}Finally, Euler's integral representation for the hypergeometric series  $ _2F_1$ \cite[Theorem 2.2.1]{AAR} and Pfaff's transformation \cite[(2.2.6)]{AAR} allow us to complete the integration over $ y\in(\xi,1)$, thus arriving at the integral formula claimed in \eqref{eq:7F6wp2F1int}. 
\end{proof} \subsection{Au's Galois descents from ``level 6'' to level 1\label{subsec:desc3}}
To tackle \eqref{eq:Au1.5b},\footnote{As indicated in Footnote \ref{fn:desc3}, for each given weight $ w$ and level $n$, this sum rule represents a Galois descent from  $ \mathfrak Z_w(3)$ [in the guise of $ \mathfrak Z_w(6)$, hence the presence of quotation marks in the title of this subsection] to   $ \mathfrak Z_w(1)$. }  we work with\begin{align}
\begin{split}\R
g_{m+n,n}(\varrho)={}&\left.\!\frac1{m!n!}\R\frac{\partial^{m+n}}{\partial u^{m}\partial v^n}\left[ {_2}F_1\left(\begin{array}{@{}c@{}}
1,v \\
1-u \\
\end{array}\middle| \frac{1}{\varrho}\right) -1\right]\right|_{u=v=0}\\={}&\frac1{m!n!}\left.\!\R\frac{\partial^{m+n}}{\partial u^{m}\partial v^n}\left[ \frac{1 }{\mathrm B(1-u-v,v)} \int_0^1 \frac{t^{v} (1-t)^{-u-v}}{\varrho-t{}} \D t\right]\right|_{u=v=0}.
\end{split}\label{eq:Re(gmn)(varrho)}
\end{align}According to our experience in \S\ref{subsec:desc4to2},
one may narrow down the task to the examination of a rational function 
\begin{align}
\R\frac{1}{t-\varrho}=\frac{1}{2}\left( \frac{1}{t-\varrho} +\frac{1}{t-\frac{1}{\varrho}}\right)=\frac{1}{2}\frac{2t-1}{1-t(1-t)}
\end{align}whose poles are two vertices of a hexagon inscribed in the unit circle.\begin{theorem}\label{thm:1.5b}We have \begin{align}
\int_0^1\left( \frac{1}{t-\varrho} +\frac{1}{t-\frac{1}{\varrho}} \right)\log^{a} t \log^{b}(1-t)\D t\in\mathfrak Z_{a+b+1}(1)\label{eq:desc3int}
\end{align}for all $a,b\in\mathbb Z_{>0}$. As a result, Au's relation $\R g_{w,n}(\varrho)\in\mathfrak Z_w(1)$ holds for all $ w\in\mathbb Z_{>0}$ and $ n\in\mathbb Z\cap[1,w]$.\end{theorem}
 
\begin{proof} It would suffice to work out the MZV characterization for  \begin{align}
 \mathscr J_{a,b}\colonequals\int_0^1\left( \frac{1}{t-\varrho} +\frac{1}{t-\frac{1}{\varrho}} \right)\log^{a} \frac{t}{1-t} \log^{b}(t(1-t))\D t.\label{eq:Jab_defn}
\end{align}With\footnote{When $ \varepsilon>0$, the generalized hypergeometric series in \eqref{eq:f5F4_defn} and   \eqref{eq:f7F6_form} are convergent [cf.\ \eqref{eq:abs_conv}]. One may demonstrate the equivalence between  \eqref{eq:f5F4_defn} and   \eqref{eq:f7F6_form} by termwise comparison of their series expansions.
 }

 {\begin{align}
\begin{split}\mathfrak f_{\delta,\varepsilon}\colonequals{}& -\frac{(\delta +\varepsilon ) \Gamma (1+\delta -\varepsilon) \Gamma (\delta +\varepsilon )}{\varepsilon  \Gamma (2+2 \delta )}\left[ \frac{1-2\varepsilon}{\delta+\varepsilon} {_5F_{4}}\left(\left.\begin{array}{@{}r@{\,}r@{,{}}r@{,{}}r@{,{}}r@{}}
1,&-\delta -\varepsilon &\frac{1}{2}+\frac{\delta-\varepsilon }{2}&1+\frac{\delta-\varepsilon }{2}&\frac{1}{2}-\frac{\varepsilon }{3}\\[4pt]
&\delta +\frac{3}{2}&{}1-\frac{\delta+\varepsilon }{2}&{}\frac{1}{2}-\frac{\delta+\varepsilon }{2}& 1-\frac{2 \varepsilon }{3}\\
\end{array}\right| 1\right)\right.\\&\left.{}-\frac{4 (1+\delta -\varepsilon) (2+\delta -\varepsilon )}{(2 \delta +3) (2-\delta -\varepsilon ) (1-\delta -\varepsilon )}{_5F_{4}}\left(\left.\begin{array}{@{}r@{\,}r@{,{}}r@{,{}}r@{,{}}r@{}}
2,&{}1-\delta -\varepsilon &\frac{3}{2}+\frac{\delta-\varepsilon }{2}&2+\frac{\delta-\varepsilon }{2}&\frac{3}{2}-\frac{\varepsilon }{3}\\[4pt]
&\delta +\frac{5}{2}&{}2-\frac{\delta+\varepsilon }{2}&{}\frac{3}{2}-\frac{\delta+\varepsilon }{2}& 2-\frac{2 \varepsilon }{3}\\
\end{array}\right| 1\right)\right]
\end{split}\label{eq:f5F4_defn}\\\begin{split}={}&\frac{(2\varepsilon -1) \Gamma (1+\delta -\varepsilon) \Gamma (\delta +\varepsilon )}{\varepsilon  \Gamma (2+2 \delta)}\\&{}\times{}_7
F_6\left(\left.\begin{array}{@{}r@{\,}r@{,{}}r@{,{}}r@{,{}}r@{,{}}r@{,{}}r@{}}
\frac{1}{2}-\varepsilon,&\frac{5}{4}-\frac{\varepsilon }{2}&1&-\delta -\varepsilon &\frac{1}{2}+\frac{\delta-\varepsilon }{2}&1+\frac{\delta-\varepsilon }{2}&\frac{1}{2}-\frac{\varepsilon }{3}\\[4pt]
&\frac{1}{4}-\frac{\varepsilon }{2}&\frac{1}{2}-\varepsilon&\delta +\frac{3}{2}&{}1-\frac{\delta+\varepsilon }{2}&{}\frac{1}{2}-\frac{\delta+\varepsilon }{2}& 1-\frac{2 \varepsilon }{3}\\
\end{array}\right| 1\right)
\end{split}\tag{\ref{eq:f5F4_defn}$'$}\label{eq:f7F6_form}\end{align}}and\footnote{The series in \eqref{eq:theta4F3_defn} converges when $ \delta+\frac{\varepsilon}{3}<0$  [cf.\ \eqref{eq:abs_conv}]. } \begin{align}
\vartheta_{\delta,\varepsilon}\colonequals -\frac{2 \Gamma (\delta -\varepsilon) \Gamma (1+\delta +\varepsilon)}{\varepsilon\Gamma (1+2 \delta)} \, _4F_3\left(\begin{array}{@{}c@{}}
1,-\delta ,-\delta -\frac{\varepsilon }{3},1+\delta +\varepsilon \\
-\delta -\varepsilon ,1+\frac{\varepsilon }{3},1-\delta +\varepsilon\  \\
\end{array}\middle|1\right),\label{eq:theta4F3_defn}
\end{align}we claim that {\allowdisplaybreaks
\begin{subequations}\begin{align}
 \mathscr J_{a,b}={}&a\sum_{n=1}^{\infty}\left.\!\frac{\partial^{a+b-1}}{\partial \varepsilon^{a-1}\partial x^{b}}\frac{\Gamma (x-\varepsilon) \Gamma (x+\varepsilon )}{x \Gamma (2 x)}\right|_{x=n,\varepsilon=0}\label{eq:Jab_MZVa}\\={}&a\lim_{\varepsilon\to0^+}\left.\!\frac{\partial^{a+b-1}}{\partial \varepsilon^{a-1}\partial\delta^{b}}\left\{\mathfrak f_{\delta,\varepsilon}-\frac{\pi}{\varepsilon\sin\frac{2\varepsilon\pi }{3}}\left[ 1+\frac{\sin \big(\big(\delta+\frac{\varepsilon }{3}\big)\pi\big)}{\sin\left((\delta+\varepsilon)\pi\right)} \right]\right\}\right|_{\delta=0}\label{eq:Jab_MZVb}\\={}&a\lim_{\substack{|\delta|+|\varepsilon|\to0\\-\delta>\frac{\varepsilon}{3}>0}}\frac{\partial^{a+b-1}}{\partial \varepsilon^{a-1}\partial\delta^{b}}\left\{\vartheta_{\delta,\varepsilon}+\frac{\pi}{\varepsilon\sin\frac{\varepsilon\pi }{3}}\left[ \frac{\sin(\delta\pi)}{\sin((\delta-\varepsilon)\pi)}-\frac{1+\frac{\sin\frac{(3\delta+\varepsilon)\pi}{3} }{\sin\left((\delta+\varepsilon)\pi\right)}}{2\cos\frac{\varepsilon\pi }{3}} \right] \right\}\label{eq:Jab_MZVc}\\\in{}&\mathfrak Z_{a+b+1}(1)\label{eq:Jab_MZVd}
\end{align}\end{subequations}
}for all $a,b\in\mathbb Z_{>0}$.  

All the components of \eqref{eq:Jab_MZVa}--\eqref{eq:Jab_MZVd} require some technical arguments, which will be organized into parts (a)--(d) in the rest of this proof.\begin{enumerate}[leftmargin=*,  label=(\alph*),ref=(\alph*),
widest=d, align=left] \item
To prove  \eqref{eq:Jab_MZVa}, we exploit the relation\begin{align}
\left( \frac{1}{t-\varrho} +\frac{1}{t-\frac{1}{\varrho}} \right)\log^{b}(t(1-t))=b!\frac{2t-1}{1-t(1-t)}G(\boldsymbol0_{b};t(1-t))=b!\frac{\partial}{\partial t}G(1,\boldsymbol0_{b};t(1-t))
\end{align} and integrate \eqref{eq:Jab_defn} by parts:\begin{align}
\begin{split}\mathscr J_{a,b}={}&-b!a\int_{0}^1\frac{G(1,\boldsymbol0_{b};t(1-t))}{t(1-t)}\log^{a-1} \frac{t}{1-t}\D t\\\xlongequal{\text{\cite[\S2.2]{SunZhou2026MCV}}}{}& a\int_{0}^1\sum_{k=1}^\infty\left.\!\frac{\partial^{b}}{\partial x^{b}}\frac{[t(1-t)]^{x-1}}{x}\right|_{x=k}\log^{a-1} \frac{t}{1-t}\D t\\={}& a\sum_{n=1}^\infty\int_{0}^1\left.\!\frac{\partial^{a+b-1}}{\partial \varepsilon^{a-1}\partial x^{b}}\frac{t^{x+\varepsilon}(1-t)^{x-\varepsilon}}{x}\right|_{x=n,\varepsilon=0}\frac{\D t}{t(1-t)}\\={}& a\sum_{n=1}^{\infty}\left.\!\frac{\partial^{a+b-1}}{\partial \varepsilon^{a-1}\partial x^{b}}\frac{\mathrm B (x-\varepsilon,x+\varepsilon )}{x }\right|_{x=n,\varepsilon=0}.\label{eq:Kab_sum}
\end{split}
\end{align} When $a$ is even, all the summands in the last line vanish identically, and one has $ \mathscr J_{a,b}=0$. This is also anticipated from  the integrand of \eqref{eq:Jab_defn}, which is an odd function of $2t-1$ if $a $ is even.

\item To prove  \eqref{eq:Jab_MZVb}, we use the Wilf--Zeilberger (WZ) method \cite{WZ1990}.

With $ x!\equiv\Gamma(x+1)$ for $x\in\mathbb C\smallsetminus \mathbb Z_{\leq0}$, we introduce a family of WZ pairs\footnote{Our  construction was facilitated    by Au's algorithm for Wilf--Zeilberger seeds \cite{Au2025b}. Concretely speaking, we ran \begin{quote}\texttt{WZSeedsSearch[1/4, \{1,1,1\}, \{1,3/2,2\}, "Level"->2, "AssoParameterRange"->Range[-4,4]]}\end{quote}and screened the candidate functions in its output. After some trials and errors, we landed on a particular  candidate  \texttt{Seed3[\{\{-2,-1,-1,1\}, \{1/2,1,1,1/2\}\}]} (which is a descendant of the \texttt{Dougall5F4} class \cite[\S2]{Au2026WZ}) and tweaked its parameters, before arriving at the form of $ \mathfrak F_\varepsilon( \nu,k)$. Subsequently, Gosper's algorithm \cite{Gosper1978} allowed us to generate  $ \mathfrak G_\varepsilon( \nu,k)$  from  $ \mathfrak F_\varepsilon( \nu,k)$. Here, our WZ pairs are effective for parameters $ k\in\mathbb Z_{\geq0}$ and $ \nu\in\mathbb R$ that make the corresponding expressions finite.  }
{\allowdisplaybreaks\begin{align}
\begin{split}\mathfrak F_\varepsilon( \nu,k)\colonequals{}& \frac{ ( \nu+\varepsilon )!\pi}{2^{2  \nu+\frac{2\varepsilon }{3}+1}    \nu! \left(\frac{\varepsilon }{3}-1\right)!\varepsilon\sin \frac{\pi  \varepsilon }{3}}\\{}&\times\frac{  (4 k-2\varepsilon+1) \left(k-\frac{\varepsilon }{3}-\frac{1}{2}\right)! (k- \nu-\varepsilon -1)! (2 k+ \nu-\varepsilon )!}{  \left(k+ \nu+\frac{1}{2}\right)! \left(k-\frac{2 \varepsilon }{3}\right)! (2 k- \nu-\varepsilon )!},
\end{split}\label{eq:seedF}\\\mathfrak G_\varepsilon( \nu,k)\colonequals {}&-\frac{(3 k-2 \varepsilon ) (2 k- \nu-\varepsilon ) (2 k- \nu-\varepsilon -1)}{2 ( \nu+1) (4 k-2 \varepsilon +1) (k- \nu-\varepsilon -1)}\mathfrak F_\varepsilon( \nu,k),\label{eq:seedG}
\end{align}}satisfying\begin{align} 
\mathfrak F_{\varepsilon}( \nu+1,k)-\mathfrak F_{\varepsilon}( \nu,k)=\mathfrak G_\varepsilon(\nu,k+1)-\mathfrak G_\varepsilon( \nu,k)\label{eq:DF=DG}
\end{align}and \begin{align}
\begin{split}\sum_{n=0}^\infty\mathfrak  G_{\varepsilon}(n+\delta,0)={}&\sum_{n=0}^\infty\left[\mathfrak  G_{\varepsilon}(n+\delta,0)-\lim_{k\to\infty}\mathfrak  G_{\varepsilon}(n+\delta,k+1)\right]\\={}&\sum_{n=0}^\infty\sum_{k=0}^\infty[\mathfrak  G_{\varepsilon}(n+\delta,k)-\mathfrak  G_{\varepsilon}(n+\delta,k+1)]\\={}&\sum_{n=0}^\infty\sum_{k=0}^\infty[\mathfrak  F_{\varepsilon}(n+\delta,k)-\mathfrak  F_{\varepsilon}(n+1+\delta,k)]\\={}&\lim_{n\to\infty}\sum_{k=0}^{\infty}\left[\mathfrak  F_\varepsilon(\delta,k)-\mathfrak F_{\varepsilon}(n+1+\delta,k)\right]=\mathfrak f_{\delta,\varepsilon}-\lim_{n\to\infty}\mathfrak f_{n+1+\delta,\varepsilon}
\end{split}
\end{align} for $ \varepsilon\in(0,3)$. Notably,  we have not switched the orders of limits and summations anywhere in the last displayed equation, while the condition  $ \varepsilon\in(0,3)$ is indispensable to  the asymptotic analysis of\begin{align}
\mathfrak G_{\varepsilon}(n+\delta,k+1)=-\frac{\pi}{(2k)^{\frac{2\varepsilon }{3}} \sin \frac{\pi  \varepsilon }{3}^{}}\frac{  (n+\delta +\varepsilon)!}{ \left(\frac{\varepsilon }{3}\right)!  (n+\delta +1)!}\left[ 1+O\left( \frac{1}{k} \right) \right].
\end{align} As we have\begin{align}
\mathfrak G_\varepsilon(x-1,0)=\frac{\mathrm B (x-\varepsilon,x+\varepsilon )}{x },
\end{align} the partial derivatives of \begin{align}
 \sum_{n=0}^\infty\frac{\mathrm B (n+1+\delta-\varepsilon,n+1+\delta+\varepsilon )}{n+1+\delta }=\mathfrak f_{\delta,\varepsilon}-\lim_{n\to\infty}\mathfrak f_{n+\delta,\varepsilon}
\end{align} reduce  \eqref{eq:Jab_MZVa} to
\begin{align}
a\lim_{\varepsilon\to0^+}\left.\!\frac{\partial^{a+b-1}}{\partial \varepsilon^{a-1}\partial\delta^{b}}\left(\mathfrak f_{\delta,\varepsilon}-\lim_{n\to\infty}\mathfrak f_{n+\delta,\varepsilon}\right)\right|_{\delta=0},
\end{align}
  in which   the inner limit  acts on a sequence indexed by $ n\in\mathbb Z_{\geq0}$. 

To evaluate the inner limit, we resort to Bailey's reduction formula \eqref{eq:Bailey7F6to4F3}. Setting $ a_0=\frac{1}{2}-\varepsilon$, $ a_2=\frac{1}{2}+\frac{\delta-\varepsilon }{2}$, $ a_3=1+\frac{\delta-\varepsilon }{2}$, $ a_4=1$, $ a_5=-(\delta+\varepsilon)$, and $ a_6=\frac{1}{2}-\frac{\varepsilon }{3}$  in  \eqref{eq:Bailey7F6to4F3}, we get\begin{align}
\begin{split}\mathfrak f_{\delta,\varepsilon}={}&\frac{\Gamma (1+\delta -\varepsilon ) \Gamma (2+\delta +\varepsilon )}{\varepsilon  \Gamma (3+2 \delta )}{_4}F_3\left(\begin{array}{@{}c@{}}
\frac{1}{2}+\frac{\delta-\varepsilon }{2},1+\frac{\delta-\varepsilon }{2},1,1+\delta +\frac{\varepsilon }{3} \\[2pt]
\delta +2,\delta +\frac{3}{2},1-\frac{2 \varepsilon }{3} \\
\end{array}\middle|1\right)\\&{}+\frac{2 \pi  \Gamma \big(\frac{\varepsilon }{3}\big) \Gamma\big(\!-\!\frac{2 \varepsilon}{3}\big) }{3\sin ((  \delta +  \varepsilon)\pi ) \Gamma \big(\!-\!\delta -\frac{2 \varepsilon }{3}\big) \Gamma \big(1+\delta +\frac{\varepsilon }{3}\big)}{_4}F_3\left(\begin{array}{@{}c@{}}
-\frac{1+\delta+\varepsilon }{2},-\frac{\delta+\varepsilon }{2},-\delta,\frac{\varepsilon }{3} \\[2pt]
\frac{1}{2},-\delta -\frac{2 \varepsilon }{3},-\delta \\
\end{array}\middle|1\right),
\end{split}\end{align}  where the second Saalsch\"utzian $ _4F_3$ simplifies to a Saalsch\"utzian $ _3F_2$, after one cancels out the rising factorial $ (-\delta)_k$ from the numerator and the denominator of its  series expansion.  Since we have \cite[(1.4.3)]{AAR} \begin{align}
\lim_{m\to\infty}\frac{\Gamma(m+x)}{\Gamma(m+y)}m^{y-x}=1<\sqrt{2},\label{eq:Gamma_ratio}
\end{align}there exists a fixed positive integer $M$, such that  the following argument holds for sufficiently large and positive $ \delta$: \begin{align}
\begin{split}&\left\vert\frac{\Gamma (1+\delta -\varepsilon ) \Gamma (2+\delta +\varepsilon )}{\varepsilon  \Gamma (2 \delta +3)}{_4}F_3\left(\begin{array}{@{}c@{}}
\frac{1}{2}+\frac{\delta-\varepsilon }{2},1+\frac{\delta-\varepsilon }{2},1,1+\delta +\frac{\varepsilon }{3} \\[2pt]
\delta +2,\delta +\frac{3}{2},1-\frac{2 \varepsilon }{3} \\
\end{array}\middle|1\right)\right\vert\\={}&\left|\sum_{k=0}^\infty \frac{2 \Gamma\big(\!-\!\frac{2 \varepsilon}3\big)\Gamma (2+\delta +\varepsilon )\Gamma (1+2 k+\delta -\varepsilon) \Gamma \left(1+k+\delta +\frac{\varepsilon }{3}\right)}{3\Gamma \big(1+\delta +\frac{\varepsilon }3\big) \Gamma (3+2 k+2 \delta ) \Gamma \big(1+k-\frac{2\varepsilon }{3}\big)}\right|\\\leq{}&\sum_{k=0}^M\left| \frac{2 \Gamma\big(\!-\!\frac{2 \varepsilon}3\big)\Gamma (2+\delta +\varepsilon )\Gamma (1+2 k+\delta -\varepsilon) \Gamma \left(1+k+\delta +\frac{\varepsilon }{3}\right)}{3\Gamma \big(1+\delta +\frac{\varepsilon }3\big) \Gamma (3+2 k+2 \delta ) \Gamma \big(1+k-\frac{2\varepsilon }{3}\big)}\right|\\{}&+\frac{\sqrt{2}}{2^{2+\delta+\varepsilon}}\sum_{k=M+1}^\infty\left| \frac{2 \Gamma\big(\!-\!\frac{2 \varepsilon}3\big)\Gamma (2+\delta +\varepsilon )}{3\Gamma \big(1+\delta +\frac{\varepsilon }3\big) k^{2}}\right|.
\end{split}
\end{align} Appealing again to \eqref{eq:Gamma_ratio},  one sees that the expression above tends to zero as 
$\delta \to\infty$. Meanwhile, plugging $a_1=-\frac{1+\delta+\varepsilon }{2} $, $a_2=-\frac{\delta+\varepsilon }{2} $, $b_1=\frac{1}{2} $, and $ b_2=-\delta -\frac{2 \varepsilon }{3}$ into \eqref{eq:Saal3F2}, we obtain\begin{align}
\begin{split}&\frac{2 \pi  \Gamma \big(\frac{\varepsilon }{3}\big) \Gamma\big(\!-\!\frac{2 \varepsilon}{3}\big) }{3\sin ((  \delta +  \varepsilon)\pi ) \Gamma \big(\!-\!\delta -\frac{2 \varepsilon }{3}\big) \Gamma \big(1+\delta +\frac{\varepsilon }{3}\big)}{_3}F_2\left(\begin{array}{@{}c@{}}
-\frac{1+\delta+\varepsilon }{2},-\frac{\delta+\varepsilon }{2},\frac{\varepsilon }{3} \\[2pt]
\frac{1}{2},-\delta -\frac{2 \varepsilon }{3} \\
\end{array}\middle|1\right)\\={}&\frac{\pi\sin \big(\big(\delta +\frac{\varepsilon }{3}\big)\pi\big) }{\varepsilon\sin\left((\delta+\varepsilon)\pi\right)\sin\frac{2 \varepsilon\pi}{3} }-\frac{   \Gamma\big(\!-\!\frac{2 \varepsilon}3\big)\Gamma (1+\delta +\varepsilon )  }{2^{1+\delta+\varepsilon}\varepsilon\Gamma \big(1+\delta +\frac{\varepsilon }3\big) } {_3}F_2\left(\begin{array}{@{}c@{}}
1+\frac{\delta+\varepsilon }{2},\frac{1}{2}+\frac{\delta+\varepsilon }{2},1 \\[2pt]
2+\delta+\varepsilon,1+\frac{\varepsilon}{3}\ \\
\end{array} \middle|1 \right).
\end{split}
\end{align}Here, the $ _3F_2$ series on the right-hand side is convergent when $ \frac12+\frac{\varepsilon}{3}>0$ [cf.\ \eqref{eq:abs_conv}], and it is represented by the following Riemann--Liouville fractional integral \cite[\S13.1(95)]{ET2}:\begin{align}
\, {_3}F_2\left(\begin{array}{@{}c@{}}
1+\frac{\delta+\varepsilon }{2},\frac{1}{2}+\frac{\delta+\varepsilon }{2},1 \\[2pt]
2+\delta+\varepsilon,1+\frac{\varepsilon}{3}\ \\
\end{array} \middle|1 \right)=\frac{\varepsilon}{3}\int_0^1(1-t)^{\frac{\varepsilon}{3}-1}\, {_2}F_1\left(\begin{array}{@{}c@{}}
1+\frac{\delta+\varepsilon }{2},\frac{1}{2}+\frac{\delta+\varepsilon }{2} \\[2pt]
2+\delta+\varepsilon\ \\
\end{array} \middle|t \right)\D t.
\end{align}Writing $ t=\frac{4x}{(1+x)^{2}}$ and employing a quadratic transformation \cite[(3.1.9)]{AAR}, we may reduce the right-hand side of the equation above into\begin{align}
\frac{4\varepsilon}{3}\int_0^1(1-x)^{\frac{2\varepsilon}{3}-1}(1+x)^{\delta+\frac{\varepsilon}{3}}\, {_2}F_1\left(\begin{array}{@{}c@{}}
1+\delta+\varepsilon,0 \\[2pt]
2+\delta+\varepsilon\ \\
\end{array} \middle|x \right)\D x=\frac{4\varepsilon}{3}\int_0^1(1-x)^{\frac{2\varepsilon}{3}-1}(1+x)^{\delta+\frac{\varepsilon}{3}}\D x.
\end{align}Therefore,\begin{align}
\begin{split}&\lim_{n\to\infty}\mathfrak f_{n+\delta,\varepsilon}\\={}&\frac{\pi\sin \big(\big(\delta +\frac{\varepsilon }{3}\big)\pi\big) }{\varepsilon\sin\left((\delta+\varepsilon)\pi\right)\sin\frac{2 \varepsilon\pi}{3}
}-\frac{4}{3}\lim_{\delta\to\infty}\frac{   \Gamma\big(\!-\!\frac{2 \varepsilon}3\big)\Gamma (1+\delta +\varepsilon )  }{2^{1+\delta+\varepsilon}\Gamma \big(1+\delta +\frac{\varepsilon }3\big) }\int_0^1(1-x)^{\frac{2\varepsilon}{3}-1}(1+x)^{\delta+\frac{\varepsilon}{3}}\D x\\={}&\frac{\pi\sin \big(\big(\delta+\frac{\varepsilon }{3}\big)\pi\big) }{\varepsilon\sin\left((\delta+\varepsilon)\pi\right)\sin\frac{2 \varepsilon\pi}{3}
}-\frac{4}{3}\lim_{\delta\to\infty}\frac{   \Gamma\big(\!-\!\frac{2 \varepsilon}3\big)\Gamma (1+\delta +\varepsilon )  }{2^{1+\delta+\varepsilon}\Gamma \big(1+\delta +\frac{\varepsilon }3\big)\delta^{2\varepsilon/3}\ }\int_0^\delta u^{\frac{2\varepsilon}{3}-1}\left(2-\frac{u}{\delta}\right)^{\delta+\frac{\varepsilon}{3}}\D u\\={}&\frac{\pi\sin \big(\big(\delta+\frac{\varepsilon }{3}\big)\pi\big) }{\varepsilon\sin\left((\delta+\varepsilon)\pi\right)\sin\frac{2 \varepsilon\pi}{3}
}-\frac{\Gamma\big(\!-\!\frac{2 \varepsilon}3\big)}{3}   \int_{0}^{\infty}\left(\frac{u}{2}\right)^{\frac{2\varepsilon }{3}-1} e^{-\frac{u}{2}}\D u
\end{split}
\end{align} is a consequence of \eqref{eq:Gamma_ratio} and the dominated convergence theorem.  This explains \eqref{eq:Jab_MZVb} in full. 

\item To prove  \eqref{eq:Jab_MZVc}, we  work with a new family of WZ pairs.

Concretely speaking, with{\begin{align}
\begin{split}\varPhi_\varepsilon(\nu,k)\colonequals{}&\frac{\left(\delta +\frac{\varepsilon }{3}\right)!\left(-\nu-\delta -\frac{\varepsilon }{3}\right)!(2 \delta +\nu+\varepsilon )! }{2^{2 \delta +\frac{2 \varepsilon }{3}+1} \left(-\nu+\frac{\varepsilon }{3}\right)!  (-\nu)!\varepsilon }\\&{}\times \frac{(4 k-2 \varepsilon +1) \left(k-\frac{\varepsilon }{3}-\frac{1}{2}\right)!  (k-\delta -\varepsilon -1)! (2 k-\nu-\varepsilon )! }{ \left(k+\delta +\frac{1}{2}\right)! \left(k-\frac{2 \varepsilon }{3}\right)! (2 k+\nu-\varepsilon )!},
\end{split}\\\varTheta_\varepsilon(\nu,k)\colonequals  {}&\frac{(2 k+2 \delta +1) (3 k-2 \varepsilon) (2 k+\nu-\varepsilon )}{ (4 k-2 \varepsilon +1) (3 \nu+3\delta+\varepsilon ) (2k-\nu -\varepsilon )}\varPhi_\varepsilon(\nu,k),
\end{align}}and\footnote{The $ _7F_6$ series in \eqref{eq:phi7F6_form} converges when $ 3(\nu+\delta)+\varepsilon>0$  [cf.\ \eqref{eq:abs_conv}].}\begin{align}
\begin{split}&\phi_{\nu,\varepsilon}\colonequals \sum_{k=0}^\infty\varPhi_\varepsilon(\nu,k)\\={}&\frac{2^{2 \delta +1} \varepsilon  \Gamma \big(\delta +\frac{3}{2}\big) \Gamma (1-\nu ) \Gamma \big(1-\frac{\varepsilon }{3}\big) \Gamma \big(1+\frac{\varepsilon }{3}-\nu \big) \Gamma (1-\varepsilon +\nu )}{\sqrt{\pi } (1-2 \varepsilon ) \Gamma (-\delta -\varepsilon ) \Gamma \big(1+\delta +\frac{\varepsilon }{3}\big) \Gamma (1-\varepsilon -\nu ) \Gamma \big(1-\delta  -\frac{\varepsilon }{3}-\nu\big) \Gamma (1+2 \delta +\varepsilon +\nu )}\\&{}\times {_7}
F_6\left(\left.\begin{array}{@{}r@{\,}r@{,{}}r@{,{}}r@{,{}}r@{,{}}r@{,{}}r@{}}
\frac{1}{2}-\varepsilon,&\frac{5}{4}-\frac{\varepsilon }{2}&1&-\delta -\varepsilon &\frac{1}{2}-\frac{\nu+\varepsilon }{2}&1-\frac{\nu+\varepsilon }{2}&\frac{1}{2}-\frac{\varepsilon }{3}\\[4pt]
&\frac{1}{4}-\frac{\varepsilon }{2}&\frac{1}{2}-\varepsilon&\delta +\frac{3}{2}&{}1+\frac{\nu-\varepsilon }{2}&{}\frac{1}{2}+\frac{\nu-\varepsilon }{2}& 1-\frac{2 \varepsilon }{3}\\
\end{array}\right| 1\right),\label{eq:phi7F6_form}
\end{split}
\end{align}we have\begin{align}
\varPhi_\varepsilon(-\delta,k)=\mathfrak  F_\varepsilon(\delta,k),\quad\phi_{-\delta,\varepsilon}=\mathfrak f_{\delta,\varepsilon},
\end{align} as well as \begin{align}
\varPhi_{\varepsilon}( \nu+1,k)-\varPhi_{\varepsilon}( \nu,k)=\varTheta_\varepsilon( \nu,k+1)-\varTheta_\varepsilon( \nu,k)
\end{align}and \begin{align}
\vartheta_{\delta,\varepsilon}=\sum_{n=0}^\infty\varTheta_{\varepsilon}(n-\delta,0)=\lim_{n\to\infty}\sum_{k=0}^{\infty}\left[\varPhi_\varepsilon(-\delta,k)-\varPhi_{\varepsilon}(n+1-\delta,k)\right]=\mathfrak f_{\delta,\varepsilon}-\lim_{n\to\infty}\phi_{n+1-\delta,\varepsilon},\label{eq:phi_f_psi}
\end{align}thanks to the asymptotic behavior \begin{align}
\varTheta_\varepsilon(\nu,k)=O\left( \frac{1}{k^{2(\nu+\delta)+\frac{2\varepsilon}{3}}} \right)
\end{align}for large and positive $k$.

Now we turn our attention to the sequence $ \phi_{n-\delta,\varepsilon}$ indexed by $ n\in\mathbb Z_{>0}$. For the moment, we assume  that $ -\frac{7}{6}<-\frac{\varepsilon }{3}-1<\delta <-\frac{\varepsilon }{3}$ and the positive integer $ n$ is  sufficiently large. Under these working assumptions, one may specialize \eqref{eq:7F6wp2F1int} to $a_0=\frac12-\varepsilon $, $ a_2=\frac{1}{2}-\frac{n-\delta+\varepsilon }{2}$, $ a_3=1-\frac{n-\delta+\varepsilon }{2}$, $ a_4=1$, $ a_5=-(\delta+\varepsilon)$, and $ a_6=\frac{1}{2}-\frac{\varepsilon }{3}$, as follows:\begin{align}
\begin{split}\phi_{n-\delta,\varepsilon}={}&\frac{\sqrt{\pi } \sin (\delta\pi   )\sin\frac{(3\delta+\varepsilon)\pi}{3}    }{ 2^{n+\delta-\frac{\varepsilon }{3}}\varepsilon  \sin ((  \varepsilon -  \delta)\pi ) \sin\frac{\varepsilon\pi   }{3} }\frac{\Gamma \big(\frac{1}{2}-\frac{\varepsilon }{3}\big)\Gamma (-\delta -\varepsilon ) \Gamma \big(n-\delta -\frac{\varepsilon }{3}\big) \Gamma (1+n+\delta +\varepsilon )}{\big[\Gamma \big(\frac{1}{2}-\varepsilon \big)\big]^2\Gamma \big(n+\frac{\varepsilon }{3}\big) \Gamma (n-\delta +\varepsilon )}\\&{}\times\int_{0}^1{_2}F_1\left( \begin{array}{@{}c@{}}
\frac{1}{2}-\frac{n-\delta+\varepsilon }{2},1-\frac{n-\delta+\varepsilon }{2} \\ [2pt]
\frac{1}{2}-\varepsilon \\[2pt]
\end{array}\middle|1-\xi \right) {_2}F_1\left( \begin{array}{@{}c@{}}
-\delta-\varepsilon,\frac{1}{2}-\frac{\varepsilon }{3} \\[2pt]
\frac{1}{2}-\varepsilon \\
\end{array}\middle|1-\xi \right) \frac{\D\xi}{(1- \xi)^{\varepsilon+\frac{1}{2}}},
\end{split}\label{eq:phi_2F1int}
\end{align} where the integral on the right-hand side is convergent.
As for the  hypergeometric series appearing in the integrand, we have two observations:\begin{itemize}
\item 
With Euler's transformation 
for $ _2F_1$ \cite[(2.2.7)]{AAR}, one can check that \begin{align}
_2F_1\left( \begin{array}{@{}c@{}}
\frac{1}{2}-\frac{n-\delta+\varepsilon }{2},1-\frac{n-\delta+\varepsilon }{2} \\ [2pt]
\frac{1}{2}-\varepsilon \\[2pt]
\end{array}\middle|1-\xi \right)=\xi ^{n-\delta -1} \, _2F_1\left( \begin{array}{@{}c@{}}
\frac{n-\delta-\varepsilon}{2}-\frac{1}{2},\frac{n-\delta-\varepsilon}{2} \\[2pt]
\frac{1}{2}-\varepsilon \\
\end{array}\middle|1-\xi \right)\label{eq:pos_expn}
\end{align}defines a positive-valued function\footnote{Note that our assumption $ -\frac{7}{6}<-\frac{\varepsilon }{3}-1<\delta <-\frac{\varepsilon }{3}$ already entails $ \frac12-\varepsilon>0$, so the right-hand side of \eqref{eq:pos_expn} is $ \xi ^{n-\delta -1}$ times a power series in $ 1-\xi$ with positive-valued coefficients, when $ n>1+\delta+\varepsilon$.}  for $ \xi\in(0,1)$, when $ n>1+\delta+\varepsilon$.
\item The expression\begin{align}
{_2}F_1\left( \begin{array}{@{}c@{}}
-\delta-\varepsilon,\frac{1}{2}-\frac{\varepsilon }{3} \\[2pt]
\frac{1}{2}-\varepsilon \\
\end{array}\middle|1-\xi \right)+\frac{\pi\xi ^{\delta +\frac{\varepsilon }{3}} }{\sin\frac{(3\delta+\varepsilon)\pi}{3}}\frac{\Gamma \big(\frac{1}{2}-\varepsilon\big) 
 }{\Gamma \big(\frac{1}{2}-\frac{\varepsilon }{3}\big) \Gamma (-\delta -\varepsilon ) \Gamma \left(1+\delta +\frac{\varepsilon }{3}\right)}\label{eq:2F1-sing}
\end{align} defines a bounded function\footnote{In fact,  \eqref{eq:2F1-sing} is the restriction of a continuous function for $ \xi\in[0,1]$, hence bounded.
} for 
$ \xi\in(0,1)$ [cf.\ \eqref{eq:near1}]. \end{itemize}Therefore, when  $ n>1+\delta+\varepsilon$, the right-hand side of \eqref{eq:phi_2F1int} splits into two parts, the first of which is equal to \begin{align}
\begin{split}&\frac{\pi\sqrt{\pi } \sin (\delta\pi   )    }{ 2^{n+\delta-\frac{\varepsilon }{3}}\varepsilon  \sin ((\delta -\varepsilon)\pi ) \sin\frac{\varepsilon\pi   }{3} }\frac{ \Gamma \big(n-\delta -\frac{\varepsilon }{3}\big) \Gamma (1+n+\delta +\varepsilon )}{\Gamma \big(\frac{1}{2}-\varepsilon \big)\Gamma \big(n+\frac{\varepsilon }{3}\big) \Gamma (n-\delta +\varepsilon ) \Gamma \left(1+\delta +\frac{\varepsilon }{3}\right)}\\&{}\times\int_{0}^1{_2}F_1\left( \begin{array}{@{}c@{}}
\frac{1}{2}-\frac{n-\delta+\varepsilon }{2},1-\frac{n-\delta+\varepsilon }{2} \\ [2pt]
\frac{1}{2}-\varepsilon \\[2pt]
\end{array}\middle|1-\xi \right)  \frac{\xi ^{\delta +\frac{\varepsilon }{3}}\D\xi}{(1- \xi)^{\varepsilon+\frac{1}{2}}}\\\xlongequal{\text{\cite[\S20.2(4)]{ET2}}}{}&\frac{\pi\sqrt{\pi } \sin (\delta\pi   )    }{ 2^{n+\delta-\frac{\varepsilon }{3}}\varepsilon  \sin ((\delta -\varepsilon)\pi ) \sin\frac{\varepsilon\pi   }{3} }\frac{ \Gamma \big(n-\delta -\frac{\varepsilon }{3}\big) \Gamma (1+n+\delta +\varepsilon )}{\Gamma \big(\frac{n+\delta}{2}+\frac{1}{2}-\frac{\varepsilon}{3} \big)\Gamma \big(\frac{n+\delta}{2}+1-\frac{\varepsilon}{3}\big) \Gamma (n-\delta +\varepsilon ) }\\={}&\frac{\pi\sin (\delta\pi   )}{\varepsilon  \sin ((\delta -\varepsilon)\pi ) \sin\frac{\varepsilon\pi   }{3}}\left[ 1+O\left( \frac{1}{n} \right) \right],
\end{split}
\end{align}
the second of which is bounded by a constant multiple of \begin{align}
\begin{split}&\frac{\Gamma \big(n-\delta -\frac{\varepsilon }{3}\big) \Gamma (1+n+\delta +\varepsilon )}{2^{n}\Gamma \big(n+\frac{\varepsilon }{3}\big) \Gamma (n-\delta +\varepsilon )}\int_0^1{_2}F_1\left( \begin{array}{@{}c@{}}
\frac{1}{2}-\frac{n-\delta+\varepsilon }{2},1-\frac{n-\delta+\varepsilon }{2} \\ [2pt]
\frac{1}{2}-\varepsilon \\[2pt]
\end{array}\middle|1-\xi \right)  \frac{\D\xi}{(1- \xi)^{\varepsilon+\frac{1}{2}}}\\\xlongequal{\text{\cite[\S20.2(4)]{ET2}}}{}&\frac{\Gamma \big(n-\delta -\frac{\varepsilon }{3}\big) \Gamma (1+n+\delta +\varepsilon )}{2^{n}\Gamma \big(n+\frac{\varepsilon }{3}\big) \Gamma (n-\delta +\varepsilon )}\frac{\Gamma (n-\delta )\Gamma \big(\frac{1}{2}-\varepsilon\big) }{\Gamma \big(\frac{n-\delta-\varepsilon}{2}+\frac12\big)\Gamma \big(\frac{n-\delta-\varepsilon}{2}+1\big)}
\\={}& \frac{\Gamma \big(\frac{1}{2}-\varepsilon\big) n^{\delta +\frac{\varepsilon }{3}}}{2^{\delta+\varepsilon}\sqrt{\pi }}\left[ 1+O\left( \frac{1}{n} \right) \right].\end{split}
\end{align}Thus far, we have computed \begin{align}
\lim_{n\to\infty}\phi_{n-\delta,\varepsilon}=\frac{\pi\sin (\delta\pi   )}{\varepsilon  \sin ((\delta -\varepsilon)\pi ) \sin\frac{\varepsilon\pi   }{3}}\label{eq:lim_phi}
\end{align}for  $ -\frac{7}{6}<-\frac{\varepsilon }{3}-1<\delta <-\frac{\varepsilon }{3}$. In view of \eqref{eq:phi_f_psi}, the quantity $ \lim_{n\to\infty}\phi_{n-\delta,\varepsilon}=\mathfrak f_{\delta,\varepsilon}-\vartheta_{\delta,\varepsilon}$ remains well-defined so long as $ -\delta>\frac{\varepsilon}{3}>0$. Accordingly, one may extend the applicability of \eqref{eq:lim_phi}
by analytic continuation. This completes the verification of   \eqref{eq:Jab_MZVc}.

\item To prove  \eqref{eq:Jab_MZVd}, we reformulate $ \vartheta_{\delta,\varepsilon}$ into an  integral similar to the right-hand side of \eqref{eq:7F6wp2F1int}, before computing (equivalents of) the partial derivatives in 
\eqref{eq:Jab_MZVc} by  the GPL recursion \eqref{eq:GPL_rec}.

First, without loss of generality, we impose a temporary constraint that $ 1+\delta+\varepsilon>0$. Through iterated applications of 
the 
Riemann--Liouville fractional integral \cite[\S13.1(95)]{ET2}, one gets\begin{align}
\begin{split}\vartheta_{\delta,\varepsilon}={}&\frac{2\sin(2\delta\pi)}{\varepsilon\sin((\delta-\varepsilon)\pi)}\int_0^1{_3}F_2\left( \begin{array}{@{}c@{}}
1,-\delta,-\delta-\frac{\varepsilon}{3}\ \\
1+\frac{\varepsilon}{3},-\delta-\varepsilon\ \\
\end{array}\middle|t \right)\frac{t^{\delta+\varepsilon}\D t}{(1-t)^{1+2\delta}}\\={}&\frac{2\sin(2\delta\pi)}{3\sin((\delta-\varepsilon)\pi)}\int_0^1\left[\int_0^1(1-u)^{\frac{\varepsilon}{3}-1}\,{_2}F_1\left( \begin{array}{@{}c@{}}
-\delta,-\delta-\frac{\varepsilon}{3}\ \\
-\delta-\varepsilon\ \\
\end{array}\middle|t u \right)\D u\right]\frac{t^{\delta+\varepsilon}\D t}{(1-t)^{1+2\delta}}.
\end{split}
\end{align}As termwise differentiation of the hypergeometric series brings us \begin{align}
\frac{\partial}{\partial u}\,{_2}F_1\left( \begin{array}{@{}c@{}}
-\delta,-\delta-\frac{\varepsilon}{3}\ \\
-\delta-\varepsilon\ \\
\end{array}\middle|t u \right)=-\frac{t\delta(3\delta+\varepsilon)}{3(\delta+\varepsilon)}\,{_2}F_1\left( \begin{array}{@{}c@{}}
1-\delta,1-\delta-\frac{\varepsilon}{3}\ \\
1-\delta-\varepsilon\ \\
\end{array}\middle|t u \right),\label{eq:2F1deriv}
\end{align}we may integrate by parts, in the following manner:
\begin{align}
\begin{split}{}&\vartheta_{\delta,\varepsilon}+\frac{2\pi}{\varepsilon\sin((\delta-\varepsilon)\pi)}\frac{\Gamma(1+\delta+\varepsilon)}{\Gamma(1+2\delta)\Gamma(1-\delta+\varepsilon)}\\={}&\vartheta_{\delta,\varepsilon}-\frac{2\sin(2\delta\pi)}{3\sin((\delta-\varepsilon)\pi)}\int_0^1\left[\int_0^1(1-u)^{\frac{\varepsilon}{3}-1}\D u\right]\frac{t^{\delta+\varepsilon}\D t}{(1-t)^{1+2\delta}}\\={}&-\frac{2\delta(3\delta+\varepsilon)\sin(2\delta\pi)}{3\varepsilon(\delta+\varepsilon)\sin((\delta-\varepsilon)\pi)}\int_0^1\left[\int_0^1(1-u)^{\frac{\varepsilon}{3}}\,{_2}F_1\left( \begin{array}{@{}c@{}}
1-\delta,1-\delta-\frac{\varepsilon}{3}\ \\
1-\delta-\varepsilon\ \\
\end{array}\middle|t u \right)\D u\right]\frac{t^{1+\delta+\varepsilon}\D t}{(1-t)^{1+2\delta}}.\end{split}\label{eq:intint2F1}
\end{align}
Applying Euler's transformation  \cite[(2.2.7)]{AAR} 
to\begin{align}
{_2}F_1\left(\begin{array}{@{}c@{}}
1-\delta ,1-\delta -\frac{\varepsilon }{3}\\
1-\delta-\varepsilon\ \\
\end{array} \middle|x \right)=\frac{1}{(1-x)^{1-\delta+\frac{2\varepsilon}{3}}}{_2}F_1\left(\begin{array}{@{}c@{}}
-\varepsilon,-\frac{2\varepsilon }{3}\\
1-\delta-\varepsilon\ \\
\end{array} \middle|x \right)
\end{align} for $x=tu $, one obtains\begin{align}
\begin{split}&\vartheta_{\delta,\varepsilon}+\frac{2\pi}{\varepsilon\sin((\delta-\varepsilon)\pi)}\frac{\Gamma(1+\delta+\varepsilon)}{\Gamma(1+2\delta)\Gamma(1-\delta+\varepsilon)}\\={}&-\frac{2\delta(3\delta+\varepsilon)\sin(2\delta\pi)}{3\varepsilon(\delta+\varepsilon)\sin((\delta-\varepsilon)\pi)}\int_0^1\left[\int_0^t\frac{\left(1-\frac{x}{t}\right)^{\frac{\varepsilon}{3}}}{(1-x)^{1-\delta+\frac{2\varepsilon}{3}}}{_2}F_1\left(\begin{array}{@{}c@{}}
-\varepsilon,-\frac{2\varepsilon }{3}\\
1-\delta-\varepsilon\ \\
\end{array} \middle|x \right)\D x\right]\frac{t^{\delta+\varepsilon}\D t}{(1-t)^{1+2\delta}}.
\end{split}
\end{align}With Euler's integral representation for   $ _2F_1$ \cite[Theorem 2.2.1]{AAR} and Pfaff's transformation \cite[(2.2.6)]{AAR}, one may  integrate over $ t\in(x,1)$ before running through $x\in(0,1)$, thereby turning the last displayed formula into
\begin{align}
\begin{split}&\vartheta_{\delta,\varepsilon}+\frac{2\pi}{\varepsilon\sin((\delta-\varepsilon)\pi)}\frac{\Gamma(1+\delta+\varepsilon)}{\Gamma(1+2\delta)\Gamma(1-\delta+\varepsilon)}
\\={}&\frac{2   \delta  (3 \delta +\varepsilon ) \pi }{3 (\delta +\varepsilon )\varepsilon \sin((\delta-\varepsilon)\pi) }\frac{\Gamma \big(1+\frac{\varepsilon }{3}\big)}{\Gamma(1+2\delta)\Gamma \big(1-2 \delta +\frac{\varepsilon }{3}\big)}\\&{}\times\int_0^1 {_2}F_1\left(\begin{array}{@{}c@{}}
-2\delta,-\delta-\frac{2\varepsilon }{3}\\
1-2 \delta +\frac{\varepsilon }{3}\ \\
\end{array} \middle|1-x \right){_2}F_1\left(\begin{array}{@{}c@{}}
-\varepsilon,-\frac{2\varepsilon }{3}\\
1-\delta-\varepsilon\ \\
\end{array} \middle|x \right)\frac{\D x}{(1-x)^{1+ \delta+\frac{\varepsilon}{3}}}.\end{split}
\end{align}With a substitution $ x=1-\xi$ and yet another application of the
Riemann--Liouville fractional integral \cite[\S13.1(95)]{ET2}, we have\begin{align}
\begin{split}&\vartheta_{\delta,\varepsilon}+\frac{2\pi}{\varepsilon\sin((\delta-\varepsilon)\pi)}\frac{\Gamma(1+\delta+\varepsilon)}{\Gamma(1+2\delta)\Gamma(1-\delta+\varepsilon)}\\&{}+\frac{2   \delta   \pi }{ (\delta +\varepsilon )\varepsilon \sin((\delta-\varepsilon)\pi) }\frac{\Gamma \big(1+\frac{\varepsilon }{3}\big)}{\Gamma(1+2\delta)\Gamma \big(1-2 \delta +\frac{\varepsilon }{3}\big)}{_3}F_2\left(\begin{array}{@{}c@{}}
-\varepsilon,-\frac{2\varepsilon }{3},1\\
1-\delta-\varepsilon,1-\delta-\frac{\varepsilon}{3}\ \\
\end{array} \middle|1 \right)
\\={}&\frac{2   \delta  (3 \delta +\varepsilon ) \pi }{3 (\delta +\varepsilon )\varepsilon \sin((\delta-\varepsilon)\pi) }\frac{\Gamma \big(1+\frac{\varepsilon }{3}\big)}{\Gamma(1+2\delta)\Gamma \big(1-2 \delta +\frac{\varepsilon }{3}\big)}\\&{}\times\int_0^1 \left[{_2}F_1\left(\begin{array}{@{}c@{}}
-2\delta,-\delta-\frac{2\varepsilon }{3}\\
1-2 \delta +\frac{\varepsilon }{3}\ \\
\end{array} \middle|\xi \right)-1\right]{_2}F_1\left(\begin{array}{@{}c@{}}
-\varepsilon,-\frac{2\varepsilon }{3}\\
1-\delta-\varepsilon\ \\
\end{array} \middle|1-\xi \right)\frac{\D \xi}{\xi^{1+ \delta+\frac{\varepsilon}{3}}}.\end{split}\label{eq:3F2_int_2F1}\end{align}Here,  the last integrand is bounded for $\xi\in(0,1)$, while the constraint  $ 1+\delta+\varepsilon>0$ can be lifted. 

Next, instead of coping with the partial derivatives in  
\eqref{eq:Jab_MZVc}, we aim to prove that\footnote{Although the individual addends inside the braces of \eqref{eq:epsLaurent} have Laurent expansions with  poles of order 2  around $ \varepsilon=0$, the braced expression as a whole has a Maclaurin expansion that is pole-free, due to the equivalence between  
\eqref{eq:Jab_MZVa}  
and \eqref{eq:Jab_MZVc}.} \begin{align}
\lim_{\varepsilon\to0^+}\frac{\partial^{n}}{\partial\varepsilon^{n}}\left\{\vartheta_{r\varepsilon,\varepsilon}+\frac{\pi}{\varepsilon\sin\frac{\varepsilon\pi }{3}}\left[ \frac{\sin(r\varepsilon\pi)}{\sin((r-1)\varepsilon\pi)}-\frac{1+\frac{\sin\frac{(3r+1)\varepsilon\pi}{3} }{\sin\left((r+1)\varepsilon\pi\right)}}{2\cos\frac{\varepsilon\pi }{3}} \right] \right\}\in\mathfrak Z_{n+2}(1)\label{eq:epsLaurent}
\end{align}  for all  $ r\in\mathbb Q\cap\big(\!-\!\infty,-\frac13\big)$ and $ n\in\mathbb Z_{\geq0}$.  Since\footnote{Here, the Riemann zeta values $ \zeta_k$ for $ k\in\mathbb Z_{>1}$ arise from the special values of Euler's polygamma functions   $ \psi^{(n)}(s)\colonequals \frac{\D^{n+1}}{\D s^{n+1}}\log\Gamma(s)$, which read  $ \psi^{(n)}(1)=(-1)^{n+1}n!\zeta_{n+1}$ for $ n\in\mathbb Z_{>0}$. The Euler--Mascheroni constant $ \gamma_0\colonequals -\psi^{(0)}(1)$ does not come into play, because the $ O(\varepsilon)$ term vanishes in  \eqref{eq:logGamma_expn}.}\begin{align}
\log\frac{\Gamma(1+(\alpha+\beta)\varepsilon)}{\Gamma(1+\alpha\varepsilon)\Gamma(1+\beta\varepsilon)}=\sum_{k=2}^\infty\frac{(-1)^{k}\zeta_{k}}{k}\big[(\alpha+\beta)^{k}-\alpha^{k}-\beta^{k}\big]\varepsilon^k
\label{eq:logGamma_expn}\end{align}
 for constants $ \alpha$ and $\beta$, the coefficient of  the $ O(\varepsilon^n)$ term in the Laurent expansion of \begin{align}
\left.\frac{2\pi}{\varepsilon\sin((\delta-\varepsilon)\pi)}\frac{\Gamma(1+\delta+\varepsilon)}{\Gamma(1+2\delta)\Gamma(1-\delta+\varepsilon)}\right|_{\delta=r\varepsilon}
\end{align}must be in $ \mathfrak Z_{n+2}(1)\cap\mathcal Z$, for each non-negative integer $ n$. Similarly, one may characterize the Laurent coefficients of \begin{align}
\left.\frac{2   \delta   \pi }{ (\delta +\varepsilon )\varepsilon \sin((\delta-\varepsilon)\pi) }\frac{\Gamma \big(1+\frac{\varepsilon }{3}\big)}{\Gamma(1+2\delta)\Gamma \big(1-2 \delta +\frac{\varepsilon }{3}\big)}\right|_{\delta=r\varepsilon}
\end{align}by depth-1 MZVs. Thus, we will focus on the analysis of the $ _3F_2$ series and the integral in \eqref{eq:3F2_int_2F1}, for $ \delta=r\varepsilon$. Let  $ \mathsf H_{k}^{(m)}\colonequals \sum_{0<n<k}\frac{1}{n^m}$ be the $ k$-th harmonic number of order $m\in\mathbb Z_{>0}$.
With\footnote{Here, the Taylor coefficients come from $ (1+z)_k=\frac{\Gamma(1+z+k)}{\Gamma(1+z)}$ and  $ \psi^{(n)}(k)-\psi^{(n)}(1)=(-1)^{n}n!\mathsf  H_{k-1}^{(n+1)}$.} \begin{align}\log\frac{
(1+z)_k}{k!}=\sum_{m=1}^\infty\frac{(-1)^{m-1}\mathsf H^{(m)}_k}{m}z^m\label{eq:PochTaylor}
\end{align}and   $ (z)_k=\frac{z}{k+z}(1+z)_k$, we see that\footnote{As usual, an empty product over $j$ is regarded as unity. One may differentiate
the $ _3F_2 $ series term by term, since the resulting summands go asymptotically like $ O\big(\frac1{k^2}\log^nk\big)$ for large $k$. Such a trick does not work on the $ _4F_3$ series that defines $\vartheta_{\delta,\varepsilon} $ in \eqref{eq:theta4F3_defn}, due to the $ O\big(\frac1k\big)$ behavior of its summand.
} 
\begin{align}
\begin{split}&\lim_{\varepsilon\to0^+}\frac{\partial^{n}}{\partial\varepsilon^{n}}{_3}F_2\left(\begin{array}{@{}c@{}}
-\varepsilon,-\frac{2\varepsilon }{3},1\\
1-(r+1)\varepsilon,1-\frac{(3r+1)\varepsilon}{3}\ \\
\end{array} \middle|1 \right)\\\in{}&\Span_{\mathbb Q}\left\{ \sum_{k=1}^\infty\frac{1}{k^{1+s}}\prod_{j=1}^M\mathsf H_{k}^{(m_j)}\middle|\begin{smallmatrix}s,m_j\in\mathbb Z_{>0}\\1+s+\sum_{j=1}^M m_j=n\\\end{smallmatrix} \right\}\subseteq \mathfrak Z_{n}(1),
\end{split}\label{eq:3F2deriv}
\end{align}  where the  ``$ \subseteq$'' step issues from \cite[Theorem 3.1(a)]{Zhou2022mkMpl}. For any positive integer 
$ n$, we have\footnote{Termwise differentiation is again permissible, thanks to the  $ O\big(\frac {x^{k}}{k^2}\log^nk\big)$ behavior for large $k$. }\begin{align}
\begin{split}\lim_{\varepsilon\to0^+}\frac{\partial^{n}}{\partial\varepsilon^{n}}{_2}F_1\left(\begin{array}{@{}c@{}}
-\varepsilon,-\frac{2\varepsilon }{3}\\
1-(r+1)\varepsilon\ \\
\end{array} \middle|1-\xi \right)\in{}&\Span_{\mathbb Q}\left\{ \sum_{k=1}^\infty\frac{(1-\xi)^{k}}{k^{1+s}}\prod_{j=1}^M\mathsf H_{k}^{(m_j)}\middle|\begin{smallmatrix}s,m_j\in\mathbb Z_{>0}\\1+s+\sum_{j=1}^M m_j=n\\\end{smallmatrix} \right\}\\\subseteq{}&\Span_{\mathbb Q}\left\{\Li_{a_1,\dots,a_k}(1-\xi,\bm1_{k-1}) \middle|\begin{smallmatrix}a_1,\dots,a_k\in\mathbb Z_{>0}\\\sum_{j=1}^k a_j=n\end{smallmatrix} \right\},
\end{split}\label{eq:2F1deriv}
\end{align} where the ``$\in $'' step comes from \eqref{eq:PochTaylor} and the ``$ \subseteq$'' step comes from the proof of \cite[Theorem 3.1(a)]{Zhou2022mkMpl}. Through inductive applications of the 
GPL recursion \eqref{eq:GPL_rec}, one can show that\begin{align}
\Li_{a_1,\dots,a_k}(1-\xi,\bm1_{k-1}) \in\Span_{\mathbb Q}\left\{Z_{\ell}G(\beta_{1},\dots,\beta_m;\xi) \middle|\begin{smallmatrix}\ell \in \mathbb Z_{\geq0},Z_\ell\in\mathfrak Z_\ell(1)\\\beta_{1},\dots,\beta_m\in\{0,1\}\\\sum_{j=1}^k a_j=\ell+m\end{smallmatrix} \right\}.
\end{align}
In the meantime, for any positive integer 
$ n$, we note that \begin{align}
\lim_{\varepsilon\to0^+}\frac{\partial^{n}}{\partial\varepsilon^{n}}\left[{_2}F_1\left(\begin{array}{@{}c@{}}
-2r\varepsilon,-\frac{(3r+2)\varepsilon}{3}\\
1-\frac{(6r-1)\varepsilon}{3}\ \\
\end{array} \middle|\xi \right)-1\right]\frac{1}{\xi}
\end{align}is bounded for $ \xi\in(0, 1)$, while \begin{align}
\begin{split}&\lim_{\varepsilon\to0^+}\frac{\partial^{n}}{\partial\varepsilon^{n}}\left[{_2}F_1\left(\begin{array}{@{}c@{}}
-2r\varepsilon,-\frac{(3r+2)\varepsilon}{3}\\
1-\frac{(6r-1)\varepsilon}{3}\ \\
\end{array} \middle|\xi \right)-1\right]\in \Span_{\mathbb Q}\left\{\Li_{a_1,\dots,a_k}(\xi,\bm1_{k-1}) \middle|\begin{smallmatrix}a_1,\dots,a_k\in\mathbb Z_{>0}\\\sum_{j=1}^k a_j=n\end{smallmatrix} \right\}\\={}&\Span_{\mathbb Q}\left\{G(\beta_{1},\dots,\beta_{n-1},1;\xi) \middle|\beta_{1},\dots,\beta_{n-1}\in\{0,1\} \right\}
\end{split}
\end{align}
follows from a variation on  \eqref{eq:2F1deriv} [for the ``$ \in$'' step] and the MPL-GPL conversion in \eqref{eq:MPLtoGPL} [for the ``$ =$'' step]. By the GPL shuffles \eqref{eq:GPL_shuffle}, we further have\begin{align}
\begin{split}\begin{split}&\lim_{\varepsilon\to0^+}\frac{\partial^{n}}{\partial\varepsilon^{n}}\left\{\frac{1}{\varepsilon}\left[{_2}F_1\left(\begin{array}{@{}c@{}}
-2r\varepsilon,-\frac{(3r+2)\varepsilon}{3}\\
1-\frac{(6r-1)\varepsilon}{3}\ \\
\end{array} \middle|\xi \right)-1\right]{_2}F_1\left(\begin{array}{@{}c@{}}
-\varepsilon,-\frac{2\varepsilon }{3}\\
1-(r+1)\varepsilon\ \\
\end{array} \middle|1-\xi \right)\frac{1}{\xi^{\frac{(3r+1)\varepsilon}{3}}}\right\}\\\in{}&\Span_{\mathbb Q}\left\{Z_{\ell}G(\beta_{1},\dots,\beta_{n+1-\ell};\xi)  \middle|\begin{smallmatrix}\ell \in \mathbb Z\cap[0,n+1],Z_\ell\in\mathfrak Z_\ell(1)\\\beta_{1},\dots,\beta_{n+1-\ell}\in\{0,1\}\end{smallmatrix} \right\}.
\end{split}
\end{split}
\end{align} Therefore, the 
GPL recursion \eqref{eq:GPL_rec} tells us that
the coefficient of the  $ O(\varepsilon^n)$ term in the Laurent expansion of $ \vartheta_{r\varepsilon,\varepsilon}$ resides in\begin{align}
\mathfrak Z_{n+2}(1)+\Span_{\mathbb Q}\left\{Z_{\ell}G(0,\beta_{1},\dots,\beta_{n+1-\ell};1)  \middle|\begin{smallmatrix}\ell \in \mathbb Z\cap[0,n+1],Z_\ell\in\mathfrak Z_\ell(1)\\\beta_{1},\dots,\beta_{n+1-\ell}\in\{0,1\}\end{smallmatrix} \right\}\xlongequal{\text{\eqref{eq:altMZV}}} \mathfrak Z_{n+2}(1),
\end{align}  which confirms \eqref{eq:epsLaurent}.

Finally, comparing the bivariate Taylor expansion of $ \vartheta_{\delta,\varepsilon}=\vartheta_{r\varepsilon,\varepsilon}$ with its univariate counterpart for sufficiently many cases of   $ r\in\mathbb Q\cap\big(\!-\!\infty,-\frac13\big)$, we see that  \eqref{eq:Jab_MZVd} is true. 
\qedhere\end{enumerate}

\end{proof}\begin{remark}One may streamline the main arguments behind \eqref{eq:Jab_MZVa}--\eqref{eq:Jab_MZVc} into the  following  hypergeometric identities:{\allowdisplaybreaks
\begin{align}
&_3F_2\left(\begin{array}{@{}c@{}}
1,1+\delta -\varepsilon,1+\delta +\varepsilon \\
\frac{3}{2}+\delta ,2+\delta  \\
\end{array}\middle|\frac{1}{4}\right)\tag{\ref{eq:Jab_MZVa}$'$}\\\begin{split}={}&\frac{(1+\delta ) (2\varepsilon -1)}{\varepsilon  (\delta +\varepsilon)}{_6}
F_5\left(\left.\begin{array}{@{}r@{\,}r@{,{}}r@{,{}}r@{,{}}r@{,{}}r@{}}
1,&\frac{5}{4}-\frac{\varepsilon }{2}&-\delta -\varepsilon &\frac{1}{2}+\frac{\delta-\varepsilon }{2}&1+\frac{\delta-\varepsilon }{2}&\frac{1}{2}-\frac{\varepsilon }{3}\\[4pt]
&\frac{1}{4}-\frac{\varepsilon }{2}&\delta +\frac{3}{2}&{}1-\frac{\delta+\varepsilon }{2}&{}\frac{1}{2}-\frac{\delta+\varepsilon }{2}& 1-\frac{2 \varepsilon }{3}\\
\end{array}\right| 1\right)\\&{}-\frac{1+\delta}{\mathrm B(1+\delta-\varepsilon,1+\delta+\varepsilon)}\frac{\pi}{\varepsilon\sin\frac{2\varepsilon\pi }{3}}\left[ 1+\frac{\sin \big(\big(\delta +\frac{\varepsilon }{3}\big)\pi\big)}{\sin\left((\delta+\varepsilon)\pi\right)} \right]
\end{split}\tag{\ref{eq:Jab_MZVb}$'$}
\\\begin{split}={}&-\frac{2 (1+\delta ) (1+2 \delta )}{\varepsilon  (\delta -\varepsilon )}{_4}F_3\left(\begin{array}{@{}c@{}}
1,-\delta ,-\delta -\frac{\varepsilon }{3},1+\delta +\varepsilon \\
-\delta -\varepsilon ,1+\frac{\varepsilon}{3},1-\delta +\varepsilon\  \\
\end{array}\middle|1\right)\\&{}+\frac{1+\delta}{\mathrm B(1+\delta-\varepsilon,1+\delta+\varepsilon)}\frac{\pi}{\varepsilon\sin\frac{\varepsilon\pi }{3}}\left[\frac{\sin(\delta\pi)}{\sin((\delta-\varepsilon)\pi)}-\frac{1+\frac{\sin\frac{(3\delta+\varepsilon)\pi}{3} }{\sin\left((\delta+\varepsilon)\pi\right)}}{2\cos\frac{\varepsilon\pi }{3}} \right],
\end{split}\tag{\ref{eq:Jab_MZVc}$'$}
\end{align}}where $ -\delta>\frac{\varepsilon}{3}>0$.
\eor\end{remark}\begin{remark}The two nearly-poised   $ _5F_4\left( \!\begin{smallmatrix}a_0,\!\!&a_1,\!\!&\dots\,,\!\!&a_4\\\!\!&b_1,\!\!&\dots\,,\!\!&b_4\end{smallmatrix}\!\middle|1 \right)$ series in \eqref{eq:f5F4_defn} are hypergeometric deformations around half-integer parameters, with the additional structure\footnote{This additional structure confines the hypergeometric deformations to the moduli space $\mathfrak M_{0,n} $ of rational curves, which is the natural abode of CMZVs \cite{Brown2009b}. Without the constraint in \eqref{eq:M0n}, hypergeometric deformations of   $ {_{q+1}}F_q\left( \!\begin{smallmatrix}a_0,\!\!&a_1,\!\!&\dots\,,\!\!&a_q\\\!\!&b_1,\!\!&\dots\,,\!\!&b_{\smash{q}}\end{smallmatrix}\!\middle|z \right)$  around half-integer parameters may reside in the moduli space  $\mathfrak M_{1,n} $  of elliptic curves \cite{SunZhou2026H2H3,Zhou2025NotesBHS}.  } \begin{align}
\lim_{\substack{\delta\to0\\\varepsilon\to0}}\frac{\Gamma(1+z)\prod_{m=1}^4\Gamma(b_m+z)}{\prod_{\ell=0}^4\Gamma(a_\ell+z)}\in\mathbb Q[z].\label{eq:M0n}
\end{align} Hence, the WZ\ pairs \eqref{eq:seedF} and \eqref{eq:seedG}  suggest that $ \mathscr J_{a,b}\in\mathfrak Z_{a+b+1}(2)$, a statement that is partly supported by some known results about hypergeometric deformations (see \cite[Theorem 1]{Au2025a} or \cite[Corollary 3.11]{Zhou2022mkMpl}).  Meanwhile, \eqref{eq:Re(gmn)(varrho)} allows us to deduce  $ \mathscr J_{a,b}\in\mathfrak Z_{a+b+1}(3)$ inductively.  Had  some folklores about the mixed Tate motives over $ \mathbb Z[1/2]$ and $ \mathbb Z[1/3]$ become verified truths, we would have been able to embed $ \mathfrak Z_w(2)\cap\mathfrak Z_w(3)$ into $ \mathfrak Z_w(1)$, as a shortcut to \eqref{eq:Jab_MZVd}. Our WZ-based approach to \eqref{eq:Jab_MZVc} and  unconditional proof of  \eqref{eq:Jab_MZVd}  did not draw on any conjectural embeddings.\footnote{Right after the proof of  \cite[Example III]{Au2025a}, Au commented that certain types of  hypergeometric deformations [subject to analogs of  the constraint 
\eqref{eq:M0n}] around half-integer parameters (which produce AMZVs at possibly different weights, according to  \cite[Theorem 1]{Au2025a}) would actually yield  $ \mathbb Q$-linear combinations of  MZVs at the same weight, ``but showing this requires a deeper development of above
method'', which he deferred to ``a later article'' (probably referring to \cite{Au2026WZ}). }
\eor\end{remark}
\subsection{Application to deformed Ap\'ery-like series}
In \cite{SunZhou2026MCV}, deformed Ap\'ery-like series referred to the following objects for $s\in\mathbb Z_{>0} $ and $n\in\mathbb Z_{\geq0} $:\begin{align}
\mathscr A_{s,n}\colonequals  \sum_{k=1}^\infty\left.\!\frac{\partial^n}{\partial x^n}\frac{[\Gamma(x)]^{2}}{2x^{s-1}\Gamma(2x)}\right|_{x=k}.
\end{align}
For all $ n\in\mathbb Z_{\geq0}$, it was known that 
\cite[Theorems 1.1 and 1.2]{SunZhou2026MCV}\footnote{Let $
\mathfrak C_w\colonequals \Span_{\mathbb Q}\left\{\Li_{a_1,\dots,a_n}(\varrho,\boldsymbol1_{n-1})\middle|a_1,\dots,a_n\in\mathbb Z_{>0},a_1+\dots+a_n=w\right\}
$ be the $ \mathbb Q$-vector space spanned by multiple Clausen values (MCVs) of weight $w$, which is a subspace of $ \mathfrak Z_w(3)$ (see \cite[Example 3.9]{Au2022a} or \cite[Proposition 3.3)]{SunZhou2026MCV}). According to 
\cite[\S\S2--3]{SunZhou2026MCV}), one may actually replace the right-hand sides of \eqref{eq:A1n} and \eqref{eq:Asn} by $ i\sqrt{3}\mathfrak C_{n+1}$ and $ \mathfrak C_{n+s}$, respectively.} \begin{align}
\mathscr A_{1,n}\in {}&i\sqrt{3}\mathfrak Z_{n+1}(3),\label{eq:A1n}\\\mathscr A_{s,n}\in{}& \mathfrak Z_{n+s}(3),\quad\text{where } s\in\mathbb Z_{>1},\label{eq:Asn}
\intertext{and that \cite[Theorem 1.3 and Remark thereafter]{SunZhou2026MCV} }
\mathscr A_{s,n}\in{}& \mathfrak Z_{n+s}(1),\quad\text{where } s\in\{2,4\}.\label{eq:descClausen}
\end{align}In the next corollary, we set the Galois descent \eqref{eq:descClausen} in a broader context.\begin{corollary}\label{cor:descClausen}
For  $ a,b\in\mathbb Z_{\geq0}$, we have\begin{align}
\sum_{n=1}^{\infty}\left.\!\frac{\partial^{a+b}}{\partial \varepsilon^{a}\partial x^{b}}\frac{\Gamma (x-\varepsilon) \Gamma (x+\varepsilon )}{x \Gamma (2 x)}\right|_{x=n,\varepsilon=0}\in{}&\mathfrak Z_{a+b+2}(1),\label{eq:A2n_eps}\\\sum_{n=1}^{\infty}\left.\!\frac{\partial^{a+b}}{\partial \varepsilon^{a}\partial x^{b}}\frac{\Gamma (x-\varepsilon) \Gamma (x+\varepsilon )}{x(x^{2}-\varepsilon^{2}) \Gamma (2 x)}\right|_{x=n,\varepsilon=0}\in{}&\mathfrak Z_{a+b+4}(1).\label{eq:A4n_eps}
\end{align} With $ a=0$, one recovers \eqref{eq:descClausen}.\end{corollary}\begin{proof}
 From \eqref{eq:Jab_MZVa} and \eqref{eq:Jab_MZVd}, we know that  \begin{align}
  \mathscr J_{a+1,b}=(a+1)\sum_{n=1}^{\infty}\left.\!\frac{\partial^{a+b}}{\partial \varepsilon^{a}\partial x^{b}}\frac{\Gamma (x-\varepsilon) \Gamma (x+\varepsilon )}{x \Gamma (2 x)}\right|_{x=n,\varepsilon=0}\in\mathfrak Z_{a+b+2}(1)
\end{align}for $ a,b\in\mathbb Z_{\geq0}$. This demonstrates \eqref{eq:A2n_eps}.

In the rest of this proof, we  will analyze certain infinite series related to   $ \mathscr J_{a+1,b}$. 

Thanks to Hou--Sun  (cf.\ \cite[Theorem 1.3]{HouSun2026}), we may pick a family of  WZ pairs\footnote{The WZ pairs listed here are just shifted versions of the proof for \cite[Theorem 1.3]{HouSun2026}. In Au's terminology \cite{Au2025a,Au2026WZ}, the WZ pairs constructed by Hou--Sun  \cite[Theorem 1.3]{HouSun2026} belongs to  the \texttt{Gauss2F1} class, whose behavior differs from that of  $(\mathfrak F_\varepsilon,\mathfrak G_\varepsilon)$ and $ (\varPhi_\varepsilon,\varTheta_\varepsilon)$  in the \texttt{Seed3} class. }\begin{align}
\mathscr F_{\varepsilon}( \nu,k)\colonequals{}&\frac{ \left( \nu-\varepsilon\right)! \left(k+ \nu+\varepsilon\right)!}{3 ( k+  \nu+\varepsilon+1) (k+2 \nu+1)!} ,\label{eq:HouSunF}\\\mathscr G_\varepsilon( \nu,k)\colonequals{}& \frac{\frac{\left( \nu-\varepsilon+1\right)^2}{k+2 \nu+2}+k+ \nu+\varepsilon+1}{ \nu-\varepsilon+1}\mathscr F_{\varepsilon}( \nu,k),\label{eq:HouSunG}
\end{align}satisfying\begin{align}
\mathscr F_{\varepsilon}( \nu+1,k)-\mathscr F_{\varepsilon}( \nu,k)=\mathscr G_\varepsilon( \nu,k+1)-\mathscr G_\varepsilon( \nu,k)\label{eq:DF=DG}
\end{align}and \begin{align}
\sum_{n=0}^{\infty}\mathscr G_\varepsilon(n+\delta,0)=\sum_{k=0}^\infty \mathscr F_\varepsilon(\delta,k),\label{eq:HouSunG=F}
\end{align}when $ \delta$ and $ \varepsilon$ are sufficiently close to zero. Unlike the situations encountered in the proof of Theorem \ref{thm:1.5b},  we are not  constrained by  $ \varepsilon>0$  while manipulating the family of WZ pairs \eqref{eq:HouSunF} and \eqref{eq:HouSunG}, nor do we  incur boundary contributions to \eqref{eq:HouSunG=F}. This is because \begin{align}
\begin{split}\left\vert\sum_{k=0}^\infty \mathscr F_\varepsilon(n+1+\delta,k)\right\vert={}& \left\vert\sum_{k=0}^\infty\frac{( n+1+\delta-\varepsilon )! (k+n+1+\delta+\varepsilon )!}{3 (k+n+2+\delta+\varepsilon) (k+2 ( n+1+\delta)+1)!}\right\vert\\\leq{}&\sum_{k=0}^\infty\frac{( n+2 )! \left(k+n+\frac{3}{2}\right )!}{3 (k+2n+1)!\left(k+n+\frac{3}{2}\right)}=\frac{2 \left(n+\frac{1}{2}\right)! (n+2)!}{3 (2 n-1) (2 n)!}\to0,
\end{split}
\end{align}as  $ n\to\infty$, while the relations{\allowdisplaybreaks\begin{align}
\mathscr G_{\varepsilon}(n+\delta,k+1)={}&\left[\frac{3}{2}+\frac{k+1+2\varepsilon }{n-\varepsilon +1}-\frac{k+1+2\varepsilon}{2 (k+2 n+3)}\right]\mathscr F_{\varepsilon}(n+\delta,k+1)\intertext{and}
\mathscr F_{\varepsilon}(n+\delta,k+1)={}&\frac{(n+\delta-\varepsilon)!}{3k^{n+2+\delta-\varepsilon}}\left[ 1+O\left( \frac{1}{k} \right) \right]\end{align} 
}bring us \begin{align}\lim_{k\to\infty}
\left\vert\mathscr G_\varepsilon(n+\delta,k+1)\right\vert=0
\end{align}for $n\in\mathbb Z_{\geq0} $. Now that \begin{align}
\mathscr G_\varepsilon(x-1,0)=\left[\frac{1}{2}+ \frac{\varepsilon }{3 (x-\varepsilon )}-\frac{\varepsilon }{3 (x+\varepsilon)} \right]\frac{\Gamma (x-\varepsilon) \Gamma (x+\varepsilon )}{x \Gamma (2 x)},
\end{align}the left-hand side of \begin{align}
\sum_{n=0}^\infty\left.\!\frac{\partial^{a+b}\mathscr G_\varepsilon(x,0)}{\partial \varepsilon^{a}\partial x^{b}}\right|_{x=n,\varepsilon=0}=\sum_{k=0}^\infty\left.\!\frac{\partial^{a+b}\mathscr F_\varepsilon(x,k)}{\partial \varepsilon^{a}\partial x^{b}}\right|_{x=0,\varepsilon=0}\label{eq:dG=dF}
\end{align}  is  equal to \begin{align}
\frac{\mathscr J_{a+1,b}}{2(a+1)}+\frac{2}{3}\sum_{n=1}^{\infty}\left.\!\frac{\partial^{a+b}}{\partial \varepsilon^{a}\partial x^{b}}\frac{\varepsilon^{2}\Gamma (x-\varepsilon) \Gamma (x+\varepsilon )}{x(x^{2}-\varepsilon^{2}) \Gamma (2 x)}\right|_{x=n,\varepsilon=0},\label{eq:dG_Ksum}
\end{align}where $ a,b\in\mathbb Z_{\geq0}$.
Meanwhile, to compute the right-hand side of \eqref{eq:dG=dF}, we will revisit some procedures related to \eqref{eq:Jab_MZVd}, as detailed in the next paragraph.

In hypergeometric terms, the generalized Hou--Sun identity \eqref{eq:HouSunG=F} reads\begin{align}\begin{split}&\frac{\left[3 (\delta +1)^2+\varepsilon ^2\right] \Gamma (1+\delta -\varepsilon ) \Gamma (1+\delta +\varepsilon )}{6 \varepsilon  \Gamma (3+2 \delta )}\left[ \frac{_4F_3\left( \!\begin{smallmatrix}1,1+\delta -\varepsilon,1+\delta -\varepsilon,1+\delta +\varepsilon\\\frac{3}{2}+\delta ,2+\delta ,2+\delta -\varepsilon \end{smallmatrix}\!\middle|\frac14 \right)}{1+\delta -\varepsilon } -\frac{_4F_3\left( \!\begin{smallmatrix}1,1+\delta +\varepsilon,1+\delta+\varepsilon,1+\delta -\varepsilon\\\frac{3}{2}+\delta ,2+\delta ,2+\delta +\varepsilon \end{smallmatrix}\!\middle|\frac14 \right)}{1+\delta +\varepsilon }\right]\\&{}+\frac{(3+2 \delta ) \Gamma (2+\delta -\varepsilon ) \Gamma (2+\delta +\varepsilon )}{2 \varepsilon  \Gamma (5+2 \delta )}\left[ \frac{_4F_3\left( \!\begin{smallmatrix}2,2+\delta -\varepsilon,2+\delta -\varepsilon,2+\delta +\varepsilon\\\frac{5}{2}+\delta ,3+\delta ,3+\delta -\varepsilon \end{smallmatrix}\!\middle|\frac14 \right)}{2+\delta -\varepsilon } -\frac{_4F_3\left( \!\begin{smallmatrix}2,2+\delta +\varepsilon,2+\delta+\varepsilon,2+\delta -\varepsilon\\\frac{5}{2}+\delta ,3+\delta ,3+\delta +\varepsilon \end{smallmatrix}\!\middle|\frac14 \right)}{2+\delta +\varepsilon }\right]\\&{}+\frac{\Gamma (3+\delta -\varepsilon ) \Gamma (3+\delta +\varepsilon )}{\varepsilon  \Gamma (7+2 \delta )}\left[ \frac{_4F_3\left( \!\begin{smallmatrix}3,3+\delta -\varepsilon,3+\delta -\varepsilon,3+\delta +\varepsilon\\\frac{7}{2}+\delta ,4+\delta ,4+\delta -\varepsilon \end{smallmatrix}\!\middle|\frac14 \right)}{3+\delta -\varepsilon } -\frac{_4F_3\left( \!\begin{smallmatrix}3,3+\delta +\varepsilon,3+\delta+\varepsilon,3+\delta -\varepsilon\\\frac{7}{2}+\delta ,4+\delta ,4+\delta +\varepsilon \end{smallmatrix}\!\middle|\frac14 \right)}{3+\delta +\varepsilon }\right]\\={}&\frac{\Gamma (1+\delta -\varepsilon ) \Gamma (1+\delta +\varepsilon )}{3 (1+\delta +\varepsilon )(1+2\delta) \Gamma (1+2 \delta )}{_3}F_2\left(\begin{array}{@{}c@{}}
1,1+\delta +\varepsilon ,1+\delta +\varepsilon  \\
2+2\delta,2+\delta+\varepsilon\ \\
\end{array} \middle|1\right),
\end{split}
\tag{\ref{eq:HouSunG=F}$'$}
\end{align}
where $ 1+\delta -\varepsilon>0$ is required for a convergent $ _3F_2$ series on the right-hand side [cf.\ \eqref{eq:abs_conv}]. Akin to \eqref{eq:logGamma_expn}, we have \begin{align}
\log\frac{\Gamma (1+\delta -\varepsilon ) \Gamma (1+\delta +\varepsilon )}{ \Gamma (1+2 \delta )}=\sum_{k=2}^\infty\frac{(-1)^{k}\zeta_{k}}{k}\big[(\delta-\varepsilon)^{k}+(\delta+\varepsilon)^{k}-(2\delta)^{k}\big],
\label{eq:Gamma(d,e)}\end{align}so  \begin{align}
\left.\!\frac{\partial^{a+b} }{\partial \varepsilon^{a}\partial\delta^{b}}\frac{\Gamma (1+\delta -\varepsilon ) \Gamma (1+\delta +\varepsilon )}{3 \Gamma (1+2 \delta )}\right|_{\delta=0,\varepsilon=0}\in\mathfrak Z_{a+b}(1)\cap\mathcal Z.
\end{align}In a similar vein as \eqref{eq:3F2deriv}, the relation $ (2+z)_k=\frac{k+1+z}{1+z}(1+z)_k$ leaves us\begin{align}
\begin{split}&\left.\!\frac{\partial^{a+b} }{\partial \varepsilon^{a}\partial\delta^{b}}\left[\frac{1}{ (1+\delta +\varepsilon )(1+2\delta)}{_3}F_2\left(\begin{array}{@{}c@{}}
1,1+\delta +\varepsilon ,1+\delta +\varepsilon  \\
2+2\delta,2+\delta+\varepsilon\ \\
\end{array} \middle|1\right)\right]\right|_{\delta =0,\varepsilon=0}\\\in{}&\mathbb Q\zeta_{a+b+2}+\Span_{\mathbb Q}\left\{ \sum_{k=1}^\infty\frac{1}{(k+1)^{1+s}}\prod_{j=1}^M\mathsf H_{k}^{(m_j)}\middle|\begin{smallmatrix}s,m_j,M\in\mathbb Z_{>0}\\1+s+\sum_{j=1}^M m_j=a+b+2\\\end{smallmatrix} \right\}\\\subseteq{}& \Span_{\mathbb Q}\left\{ \sum_{k=0}^\infty\frac{1}{(k+1)^{1+s}}\prod_{j=1}^M\mathsf H_{k+1}^{(m_j)}\middle|\begin{smallmatrix}s,m_j,M-1\in\mathbb Z_{>0}\\1+s+\sum_{j=1}^M m_j=a+b+2\\\end{smallmatrix} \right\}\subseteq\mathfrak Z_{a+b+2}(1).
\end{split}
\end{align}
Here, the first ``$ \subseteq$'' builds on the facts that $\mathsf H_{0}^{(m)}=0$ and $ \mathsf H^{(m)}_{k+1}=\mathsf H^{(m)}_{k}+\frac{1}{(k+1)^m}$. 
Since we already have an MZV characterization of  $ \mathscr J_{a+1,b}$, the remaining infinite series in  \eqref{eq:dG_Ksum}  should also belong to $ \mathfrak Z_{a+b+2}(1)$. This proves \eqref{eq:A4n_eps} by recursion. 
\end{proof}
\begin{remark}In \cite[Remark after Theorem 1.3]{SunZhou2026MCV}, the difficulty of the Galois descent  $ \mathscr A_{4,n}\in \mathfrak Z_{n+4}(1)$ was attributed to the lack of a purely WZ-based approach. This statement should be amended in view of   Theorem \ref{thm:1.5b} and Corollary \ref{cor:descClausen}. While neither    $ \mathscr A_{4,n}\in \mathfrak Z_{n+4}(1)$  nor its generalization  \eqref{eq:A4n_eps}  yielded to a single WZ pair, both of them are now direct consequences of three families of WZ pairs $(\mathfrak F_\varepsilon,\mathfrak G_\varepsilon)$, $ (\varPhi_\varepsilon,\varTheta_\varepsilon)$, and $ (\mathscr F_\varepsilon,\mathscr G_\varepsilon)$ simultaneously in action.     
\eor\end{remark}

\subsection{Application to Broadhurst's conjectures}Now we verify Broadhurst's conjecture on honorary MZVs at even weights \cite[Conjecture 5]{Broadhurst2014MDV}.\begin{corollary}
For all  $w\in\mathbb Z_{>1}$, we have \begin{align}
\frac{\pi\I\Li_{w-1}(\varrho)}{3}+\R\Li_{w,1}(\varrho)\in\mathfrak Z_{w}(1).\label{eq:gw2}
\end{align}In particular, when $w$ is even, we recover  Broadhurst's relation \eqref{eq:BroadhurstConj5}.\end{corollary}\begin{proof}According to   \eqref{eq:gwn_prod_expn}, the left-hand side of  \eqref{eq:gw2} is equal to $ -\R g_{w,2}(\varrho)$, which is representable as a $ \mathbb Q$-linear combination of   MZVs at weight $w\in\mathbb Z_{>1}$.

For $ s>0$, we may express \begin{align}
\I \Li_s(\varrho)=\left( 1+\frac{1}{2^{s-1}} \right)\frac{\sqrt{3}L(\chi_{-3},s)}{2}
\end{align} through a special Dirichlet $L$-function\begin{align}
L(\chi_{-3},s)\colonequals \sum_{n=0}^\infty\left[\frac{1}{(3n+1)^s}-\frac{1}{(3n+2)^s}\right].
\end{align}If $ w$ is an even  positive  integer, then  $ L(\chi_{-3},w-1)$ is a rational multiple of $\sqrt{3}\pi^{w-1}$, which simplifies   \eqref{eq:gw2} to   $\R\Li_{w,1}(\varrho)\in\mathfrak Z_{w}(1)$.
\end{proof}
\begin{remark}At this point, one may regard Au's  \eqref{eq:Au1.5b}  as an extension of Broadhurst's depth-2 relation \eqref{eq:BroadhurstConj5} to arbitrary depths. In the meantime, we note that  Xu's relation \eqref{eq:XuConj} does not    directly transcribe the generic behavior of  $ g_{w,n}(-1)$:  there are  weights $w$ and depths $n\geq1$,  such that Au's sums $ g_{w,n}(-1)$ do not degenerate to MZVs. Symbolic computations of $\R g_{w,n}(e^{2\pi i/N})$ and $\I g_{w,n}(e^{2\pi i/N})$ at small weights (with the aid of Au's \texttt{MultipleZetaValues} package \cite{Au2022a}) also suggest the absence of Galois descents when  $N\in\{3,5,7,8,9,10,12\}$.
\eor\end{remark}\begin{table}[t]\caption{Explicit forms of $ \frac{\pi\I\Li_{w-1}(\varrho)}{3}+\R\Li_{w,1}(\varrho)=-\R g_{w,2}(\varrho)$ at small weights\label{tab:gw2}}
{\scriptsize\begin{align*}\begin{array}{c|l}\hline\hline w&-\R g_{w,2}(\varrho)\\\hline2&\frac{\pi ^{2}}{2\cdot3^{2}}\\[2pt]\rowcolor{gray!30}3&\frac{2\zeta _{3}}{3}\\[2pt]4&\frac{59\pi ^{4}}{2^{4}\cdot3^{4}\cdot5}\\[2pt]\rowcolor{gray!30}5&\frac{43\zeta _{5}}{2^{2}\cdot3^{2}}-\frac{\pi ^{2} \zeta _{3}}{2^{2}\cdot3^{2}}\\[2pt]6&-\frac{\zeta _{3}^2}{2\cdot3}+\frac{2203\pi ^{6}}{2^{3}\cdot3^{8}\cdot5\cdot7}\\[2pt]\rowcolor{gray!30}7&\frac{4481\zeta _{7}}{2^{5}\cdot3^{4}}-\frac{\pi ^{2} \zeta _{5}}{2^{2}\cdot3^{2}}-\frac{7\cdot13\pi ^{4} \zeta _{3}}{2^{4}\cdot3^{5}\cdot5}\\[2pt]8&-\frac{7\zeta _{5,3}}{2\cdot3^{3}\cdot5}-\frac{5^{2} \zeta _{5}\zeta _{3} }{2\cdot3^{3}}+\frac{13\cdot43\cdot3709\pi ^{8}}{2^{8}\cdot3^{10}\cdot5^{3}\cdot7}\\[2pt]\rowcolor{gray!30}9&\frac{156967\zeta _{9}}{2^{5}\cdot3^{7}}-\frac{\pi ^{2} \zeta _{7}}{2^{2}\cdot3^{2}}-\frac{7\cdot13\pi ^{4} \zeta _{5}}{2^{4}\cdot3^{5}\cdot5}-\frac{11^{2}\cdot31\pi ^{6} \zeta _{3}}{2^{5}\cdot3^{8}\cdot5\cdot7}\\[2pt]10&-\frac{5\cdot41\zeta _{7,3}}{2^{5}\cdot3^{4}\cdot7}-\frac{7^{2}\cdot13\zeta _{7}\zeta _{3} }{2^{4}\cdot3^{4}}-\frac{5\cdot107\zeta _{5}^2}{2^{5}\cdot3^{2}\cdot7}+\frac{7003543\pi ^{10}}{2^{8}\cdot3^{12}\cdot5\cdot7^{2}\cdot11}\\[2pt]\rowcolor{gray!30}11&\frac{293\cdot10501\zeta _{11}}{2^{9}\cdot3^{7}}-\frac{\pi ^{2} \zeta _{9}}{2^{2}\cdot3^{2}}-\frac{7\cdot13\pi ^{4} \zeta _{7}}{2^{4}\cdot3^{5}\cdot5}-\frac{11^{2}\cdot31\pi ^{6} \zeta _{5}}{2^{5}\cdot3^{8}\cdot5\cdot7}-\frac{127\cdot1093\pi ^{8} \zeta _{3}}{2^{8}\cdot3^{10}\cdot5^{2}\cdot7}\\[2pt]12&-\frac{17\cdot59\zeta _{9,3}}{2^{4}\cdot3^{7}}+\frac{2\pi ^{2} \zeta _{7,3}}{3^{6}}-\frac{2^{2} \zeta _{6,4,1,1}}{3^{6}}-\frac{7\cdot17^{2}\cdot29\zeta _{9} \zeta _{3}}{2^{4}\cdot3^{8}}+\frac{\pi ^{4} \zeta _{5,3}}{3^{7}\cdot5}-\frac{37\cdot53\zeta _{7}\zeta _{5} }{2^{4}\cdot3^{5}}+\frac{2^{2}\cdot5\pi ^{2} \zeta _{7} \zeta _{3}}{3^{7}}+\frac{7\pi ^{2} \zeta _{5}^2}{3^{7}}-\frac{\pi ^{4} \zeta _{5} \zeta _{3}}{3^{8}\cdot5}+\frac{\zeta _{3}^4}{3^{7}}-\frac{\pi ^{6} \zeta _{3}^2}{3^{7}\cdot5\cdot7}+\frac{35677\cdot13941883\pi ^{12}}{2^{11}\cdot3^{17}\cdot5^{3}\cdot7^{2}\cdot11\cdot13}\\[2pt]\rowcolor{gray!30}13&\frac{1087\cdot180749\zeta _{13}}{2^{10}\cdot3^{10}}-\frac{\pi ^{2} \zeta _{11}}{2^{2}\cdot3^{2}}-\frac{7\cdot13\pi ^{4} \zeta _{9}}{2^{4}\cdot3^{5}\cdot5}-\frac{11^{2}\cdot31\pi ^{6} \zeta _{7}}{2^{5}\cdot3^{8}\cdot5\cdot7}-\frac{127\cdot1093\pi ^{8} \zeta _{5}}{2^{8}\cdot3^{10}\cdot5^{2}\cdot7}-\frac{13\cdot73\cdot757\pi ^{10} \zeta _{3}}{2^{9}\cdot3^{14}\cdot5\cdot11}\\[2pt]14&\frac{5\cdot7\cdot273029\zeta _{11,3}}{2^{9}\cdot3^{11}\cdot11}-\frac{5\cdot7\cdot109\cdot127\zeta _{9,5}}{2^{8}\cdot3^{12}}-\frac{11^{3}\cdot31\cdot61 \zeta _{11}\zeta _{3}}{2^{8}\cdot3^{9}}-\frac{5\cdot14649517 \zeta _{9}\zeta _{5}}{2^{8}\cdot3^{10}\cdot11}-\frac{7\cdot57190459\zeta _{7}^2}{2^{10}\cdot3^{11}\cdot11}+\frac{599\cdot2032343\pi ^{14}}{2^{4}\cdot3^{19}\cdot5^{2}\cdot7\cdot11^{2}\cdot13}\\[2pt]\rowcolor{gray!30}15&\frac{2129\cdot23003759\zeta _{15}}{2^{13}\cdot3^{13}}-\frac{\pi ^{2} \zeta _{13}}{2^{2}\cdot3^{2}}-\frac{7\cdot13\pi ^{4} \zeta _{11}}{2^{4}\cdot3^{5}\cdot5}-\frac{11^{2}\cdot31\pi ^{6} \zeta _{9}}{2^{5}\cdot3^{8}\cdot5\cdot7}-\frac{127\cdot1093\pi ^{8} \zeta _{7}}{2^{8}\cdot3^{10}\cdot5^{2}\cdot7}-\frac{13\cdot73\cdot757\pi ^{10} \zeta _{5}}{2^{9}\cdot3^{14}\cdot5\cdot11}-\frac{23^{2}\cdot89\cdot691\cdot3851\pi ^{12} \zeta _{3}}{2^{11}\cdot3^{17}\cdot5^{3}\cdot7^{2}\cdot11\cdot13}\\[3pt]\hline\hline\end{array}\end{align*}}
\end{table}

In \S\ref{subsec:shf_gwn}, we introduced ``depth control'' for the pre-descent input. One can perform a similar service on the post-descent output for $\R g_{w,2}(\varrho)$ as well (cf.\ Table \ref{tab:gw2} for some concrete examples). Before stating our results in the next two corollaries, we construct a $ \mathbb Q$-vector space [cf.\ \eqref{eq:QalgZ} for the definition of  $ \mathcal Z$]\begin{align}
\mathfrak z_{w,d}(N)\colonequals \Span_\mathbb Q\left\{ Z_m \Li_{a_1,\dots,a_n}(z_{1},\dots,z_n)\middle|\begin{smallmatrix}Z_m\in\mathfrak Z_m(1)\cap\mathcal Z\\m+1,a_1,\dots,a_n\in\mathbb Z_{>0}\\m+a_{1}+\dots+a_n=w,n\leq d \\z_1^N=\dots=z_n^N=1,(a_1,z_1)\neq(1,1)\\\end{smallmatrix}\right\}
\end{align}for each positive weight $ w$ and depth $d$,  so that $\mathfrak Z_w(1)\cap \mathcal Z= \mathfrak z_{w,1}(1)$ and $ \mathfrak z_{w,d_{1}}(N)\subseteq \mathfrak z_{w,d_{2}}(N)$ for $ d_1\leq d_2$.
\begin{corollary}If  $w-1\in2\mathbb Z_{>0}$, then we have\begin{align}
\R g_{w,2}(\varrho)\in \mathfrak z_{w,1}(1)=\mathfrak Z_w(1)\cap \mathcal Z.\label{eq:odd_dim_dp1MZV}
\end{align} 
\end{corollary}\begin{proof}For $ a,b\in\mathbb Z_{\geq0}$, define \begin{align}
 \mathscr K_{a,b}\colonequals\int_0^1\left( \frac{1}{t-\varrho} +\frac{1}{t-\frac{1}{\varrho}} \right)\log^{a} t \log^{b}\frac{(1-t)^{2}}{t}\D t.\label{eq:Kab_defn}
\end{align}These integrals enjoy the following properties:\begin{itemize}
\item When $ a,b\in\mathbb Z_{\geq0}$, we have  \begin{align}
\mathscr K_{a,0}\in{}& \mathfrak Z_{a+1}(1)\cap\mathcal{Z},\label{eq:Ka0}\\ \mathscr K_{0,b}\in {}&\mathfrak Z_{b+1}(1)\cap\mathcal{Z}.\label{eq:K0b}
\end{align}
\item When  $a\in2\mathbb Z_{>0}$ and $ b\in\mathbb Z_{\geq0}$,   we have     
\begin{align}
(b+1)\mathscr K_{a,b}-a \mathscr K_{a-1,b+1}\in \mathfrak Z_{a+b+1}(1)\cap\mathcal{Z}+\Span_{\mathbb Q}\left\{ \pi ^{2m}\mathscr K_{a',b'}\middle|\begin{smallmatrix}m,a',b',a-a'+1,b-b'+2\in\mathbb Z_{>0}\\ 2m+a'+b'=a+b\end{smallmatrix}\right\}.\label{eq:Kab_redn}
\end{align}
\end{itemize}  In the next two paragraphs that run parallel to the proof of Theorem \ref{thm:1.5a1.6}, we will demonstrate these properties, before applying them to \eqref{eq:odd_dim_dp1MZV} during the rest of this proof.

Integrating \cite[(1.1)]{Koelbig1986} by parts, we get\begin{align}
\int_0^1\frac{\log^at}{t-z}\D t=(-1)^{a+1}a!\Li_{a+1}\left( \frac{1}{z} \right),\quad z\in\mathbb C\smallsetminus[0,1), \end{align}so\begin{align}
 \mathscr K_{a,0}=2(-1)^{a+1}a!\R\Li_{a+1}(\varrho)=(-1)^{a+1}a!\left( 1-\frac{1}{2^{a}} \right)\left( 1-\frac{1}{3^{a}} \right){\zeta_{a+1}}
\end{align}for all $ a\in\mathbb Z_{>0}$, which proves \eqref{eq:Ka0}.
If $a=0$, then we are left with \begin{align}
\begin{split}F_{0,b}\colonequals {}&\frac{1}{2\pi i}\sum_{\ell=0}^{b+1}\binom{b+1}{\ell}[(\pi i)^\ell-(-\pi i)^\ell]\mathscr K_{0,b+1-\ell}\\={}& \frac{1}{2\pi i}\left(\int_{i0^{+}}^{i0^++1}-\int_{-i0^{+}}^{-i0^++1}\right)\left( \frac{1}{z-\varrho} +\frac{1}{z-\frac{1}{\varrho}} \right) \log^{b+1}\frac{(1-z)^{2}}{-z}\D z\\={}&-\frac{1}{2\pi i}\oint_{|z|=1}\left( \frac{1}{z-\varrho} +\frac{1}{z-\frac{1}{\varrho}} \right) \log^{b+1}\frac{(1-z)^{2}}{-z}\D z,
\end{split}
\end{align}where the last contour runs counterclockwise along the unit circle. Suppose that $ f(\varrho)=f(1/\varrho)=0$ and $ f(z)$ is well-behaved on the unit circle, then we have \begin{align}
\begin{split}\oint_{|z|=1}f(z)\frac{\D z}{z}={}&\frac{1}{2}\oint_{|z|=1}\left( \frac{1}{z-\varrho} +\frac{1}{z-\frac{1}{\varrho}} \right) \left[f(z)+f\bigg( \frac{1}{z} \bigg)\right]\D z\label{eq:f_refl}
\end{split}
\end{align}after pairing up the integrands at $ z$ and $1/z$. In particular, this implies that \begin{align}
\begin{split}F_{0,b}={}&-\frac{1}{2\pi i}\oint_{|z|=1}\frac{1}{z} \log^{b+1}\frac{(1-z)^{2}}{-z}\D z\\={}&-\frac{2^{b}}{\pi }\int_{0}^{2\pi}\log^{b+1}\left\vert 2\sin\frac{\theta}{2}\right\vert\D\theta=\frac{2^{b}}{\pi }\Ls_{b+2}(2\pi)\in\mathfrak Z_{b+1}(1)\cap\mathcal{Z},
\end{split}
\end{align}where the ``$\in$'' step comes from  Lewin's theory \cite[\S7.9.8]{Lewin1981}  for the log-sine integrals  (cf.\ \cite[(1) and (15)]{BorweinStraub2011ISSAC})\begin{align}
\Ls_n(2\pi)\colonequals -\int_0^{2\pi} \log^{n-1}\left\vert 2\sin\frac{\theta}{2}\right\vert\D\theta.
\end{align}In other words, we have \begin{align}
\mathscr K_{0,b}\in \mathfrak Z_{b+1}(1)\cap\mathcal{Z}+\Span_{\mathbb Z}\left\{ \pi^{2m}\mathscr K_{0,b-2m}\middle|m\in\mathbb Z\cap(0,b/2] \right\},
\end{align}which allows us to verify \eqref{eq:K0b} inductively.

When $a$ is even and positive, we consider \begin{align}
F_{a,b}\colonequals \frac{1}{2\pi i}\left(\int_{i0^{+}}^{i0^++1}-\int_{-i0^{+}}^{-i0^++1}\right)\left( \frac{1}{z-\varrho} +\frac{1}{z-\frac{1}{\varrho}} \right)\log^a(-z) \log^{b+1}\frac{(1-z)^{2}}{-z}\D z.\label{eq:Fab_defn}
\end{align}On one hand, according to our previous experience with $ F_{0,b}$ [cf.\  \eqref{eq:f_refl}], we have \begin{align}
\begin{split}F_{a,b}={}&-\frac{1}{2\pi i}\oint_{|z|=1}\frac{\log^a(-z)}{z} \log^{b+1}\frac{(1-z)^{2}}{-z}\D z\\={}&-\frac{i^{a}2^{b}}{\pi }\int_{0}^{2\pi}(\pi-\theta)^{a}\log^{b+1}\left\vert 2\sin\frac{\theta}{2}\right\vert\D\theta\in \mathfrak Z_{a+b+1}(1)\cap\mathcal{Z},
\end{split}
\end{align} where the ``$\in$'' step comes from known structure of the log-sine integrals (see \cite[(1) and (15)]{BorweinStraub2011ISSAC} as well as \cite[(7.140)]{Lewin1981})\begin{align}
\Ls_n^{(k)}(2\pi)\colonequals -\int_0^{2\pi} \theta^{k}\log^{n-1-k}\left\vert 2\sin\frac{\theta}{2}\right\vert\D\theta.
\end{align}On the other hand, we may evaluate $ F_{a,b}$ by spelling out the jump discontinuities across the branch cut in \eqref{eq:Fab_defn}. Therefore, \eqref{eq:Kab_redn} holds when $a$ is even and positive. Here, the first term on the right-hand side of \eqref{eq:Kab_redn}  absorbs the contributions from   \eqref{eq:Ka0}
 and \eqref{eq:K0b}.

For all  $w\in\mathbb Z_{>1}$, we have [cf.\ \eqref{eq:Re(gmn)(varrho)}] \begin{align}\begin{split}\R g_{w,2}(\varrho)={}&
\frac1{(w-2)!2!}\left.\!\R\frac{\partial^{w}}{\partial u^{w-2}\partial v^2}\left[ \frac{1 }{\mathrm B(1-u-v,v)} \int_0^1 \frac{t^{v} (1-t)^{-u-v}}{\varrho-t{}} \D t\right]\right|_{u=v=0}.
\end{split}
\end{align} 
As we expand the expression above with the product rule for differentiation, we encounter many factors [cf.\ \eqref{eq:Beta_deriv} and \eqref{eq:Ka0}] that live in the $ \mathbb Q$-algebra $ \mathcal Z$ generated from depth-1 MZVs. In particular, the Taylor expansion [cf.\ \eqref{eq:logGamma_expn}]\begin{align}
\frac{1 }{\mathrm B(1-u-v,v)}=v-\left[ \sum_{k=2}^N\zeta_k u^{k-1}+O\big(u^N\big) \right]v^{2}+O\big(v^3\big)
\end{align} and \eqref{eq:Ka0} tell us that \begin{align}
\R g_{w,2}(\varrho)-
\frac1{(w-2)!2!}\left.\!\R\frac{\partial^{w}}{\partial u^{w-2}\partial v^2}\left[ v \int_0^1 \frac{t^{v} (1-t)^{-u-v}}{\varrho-t{}} \D t\right]\right|_{u=v=0}\in\mathcal Z,\label{eq:gw2_baseline}
\end{align}so the only possible sources of deeper MZVs  are\begin{align}
\begin{split}&\int_0^1\left( \frac{1}{t-\varrho} +\frac{1}{t-\frac{1}{\varrho}} \right)\log^{a} (1-t) \log\frac{t}{1-t}\D t\\\xlongequal{t=1-s}{}&\int_0^1\left( \frac{1}{s-\varrho} +\frac{1}{s-\frac{1}{\varrho}} \right)\log^{a} s \log\frac{s}{1-s}\D s=\frac{\mathscr K_{a+1,0}-\mathscr K_{a,1}}{2},
\end{split}\label{eq:deep_source}
\end{align}where $ a=w-2$. 
Suppose that $a$ is odd. In view of \eqref{eq:Ka0} and \eqref{eq:Kab_redn},  we have \begin{align}
 \mathscr K_{a,1}\in \mathfrak Z_{a+2}(1)\cap\mathcal{Z}+\Span_{\mathbb Q}\left\{ \pi ^{2m}\mathscr K_{a',1}\middle|\begin{smallmatrix}m,a'\in\mathbb Z_{>0}\\ 2m+a'=a\end{smallmatrix}\right\},
\end{align}which allows us to establish \begin{align}
 \mathscr K_{a,1}\in \mathfrak Z_{a+2}(1)\cap\mathcal{Z}
\end{align}for all odd and positive $a$, by induction. Referring back to \eqref{eq:gw2_baseline}, we see that  $ \R g_{w,2}(\varrho)$ involves no sources of deeper MZVs when $w$ is an odd integer greater than one, as claimed in  \eqref{eq:odd_dim_dp1MZV}.
\end{proof}
Broadhurst's numerical experiments  suggested that one could embed     $ \R\Li_{w,1}(\varrho)$  into  $  \mathfrak z_{w,2}(2)$, for even and positive weights $w$   \cite[Conjecture 6]{Broadhurst2014MDV}.  In the next corollary, we establish such an embedding unconditionally.\begin{corollary}\label{cor:dep2AMZV}If  $w\in2\mathbb Z_{>0}$, then we have\begin{align}
\R g_{w,2}(\varrho)\in \mathfrak z_{w,2}(2).\label{eq:dp2AMZV}
\end{align}
\end{corollary}\begin{proof}According to our experience in the last corollary, it would suffice to show that \eqref{eq:deep_source} belongs to $\mathfrak  z_{w,2}(2)$ when $a=w-2$ is even. Since\begin{align}
\begin{split}&\int_0^1\left( \frac{1}{s-\varrho} +\frac{1}{s-\frac{1}{\varrho}} \right)\log^{w-2} s \log\frac{s}{1-s}\D s\\={}&2\int_0^1\left( \frac{1}{s-\varrho} +\frac{1}{s-\frac{1}{\varrho}} \right)\log^{w-1} s\D s -\int_0^1\left( \frac{1}{s-\varrho} +\frac{1}{s-\frac{1}{\varrho}} \right)\log^{w-2} s \log(s(1-s))\D s
\end{split}
\end{align} and the first summand on the right-hand side belongs to $  \mathfrak z_{w,1}(1)=  \mathfrak Z_{w}(1)\cap\mathcal{Z}$ [cf.\ \eqref{eq:Ka0}],  our task boils down to the demonstration of\begin{align}
\mathscr L_w\colonequals \int_0^1\left( \frac{1}{t-\varrho} +\frac{1}{t-\frac{1}{\varrho}} \right)\log^{w-2} t \log(t(1-t))\D t\in  \mathfrak z_{w,2}(2).
\end{align}As in the proof of Theorem \ref{thm:1.5b},
 we use the function $ \mathfrak f_{\delta,\varepsilon}$ [defined in \eqref{eq:f7F6_form}] to  subdivide the aforementioned task into the following components:
\begin{subequations}\begin{align}\mathscr L_w={}&(w-2)\left.\!\frac{\partial^{w-2}}{\partial\varepsilon^{w-3}\partial\delta}\frac{\Gamma (1+\delta) \Gamma (1+\delta+\varepsilon )}{\Gamma (3+2 \delta+\varepsilon)} \, _3F_2\left(\begin{array}{@{}c@{}}1,1+\delta,1+\delta+\varepsilon \\\frac{3}{2}+\delta+\frac{\varepsilon }{2},2+\delta+\frac{\varepsilon }{2}\end{array}\middle|\frac{1}{4}\right)\right|_{\delta=\varepsilon=0}\label{eq:Lw_AMZVa}\\={}&\frac{w-2}{2}\lim_{\varepsilon\to0^-}\left.\!\frac{\partial^{w-2}}{\partial\varepsilon^{w-3}\partial\delta}\left\{\mathfrak f_{\delta+\frac{\varepsilon}{2},-\frac{\varepsilon}{2}}
-\frac{2\pi}{\varepsilon\sin\frac{\varepsilon\pi}{3}}\left[ 1+\frac{\sin((\delta+\frac{\varepsilon}{3})\pi)}{\sin(\delta\pi)} \right]\right\}\right|_{\delta=0}\label{eq:Lw_AMZVb}\\\in{}&  \mathfrak z_{w,2}(2).\label{eq:Lw_AMZVc}\end{align}\end{subequations}
\begin{enumerate}[leftmargin=*,  label=(\alph*),ref=(\alph*),
widest=d, align=left] \item
Readapting the proof of \eqref{eq:Jab_MZVa}, we get\begin{align}
\begin{split}\mathscr L_w={}&- (w-2)\int_0^1 G(1,0;t(1-t))\frac{\log^{w-3} t}{t}\D t\\\xlongequal{\text{\cite[\S2.2]{SunZhou2026MCV}}}{}&(w-2)\int_{0}^1\sum_{k=1}^\infty\left.\!\frac{\partial}{\partial x}\frac{[t(1-t)]^{x}}{x}\right|_{x=k}\frac{\log^{w-3} t}{t}\D t\\={}&(w-2)\sum_{n=1}^\infty \left.\!\frac{\partial^{w-2}}{\partial\varepsilon^{w-3}\partial\delta}\frac{\mathrm B(n+\delta,n+\delta+\varepsilon)}{2n+2\delta+\varepsilon}\right|_{\delta=\varepsilon=0},
\end{split}
\end{align}which verifies \eqref{eq:Lw_AMZVa}.
\item One can check \begin{align}
\sum_{n=1}^\infty \frac{\mathrm B(n+\delta,n+\delta+\varepsilon)}{2n+2\delta+\varepsilon}=\frac{1}{2}\left\{\mathfrak f_{\delta+\frac{\varepsilon}{2},-\frac{\varepsilon}{2}}
-\frac{2\pi}{\varepsilon\sin\frac{\varepsilon\pi}{3}}\left[ 1+\frac{\sin((\delta+\frac{\varepsilon}{3})\pi)}{\sin(\delta\pi)} \right]\right\}
\end{align}by shifting the parameters in the proof of \eqref{eq:Jab_MZVb}.\item Specializing  Bailey's reduction formula \eqref{eq:Bailey7F6to4F3} to $ a_0=\frac{1}{2}+\frac{\varepsilon}{2}$, $ a_2=1$, $ a_3=\frac{1}{2}+\frac{\varepsilon}{6}$, $ a_4=\frac12+\frac{\delta+\varepsilon}{2}$, $ a_5=1+\frac{\delta+\varepsilon}{2}$, and $ a_6=-\delta$,  we get\begin{align}\begin{split}
\mathfrak f_{\delta+\frac{\varepsilon}{2},-\frac{\varepsilon}{2}}={}&\frac{4 \Gamma (1+\delta ) \Gamma (1+\delta +\varepsilon )}{(3+3 \delta +\varepsilon )(1+2\delta+\varepsilon) \Gamma (1+2 \delta +\varepsilon )}\, {_4}F_3\left( \begin{array}{@{}c@{}}
1,\frac{1+\delta }{2},\frac{1}{2}+\frac{\varepsilon}{6},\frac12+\frac{\delta+\varepsilon}{2} \\[2pt]
\frac{1-\delta}{2},\frac{3+\delta }{2}+\frac{\varepsilon }{6},\frac{3}{2}+\delta +\frac{\varepsilon }{2} \\
\end{array}\middle|1 \right)\\{}&+\frac{2\left( \delta+\frac{\varepsilon}{6} \right)\pi}{\delta\varepsilon\sin\frac{\varepsilon\pi}{6}}\frac{  \Gamma (1-\delta ) \Gamma \big(1+\frac{\delta }{2}\big) \Gamma \big(1+\delta +\frac{\varepsilon }{3}\big) \Gamma \big(1+\frac{\delta+\varepsilon }{2}\big)}{  \Gamma \big(1-\delta -\frac{\varepsilon }{6}\big) \Gamma \big(1+\frac{\delta }{2}+\frac{\varepsilon }{6}\big) \Gamma \big(1+\frac{\delta }{2}+\frac{\varepsilon }{3}\big) \Gamma \big(1+\delta +\frac{\varepsilon }{2}\big)}\\&{}\times{_3}F_2\left(\begin{array}{@{}c@{}}
-\frac{\delta}{2},-\frac{\varepsilon}{6} ,\frac{\varepsilon}{3}\\[2pt]
-\delta-\frac{\varepsilon}{6} ,1+\frac{\delta }{2}+\frac{\varepsilon }{3} \\
\end{array}\middle| 1 \right).
\end{split}\label{eq:BB4F3}
\end{align}In the next two paragraphs, we will handle the two summands on the right-hand side of \eqref{eq:BB4F3} separately.

First, we  prove that\begin{align}
\begin{split}\mathscr M_w\colonequals{}& \lim_{\varepsilon\to0^-}\left.\!\frac{\partial^{w-2}}{\partial\varepsilon^{w-3}\partial\delta}\left[\frac{1}{(3+3 \delta +\varepsilon )(1+2\delta+\varepsilon)}{_4}F_3\left( \begin{array}{@{}c@{}}
1,\frac{1+\delta }{2},\frac{1}{2}+\frac{\varepsilon}{6},\frac12+\frac{\delta+\varepsilon}{2} \\[2pt]
\frac{1-\delta}{2},\frac{3+\delta }{2}+\frac{\varepsilon }{6},\frac{3}{2}+\delta +\frac{\varepsilon }{2} \\
\end{array}\middle|1 \right)\right]\right|_{\delta=0}\\\in{}&   \mathfrak z_{w,2}(2).
\end{split}\end{align} By the definition of generalized hypergeometric series in \eqref{eq:pFq_defn} along with $ (A)_0\colonequals 1$, we have\begin{align}
{_4}F_3\left( \begin{array}{@{}c@{}}
1,\frac{1+\delta }{2},\frac{1}{2}+\frac{\varepsilon}{6},\frac12+\frac{\delta+\varepsilon}{2} \\[2pt]
\frac{1-\delta}{2},\frac{3+\delta }{2}+\frac{\varepsilon }{6},\frac{3}{2}+\delta +\frac{\varepsilon }{2} \\
\end{array}\middle|1 \right)=\sum_{k=0}^\infty\frac{\big(\frac{1+\delta }{2}\big)_k}{\big(\frac{1-\delta }{2}\big)_k}\frac{\big(\frac12+\frac{\varepsilon}{6}\big)_k}{\big(\frac{3+\delta}2+\frac{\varepsilon}{6}\big)_k}\frac{\big(\frac12+\frac{\delta+\varepsilon}{2}\big)_k}{\big(\frac32+\delta+\frac{\varepsilon}{2}\big)_k},
\end{align}where the summands satisfy \begin{align}
\begin{split}&\frac{1}{(3+3 \delta +\varepsilon )(1+2\delta+\varepsilon)}\frac{\big(\frac{1+\delta }{2}\big)_k}{\big(\frac{1-\delta }{2}\big)_k}\frac{\big(\frac12+\frac{\varepsilon}{6}\big)_k}{\big(\frac{3+\delta}2+\frac{\varepsilon}{6}\big)_k}\frac{\big(\frac12+\frac{\delta+\varepsilon}{2}\big)_k}{\big(\frac32+\delta+\frac{\varepsilon}{2}\big)_k}\\={}&\frac{1}{(6k+3+3\delta+\varepsilon)(2 k+1+2 \delta +\varepsilon )}\frac{\big(\frac{1+\delta }{2}\big)_k}{\big(\frac{1-\delta }{2}\big)_k}\frac{\big(\frac12+\frac{\varepsilon}{6}\big)_k}{\big(\frac{1+\delta}2+\frac{\varepsilon}{6}\big)_k}\frac{\big(\frac12+\frac{\delta+\varepsilon}{2}\big)_k}{\big(\frac12+\delta+\frac{\varepsilon}{2}\big)_k}.
\end{split}
\end{align}Direct computation reveals that \begin{align}
\left.\frac{\partial}{\partial\delta}\frac{\big(\frac{1+\delta }{2}\big)_k}{\big(\frac{1-\delta }{2}\big)_k}\right|_{\delta=0}=\psi^{(0)}\left( k+\frac{1}{2} \right)-\psi^{(0)}\left( \frac{1}{2} \right),
\end{align}where $ \psi^{(0)}(z)=\frac{\D}{\D z}\log\Gamma(z)$ is Euler's digamma function. More generally, for constants $ \alpha$ and $\beta$, we have\begin{align}
\left.\frac{\partial}{\partial\delta}\frac{(1+z+\alpha\delta)_k}{(1+z+\beta\delta)_k}\right|_{\delta=0}=(\alpha-\beta)\left[ \psi^{(0)}(1+k+z)- \psi^{(0)}(1+z)\right].\label{eq:DlogPoch}
\end{align}Thus, for $ w\in\mathbb Z_{>1}$, one has\begin{align}
\mathscr M_w\in\mathbb Q\zeta_{w}+\Span_{\mathbb Q}\left\{ \left.\sum_{k=0}^\infty\frac{\psi^{(w-3-s)}\left( k+\frac{1}{2} \right)-\psi^{(w-3-s)}\left( \frac{1}{2} \right)}{(2k+1)^{2+s}}\right|s\in\mathbb Z\cap[0,w-3]\right\},
\end{align}where $ \psi^{(m)}(z)=\frac{\D^{m+1}}{\D z^{m+1}}\log\Gamma(z)$ is Euler's polygamma function. To represent the remaining infinite sum as an integral, we follow the Borwein--Zucker--Boersma construction of \cite[Table 2]{BZB2008}. Concretely speaking, from the generating function\begin{align}
\sum_{k=0}^\infty\left[ \psi^{(m)}\left( k+\frac{1}{2} \right)-\psi^{(m)}\left( \frac{1}{2} \right) \right] x^{2k}=\frac{(-2)^{m}m!x}{1-x^{2}}\left[\Li_{m+1}(x)-\Li_{m+1}(-x)\right]
\end{align}for  $|x|<1 $ and the integral formula\begin{align}
\int_0^1 x^{2k}\log^{1+s} x\D x=\frac{(-1)^{1+s}(1+s)!}{(2k+1)^{2+s}},
\end{align}we may embed $ \mathscr M_w$ into  $\mathbb  Q\zeta_{w}$ plus a $ \mathbb Q$-vector space spanned by the following types of integrals:\begin{align}
\int_0^1 \left( \frac{1}{x-1} +\frac{1}{x+1}\right)\left[\Li_{w-2-s}(x)-\Li_{w-2-s}(-x)\right]\log^{1+s} x\D x\in   \mathfrak z_{w,2}(2).
\end{align}Here, to justify the ``$ \in$'' step, simply employ the GPL recursion in \eqref{eq:GPL_rec} with logarithmic regularization \cite[\S2.3]{Panzer2015} at $ x=1$.

Then, we show that\begin{align}\mathscr N_w\colonequals \lim_{\varepsilon\to0^-}\left.\!\frac{\partial^{w-2}}{\partial\varepsilon^{w-3}\partial\delta}\left\{\frac{2\left( \delta+\frac{\varepsilon}{6} \right)\pi}{\delta\varepsilon\sin\frac{\varepsilon\pi}{6}}\left[  {_3}F_2\left(\begin{array}{@{}c@{}}
-\frac{\delta}{2},-\frac{\varepsilon}{6} ,\frac{\varepsilon}{3}\\[2pt]
-\delta-\frac{\varepsilon}{6} ,1+\frac{\delta }{2}+\frac{\varepsilon }{3} \\
\end{array}\middle| 1 \right) -1\right]\right\}\right|_{\delta=0}\in   \mathfrak z_{w,2}(1).
\end{align}Thanks to  the relation $ (z)_k=\frac{z}{k+z}(1+z)_k$, one has\begin{align}
\begin{split}&\frac{2\left( \delta+\frac{\varepsilon}{6} \right)\pi}{\delta\varepsilon\sin\frac{\varepsilon\pi}{6}}\left[  {_3}F_2\left(\begin{array}{@{}c@{}}
-\frac{\delta}{2},-\frac{\varepsilon}{6} ,\frac{\varepsilon}{3}\\[2pt]
-\delta-\frac{\varepsilon}{6} ,1+\frac{\delta }{2}+\frac{\varepsilon }{3} \\
\end{array}\middle| 1 \right) -1\right]\\={}& -\frac{\pi  \varepsilon  }{3\sin\frac{\varepsilon\pi}{6} }\sum_{k=1}^{\infty}\frac{ 6 k-6 \delta-\varepsilon }{(2 k-\delta ) (6 k-\varepsilon ) (3 k+\varepsilon )}\frac{\big(1-\frac{\delta}{2}\big)_k}{(1)_{k}}\frac{\big(1-\frac{\varepsilon}{6}\big)_k}{\big(1-\delta-\frac{\varepsilon}{6}\big)_k}\frac{\big(1+\frac{\varepsilon}{3}\big)_k}{\big(1+\frac{\delta}{2}+\frac{\varepsilon}{3}\big)_k}.
\end{split}
\end{align}The derivative formula \eqref{eq:DlogPoch} then entails \begin{align}
\mathscr N_w\in\mathbb Q\zeta_w+\Span_{\mathbb Q}\left\{ \left.\sum_{k=1}^\infty\frac{\psi^{(w-3-s)}(k+1)-\psi^{(w-3-s)}(1)}{k^{2+s}}\right|s\in\mathbb Z\cap[0,w-3]\right\}.
\end{align}Again, we fall back on the Borwein--Zucker--Boersma method behind \cite[Table 2]{BZB2008}. Via the generating function\begin{align}
\sum_{k=0}^\infty\left[ \psi^{(m)}(k+1)-\psi^{(m)}(1) \right] x^{k}=\frac{(-1)^mm!}{1-x}\Li_{m+1}(x)
\end{align}for  $|x|<1 $ and the integral formula\begin{align}
\int_0^1 x^{k-1}\log^{1+s} x\D x=\frac{(-1)^{1+s}(1+s)!}{k^{2+s}},
\end{align} we may represent $ \mathscr N_w$ as  a rational multiple of  $
\zeta_{w}$  plus a $  \mathbb Q$-linear combination of integrals in the following form:\begin{align}
\int_0^1 \left( \frac{1}{x-1} -\frac{1}{x}\right)\Li_{w-2-s}(x)\log^{1+s} x\D x\in   \mathfrak z_{w,2}(1)\subseteq\mathfrak z_{w,2}(2).
\end{align}As before, the embedding above is made possible by  the GPL recursion in \eqref{eq:GPL_rec} with logarithmic regularization \cite[\S2.3]{Panzer2015} at $ x=1$.

On the right-hand side of \eqref{eq:BB4F3}, the factors\begin{align}
\frac{4 \Gamma (1+\delta ) \Gamma (1+\delta +\varepsilon )}{\Gamma (1+2 \delta +\varepsilon )}\end{align}and \begin{align}\frac{ 
\Gamma (1-\delta ) \Gamma \big(1+\frac{\delta }{2}\big) \Gamma \big(1+\delta +\frac{\varepsilon }{3}\big) \Gamma \big(1+\frac{\delta+\varepsilon }{2}\big)}{  \Gamma \big(1-\delta -\frac{\varepsilon }{6}\big) \Gamma \big(1+\frac{\delta }{2}+\frac{\varepsilon }{6}\big) \Gamma \big(1+\frac{\delta }{2}+\frac{\varepsilon }{3}\big) \Gamma \big(1+\delta +\frac{\varepsilon }{2}\big)}
\end{align} produce only depth-1 MZVs during Taylor expansions, as elucidated by the mechanism in \eqref{eq:Gamma(d,e)}. Therefore, \eqref{eq:Lw_AMZVc} is true. 
\qedhere\end{enumerate}

\end{proof}

\subsection*{Acknowledgments}The author thanks Kam Cheong Au and David Broadhurst for their thoughtful feedback on the initial draft of this article.

\end{document}